\documentclass[10pt,oneside,a4paper]{article}

\usepackage{amscd,amssymb,amsmath,amsthm,mathrsfs,dsfont}
\usepackage[a4paper]{geometry}
\usepackage{color}
\usepackage{tikz} 
\usetikzlibrary{positioning}
\usetikzlibrary{backgrounds}
\usetikzlibrary{patterns}
\usepackage[shortlabels]{enumitem}
\usepackage{comment}
\usepackage{caption}
\usepackage{array}
\usepackage{dirtytalk}
 \usepackage[english]{babel}
\usepackage[normalem]{ulem}

\usepackage{subcaption}
\usepackage[thinc]{esdiff}
 
\usepackage{hyperref}
\hypersetup{
    colorlinks=true,
    linkcolor=blue,
    filecolor=black,      
    urlcolor=blue,
    citecolor=blue,
}
\newtheorem{theorem}{Theorem}

\usepackage[hang,flushmargin]{footmisc} 
\newtheorem{proposition}{Proposition}
\newtheorem{lemma}{Lemma}
\newtheorem{remark}{Remark}
\newtheorem{definition}{Definition}
\newtheorem{corollary}{Corollary}
\def\N{\mathbb{N}}
\def\Z{\mathbb{Z}}
\def\R{\mathbb{R}}

\numberwithin{equation}{section}
\usepackage{fancyhdr}

\newcommand\shorttitle{Extremal decomposition of A-localized states}

\addto\captionsenglish{}
\usepackage[labelfont=bf]{caption}
\begin{document}

\hypersetup{linkcolor=black}
\hypersetup{urlcolor=black}
\title{$A$-localized states for clock models on trees and their extremal decomposition into glassy states}
\author{Christof Külske\footnotemark[1] \and Niklas Schubert\footnotemark[2]}
\date{\today}

\maketitle
\begin{abstract}
    We consider $\Z_q$-valued clock models on a regular tree, for general classes of ferromagnetic nearest neighbor interactions which have a discrete rotational symmetry.
It has been proved recently that, at strong enough coupling, families of homogeneous Markov chain 
Gibbs states $\mu_A$ coexist whose single-site marginals 
concentrate on $A\subset \Z_q$, and which are not convex combinations of each other \cite{AbHeKuMa24}.  
In this note, we aim at a description of  
the extremal decomposition of $\mu_A$ for $|A|\geq 2$ into all extremal Gibbs measures, which may be spatially 
inhomogeneous.  
First, we show that in regimes of very strong coupling, $\mu_A$ is not extremal. Moreover, 
$\mu_A$ possesses a single-site reconstruction property which holds for spin values sent from the origin 
to infinity, when these initial values are chosen from $A$. 
As our main result, we show that $\mu_A$ decomposes into uncountably many extremal inhomogeneous states. The proof is based on multi-site reconstruction which allows to derive concentration properties of branch overlaps. Our method is based 
on a new good site/bad site decomposition adapted to the $A$-localization property, together with a coarse graining argument in local state space. 
\end{abstract}
\vspace{0.3cm}

\footnotetext[1]{Faculty of Mathematics, Ruhr-University Bochum, 44801 Bochum, Germany}
\footnotetext[0]{E-Mail: \href{mailto:Christof.Kuelske@ruhr-uni-bochum.de}{Christof.Kuelske@ruhr-uni-bochum.de}; ORCID iD: \href{https://orcid.org/0000-0001-9975-8329}{0000-0001-9975-8329}}
\footnotetext[0]{\url{https://math.ruhr-uni-bochum.de/fakultaet/arbeitsbereiche/stochastik/gruppe-kuelske/mitarbeiter/christof-kuelske/}}
\footnotetext[2]{E-Mail: \href{mailto:Niklas.Schubert@ruhr-uni-bochum.de}{Niklas.Schubert@ruhr-uni-bochum.de}; ORCID iD: \href{https://orcid.org/0009-0000-8912-4701}{0009-0000-8912-4701}}
\footnotetext[0]{\url{https://math.ruhr-uni-bochum.de/fakultaet/arbeitsbereiche/stochastik/gruppe-kuelske/mitarbeiter/niklas-schubert/}}
\tableofcontents
\hypersetup{linkcolor=blue}

\section{Introduction}

{\bf Reconstruction, non-extremality, extremal decomposition for free states on trees.} Gibbs measures on infinite trees show amazing properties which are not present on lattices. A first example for this is the free state of the Ising model in zero external field, which can be obtained as the infinite-volume limit of the finite-volume Gibbs distributions with free (open) boundary conditions. It displays two transition temperatures, 
one for the transition between uniqueness and non-uniqueness of the Gibbs measure, and another one for the transition between extremality and non-extremality of the free state.
Proofs of non-extremality on trees can be given in terms of \textit{reconstruction properties}, while single-site reconstruction does not yield information about the nature of the extremal decomposition. 
Here, (single-site) reconstruction for a measure $\mu$ is said to hold, if a spin-value $a$ which is sent from a vertex of the tree to the outside with $\mu$, 
 can be reconstructed from the observation at infinity (with a better probability than just guessing according  to the single-site marginal of the measure $\mu$ itself). 
 This relation is the reason that the topics of statistical physics, probability and information science come together fruitfully in this 
 type of question, see for example \cite{EvKePeSc00}, \cite{KrMeSaSuZd12} and \cite{Mo01}.

Now, given that we are in a parameter regime for non-extremality of a certain 
measure $\mu$, which could e.g. be the free state of an Ising or Potts model, 
it is challenging to go further and to analyze the nature  of the corresponding \textit{extremal decomposition measure} which is associated to $\mu$. 
This extremal decomposition measure describes and quantifies 
how the state $\mu$ breaks into extremal Gibbs states, see formula \eqref{eq: Extremal decomposition GM} below. It is a dirac measure in the case of extremality of the state $\mu$, 
it could be supported on a finite or countable number of atoms (this is what 
happens for the free state of the Ising model on lattices), or more interestingly 
it can be a continuous 
measure with support spread over uncountably many states.

In contrast to the 
behavior on lattices, measures on trees behave in a different 
and very rich way which is not straightforward to analyze. 
Indeed, it took the community 
quite some time to prove the atomless property of the free state of the Ising model, see \cite{GaMaRuSh20}. 
For recent more general results about the extremal decomposition of free states into \textit{glassy states}, 
in the cases of Potts models and clock models, see \cite{CoKuLe24}. This work used an approach which was inspired by \cite{GaMaRuSh20}, however it was significantly different as it involves \textit{branch overlaps}, 
whose use is motivated from \textit{spin-glass theory}. 

Spin glasses started to become famous in the 1970s and are examples of magnetic spin systems with \textit{quenched (or frozen) disorder}, in which 
ferromagnetic and antiferromagnetic couplings are competing. 
This competition may result in many non-translation invariant 
groundstates and almost groundstates, which then give rise to glassy states at low enough temperature. 
For the history of the famous Parisi solution of the 
mean-field Sherrington-Kirkpatrick spin-glass, in 
which spin overlaps play a central role, and 
the mathematics it initiated, see \cite{ArCh13}, \cite{Gu01}, \cite{Gu03}, \cite{MePaMi87}, \cite{Pa80}, \cite{Pa92}, \cite{Tal06}, \cite{Ta11} and \cite{Ta112}. 
Lattice spin glasses are even more difficult to come by. An important approach 
to describe their infinite-volume asymptotics in situations of
random symmetry breaking with competition of many states, is via \textit{metastates}.  Metastates are measures 
on infinite-volume Gibbs states of a disordered system, 
in the same way as the extremal decomposition measure 
is a measure on the available infinite volume states, 
but they come with the different meaning to capture a possible chaotic volume dependence.  
If one is interested in the development of the 
metastate concept to describe chaotic volume 
dependence of disordered systems, one can consult the book \cite{StNe13}. For more generalities and details about disordered systems and spin glasses, see \cite{ArDaMiNeSt10}, \cite{Bo06}, \cite{DoMo24}, \cite{StNe13}. 

Let us now describe the type of clock (discrete rotator) models, type of Gibbs measures and 
the new glassy states which we will find in 
their extremal decomposition. \newline

\noindent {\bf $A$-localized states for clock-models.} Clock models are models of discrete rotators where the spin variables take values in the 
local state space $\Z_q=\{0,1,\dots, q-1\}$, under the influence of  
a nearest neighbor interaction which is invariant under joint 
discrete rotation in $\Z_q$ of all spins, see \eqref{eq: Potential ferr n.n. model} below. 

For work on clock models see, \cite{FrSp81}, \cite{MaSh11}, \cite{Ueno95} from a mathematical perspective, and \cite{BrReJoPl10}, \cite{GoKuGoSa24}, \cite{OrCoNu12}, \cite{ScIr93}, \cite{ToUeYoMi02}, \cite{ToOk02}, \cite{UeKa93} from a physical perspective. 
Let us notice that one particular 
model, which is always included as a first example, is the Potts model. The Potts model provides 
a benchmark case for which our theorems are more easily understood (and proved), but our 
theory is much more general.

In a recent paper \textit{$A$-localized states $\mu_A$} on Cayley trees have been constructed for 
$\Z_q$-valued clock models \cite{AbHeKuMa24}, for sufficiently strong coupling (large inverse temperature).  
While the states could be described in terms of estimates for instance 
on the level of single-site marginals, nothing was  
said about the issue of extremality and the extremal 
decomposition measure, and this is what we going to do 
in this work.

What was then known about these states $\mu_A$? 
Throughout the paper $A\subset \Z_q$ is an arbitrary subset. 
We say that $A$-localization holds, if the single-site 
marginals of the states $\mu_A$ concentrate on the set $A$. Moreover, $\mu_A$ even 
resembles the equidistribution on the set $A$, up to small controllable errors (cp. formula $(iv)$ in \eqref{eq: bounds A-localized states} below). 
The existence of such states was previously only known in the special case of the Potts model, where 
it could be discovered and proved to hold via explicit computations \cite{KuRo17}. 
The construction in the general clock-model case by the authors of \cite{AbHeKuMa24} is necessarily less explicit, but obtained
via a combination of fixed point methods in the framework of the boundary law formalism, see \cite{Ge11} or \cite{Za83} for more information about boundary laws. 

Physically speaking, the clock-model states $\mu_A$ can be viewed as (in general) non-symmetric excitations and perturbations 
of free states concentrating on the smaller effective local state space $A$. Do these states have a symmetry in spin-space? 
Let us start with a set $A$ which sits in $\Z_q$ in a symmetric way, as it is always the 
case for $|A|=2$. Then the state $\mu_A$ will possess a symmetry in spin-space. For general sets $A$ however, 
the corresponding $\mu_A$ will typically be non-symmetric
(except in the very special case of the Potts model).
The authors of \cite{AbHeKuMa24} have proved that these states can not be written as combinations of each other, 
and in particular $\mu_A\neq \frac{1}{|A|} \sum_{a\in A}\mu_{\{a\}}$. It is already more difficult to decide whether these  
states are nevertheless non-extremal and it is 
much more subtle to 
prove results about the character of the extremal decomposition into all (possibly non-homogeneous) Gibbs measures. 
In this work, we are solving this issue in the low temperature region, and we have to do it using a method which does not use any symmetries in spin space.  \medskip

\noindent {\bf Main results. Atomless property of the decomposition measure for $A$-localized states when $|A|\geq 2$.} Our main results of the present paper (summarized in Theorem \ref{thm: Main results}) are informally the following:  \\ 
At sufficiently low temperature the clock-model states $\mu_{A}$ have

\noindent $\bullet$ a \textit{reconstruction} property, but a \textit{restricted} one, 
i.e. (only) spin-values $a$ from the localization set $A$  
can be reconstructed from their noisy images at infinity, see Theorem \ref{thm: mu_A reconstruction theorem}. \\
$\bullet$ an \textit{atomless extremal decomposition measure}, where the decomposition measure is 
dispersed among uncountably many spatially inhomogeneous glassy states. The extremal 
states in the support of the decomposition measure, concentrate on  
low temperature perturbations around mostly flat "almost groundstates" with spin values in the concentration set $A$, see Theorem \ref{thm: A.s. singularity extremals}.

Our present result holds for rather general clock models without any assumption on FKG properties. 
Note that it is also a new result for the $A$-localized states of the 
Potts model \cite{KuRo17}. While the latter paper 
was able to give partial statements about extremality 
and non-extremality, 
the character of the extremal decomposition was not analyzed. 

What assumptions on the potential do we need? While our present proof is very general, we nevertheless need certain degree $d$-dependent assumptions on 
minimal and maximal costs of changing spins along an edge (see the $u,U,d$-assumption \eqref{eq: u,U-bounds} below). 
This condition holds true for the Potts model and many more models. It holds true
for any fixed ferromagnetic interaction 
at sufficiently large degrees $d$. 

Note that the atomless property can never be expected to hold for $|A|=1$, as already the corresponding 
result for the Potts model based on FKG says that the corresponding states are always extremal, see \cite{Ga92}.  

Let us now come to proof ideas, their relation 
to previous work on free states, and challenges and 
novelties of the present case of $A$-localized measures.

\medskip 

\noindent {\bf Proof ideas: Multi-site reconstruction along branches via $\pi$-kernel,  $A$-irregular edges, $A-A^c$ coarse-graining.} We will start to work within the approach of \cite{CoKuLe24} developed for free states, 
which is based on properties of branch overlaps. 
It is clear though that there will need to be essential modifications. 
In the spirit of disordered systems one derives 
concentration properties of suitable branch overlaps, by proving a multi-site reconstruction for which 
Peierls bounds with errors for suitably defined contours can be used. This already allowed to treat also certain 
classes of possibly nonsymmetric 
perturbations of the free states (called central states).
Unfortunately there are severe problems to apply the method to our case directly. 
This method in its original form would fail in application to $\mu_A$, as it is crucially based on a quantitative \textit{lazyness} assumption,  which $\mu_A$ for $A\neq \Z_q$ 
can never satisfy. (Lazyness refers here to the property of 
the associated Markov chain transition matrix $P_A$ to keep any given spin value along an edge with sufficiently high probability.) 
Serious modifications are then needed, 
and these are adapted to concentration properties of the states $\mu_A$.  
To treat them properly, we devise in this paper 
the novel notion of \textit{$A$-irregular edges.} Then we use a \textit{coarse-graining} in local state space to formulate and prove the key estimate for our $\mu_A$ which is 
Lemma \ref{lem: Key Lemma}.

Let us now give some more ideas. To begin with, it is important to recall from abstract Gibbs theory 
that the extremal decomposition measure $\alpha_{\mu}$, see \eqref{eq: extremal decomposition measure}, for a given Gibbs measure  $\mu$ is the distribution of the tail-measurable 
kernels $\pi(\cdot| \omega)$ (see definition \eqref{eq: Limit pi-kernel}) if the $\omega$'s (which are spin-configurations at infinity) 
are chosen from the measure we want to decompose. 

Our core idea if we are on trees, is to view the kernel $\pi(\cdot| \omega) $ as a reconstruction device which 
receives a spin configuration  $\omega$ at infinity and gives back a (random) spin configuration $\sigma$. 
By general theory $\pi(d \sigma| \omega)$ is an extremal Gibbs measure for $\mu$-a.e. $\omega$, see \cite{Ge11}, \cite{Le08} or \hyperref[Sec: Appendix A]{Appendix A}.  
Expressing the abstract probabilistic situation in physical language, one can say that  
it is \textit{the (extremal) Gibbs measure 
 with boundary condition $\omega$ at infinity}, cp. \cite{CoKuLe24} or \cite{GaMaRuSh20}.  

Our method then consists in showing that in parameter regimes of reconstruction, not only single-site reconstruction  (which is commonly 
considered in information theory) holds, but also multi-site reconstruction in whole volumes. 
These volumes will be for us thinned infinite branches of the tree. 
The general idea in all of this is to show that $\sigma$ looks very much 
like $\omega$, with controlled errors. This in turn allows to distinguish two configurations $\omega,\omega'$ a.s. 
when they are independently chosen from the measure under consideration we want to reconstruct. 
Finally, this yields in particular the atomless-property of the decomposition measure, which hence must be 
supported on uncountably many pure states. 

How can all of this be applied to the $A$-concentrated states? 
The intuition why this could have a chance to work is to consider $\mu_A$ as a perturbation of the free measure 
for an $A$-spin model in which all spin-excitations outside of the set $A$ simply have been forgotten. 
If and how these excitations can be rigorously controlled is the main serious problem to be solved in the present work. 
To approach this difficulty  we first devise the notion of an $A$-irregular set of edges. 
 This  will be then used to quantify how much a certain 
"almost ground-state" $\omega$ differs from homogeneous (flat) groundstates in $A$. 
An essential part of the present proof is then to prove a "Peierls bound with errors" in our situation. This is done via 
Lemma \ref{lem: Key Lemma} which is a new main technical step of the proof. 
To yield this lemma, it turns out moreover 
that a local state-space coarse-graining is technically very helpful which only distinguishes 
whether the spin is in $A$ or in its complement. This may create other problems on first sight, as it necessarily 
destroys the (tree-indexed) Markov chain property (unless we are in the very special case of the Potts model).   
Nevertheless, we are able to use the coarse-graining constructively to prove our desired upper bounds 
on the contour activities via a non-stochastic upper bounding matrix $M$, see \eqref{eq: Matrix M}. Using Lemma \ref{lem: Key Lemma}, we are able to harvest our main results of Theorem \ref{thm: mu_A reconstruction theorem} and Theorem \ref{thm: A.s. singularity extremals}.

\medskip 

\textbf{Organization of the paper.} The remainder of the paper is organized as follows. Section \ref{sec: Math prel and main results} provides basic definitions and presents the setup of the model we consider in this work. Especially Theorem \ref{thm: Existence A-localized states} will be discussed which is a recent result in \cite{AbHeKuMa24} about the existence and properties of the so-called $A$-localized states $\mu_A$. Its consequences for our model are stated in Proposition \ref{prop: Bounds A-loc. states}. This section ends by giving our main result of this paper in Theorem \ref{thm: Main results}, namely that these $A$-localized states are not extremal and possess even an uncountable extremal decomposition in a low-temperature regime. The proof of this statement is then split into Section \ref{sec: Key lemma} and \ref{sec: Non-extremality}. In Section \ref{sec: Key lemma}, we will state and prove Lemma \ref{lem: Key Lemma} which is the new main aspect of this paper and essential to prove both statements of the main result (Theorem \ref{thm: Main results}). On the one hand, this lemma proves a Peierls bound with errors for regions with not-too dense occurrence of $A$-irregular edges. On the other hand, it controls such regions probabilistically.  
In more detail, it will turn out that these events are exponentially improbable in the inverse temperature and size of the region under $\mu_A$. The proof is based on our idea of coarse-graining the state space, see Lemma \ref{lem: Upper bound coarse grain}, together with the \textit{propagation of smallness} inductively over the contour activities, presented in Lemma \ref{lem: Propagation lemma}. Section \ref{sec: Non-extremality} deals with the consequences of Lemma \ref{lem: Key Lemma}, which are the restricted single-site reconstruction (Theorem \ref{thm: mu_A reconstruction theorem}) and the multi-site reconstruction (Theorem \ref{thm: Branch overlap}). The proof of the single-site reconstruction implies the positivity of the Edwards-Anderson parameter \eqref{def: Edwards-Anderson parameter}. For the proof of the multi-site reconstruction, one can consider thinned branch overlaps $\underline{\phi}^\omega$ which are tail-measurable observables and has to show that they concentrate under typical extremals $\pi(\cdot |\omega)$. Given this multi-site reconstruction, one is able to conclude the almost sure-singularity of the $\pi$-kernels entering the extremal decomposition of an $A$-localized state $\mu_A$ (Theorem \ref{thm: A.s. singularity extremals}). 
This very last probabilistic part follows with similar 
arguments as the ones given in \cite{CoKuLe24}, 
with relatively minor adaptations to our situation of the $A$-localized states. In this very last 
part we can keep the presentation short, but self-contained for the convenience of the reader, while 
we are more detailed in the first parts of the paper 
dealing with the control of $A$-irregularity.

\section{Mathematical preliminaries and main results}\label{sec: Math prel and main results}

Starting in Subsection \ref{subsec: GM on the tree}, we give a required background to analyze Gibbs measures on trees and introduce the setup of the model considered in this work. The constructions in this section can be found in the book of Georgii \cite{Ge11}, and are included 
here to provide a self-contained presentation. Subsection \ref{Subsec: Gibbs states with finite loc set} deals with Gibbs states which are localized on finite subsets of the state space. Existence results and some properties for our model are stated in Proposition \ref{prop: Bounds A-loc. states} which results from Theorem 3.1 in \cite{AbHeKuMa24}. Finally, we will give the main results of our paper in Subsection \ref{subsec: Main results} summarized in Theorem \ref{thm: mu_A reconstruction theorem}.

\subsection{Gibbs measures on trees}\label{subsec: GM on the tree}

\textbf{Countably regular trees.} Let $(V,E)$ be a tree where $V$ is the set of \textit{vertices} (or \textit{sites}) and $E:=\big\{\{x,y\}: x,y\in V\big\}$ the set of \textit{edges}. Moreover, we will consider \textit{rooted} trees in the sense that we emphasize one vertex which we denote with $\rho$ and call it the \textit{root}. Note that two vertices $x,y\in V$ are called \textit{nearest neighbors} if $\{x,y\}\in E$ and we write $x\sim y$. We will consider \textit{d-regular} trees for $d\in\N_{\geq 2}$, which are trees $(V,E)$ where each vertex has $d+1$ nearest neighbors. Let $x,y\in V$, then there exists a metric $d:V\times V \rightarrow \N_0$ where $d(x,y)$ is defined as the length of the unique minimal path from $x$ to $y$. For each $\Lambda \subset V$, we denote with $\partial \Lambda:=\{x\in V\setminus \Lambda: \exists y\in \Lambda: x\sim y\}$ the \textit{outer boundary} of $\Lambda$ and if $\Lambda$ is finite, it is common to write $\Lambda \Subset V$.\newline

\noindent\textbf{Spin configurations.} Let $q\in \N_{\geq 2}$. We assign to each vertex $x\in V$ a state $\omega_x\in \Z_q$ and call the family $\omega=(\omega_x)_{x\in V}$ a \textit{configuration} on the $d$-regular tree $(V,E)$. The set of all configurations $\Omega:=(\Z_q)^{V}$ is called \textit{configuration space} and the restriction on a subset $\Lambda \subset V$ will we denote with $\Omega_\Lambda:=(\Z_q)^\Lambda$. We define the projection on a subset $\Lambda\subset V$ as the map $\sigma_\Lambda:\Omega \rightarrow \Omega_\Lambda$ where $\sigma_\Lambda(\omega):=\omega_\Lambda$. Note that we write $\sigma_x$ instead of $\sigma_{\{x\}}$ for the projection on a single vertex $x\in V$ and we call this projection the \textit{spin} at $x$. Let $\Lambda\subset \Delta \subset V$, $\omega\in \Omega_\Lambda$ and $\eta \in \Omega_{\Delta \setminus \Lambda}$, then the \textit{concatenation} $\omega\eta\in \Omega_\Delta$ is defined as $\sigma_\Lambda(\omega\eta)=\omega$ and $\sigma_{\Delta \setminus \Lambda}(\omega\eta)=\eta$. The state space $\Z_q$ is endowed with the $\sigma$-algebra $\mathcal{P}(\Z_q)$ and the configuration space $\Omega$ with the product $\sigma$-algebra $\mathscr{F}:=\big(\mathcal{P}(\Z_q)\big)^{\otimes V}$. For a subset $\Delta \subset V$, we will also consider $\mathscr{F}_\Delta:=\sigma(\sigma_x,~x\in \Delta)$ the $\sigma$-algebra on $\Omega$ which is generated by all events occurring in $\Delta$.\newline

\noindent\textbf{Ferromagnetic nearest-neighbor clock potentials with $u,U,d$-bounds.}
In this work, we will consider ferromagnetic nearest-neighbor potentials $\Phi$ of the form
\begin{equation}\label{eq: Potential ferr n.n. model}
    \Phi_{\{x,y\}}(\omega)=\sum^{q-1}_{i,j=0}\overline{u}(|i-j|) \cdot \mathds{1}_{(\omega_x,\omega_y)=(i,j)}\cdot \mathds{1}_{x\sim y}
\end{equation}
where $\overline{u}$ is a function of the distance $|i-j|$ between $i$ and $j$ in $\Z_q$ which satisfies $\overline{u}(0)=0$. Models where the pair potential possesses discrete rotational symmetry are also called \textit{clock models}. Further, we assume that the function $\overline{u}$ fulfills the following \textit{$u,U,d$-bounds}
\begin{equation}\label{eq: u,U-bounds}
   0<u:=\min_{\substack{i,j\in \Z_q\\ i\neq j}}\overline{u}(|i-j|)\leq \max_{\substack{i,j\in \Z_q\\ i\neq j}}\overline{u}(|i-j|)=:U<\infty~~\text{and}~~(d^2+1)u>d \,U.
\end{equation}
 Hence, we need to pay at least an amount of energy $u$ and at most $U$ to change the state in a nearest neighbor of a given vertex. Moreover, the last inequality of \eqref{eq: u,U-bounds} ensures a one step jump back from $A^c$ into the finite localization $A\subset \Z_q$ with sufficiently high probability for the corresponding $A$-localized Gibbs state $\mu_A$, defined in Subsection \ref{Subsec: Gibbs states with finite loc set}.\newline

\noindent\textbf{Gibbsian specifications.} We would like to investigate models via the given interaction potential $\Phi$ and in order to do this, we will need the notion of the so-called \textit{specifications}. In general, these are families $\gamma=(\gamma_\Lambda)_{\Lambda \Subset V}$ of proper probability kernels each from $(\Omega,\mathscr{F}_{\Lambda^c})$ to $(\Omega,\mathscr{F})$ satisfying a consistency relation. A kernel $\gamma_\Lambda$ is called \textit{proper} if $\gamma_\Lambda(A|\omega)=\mathds{1}_A(\omega)$ for all $A\in \mathscr{F}_{\Lambda^c}$. Let $\Lambda \subset \Delta\Subset V$, then the two kernels $\gamma_\Lambda$ and $\gamma_\Delta$ should satisfy the following \textit{consistency relation} $(\gamma_\Delta \gamma_\Lambda)(A|\omega)=\gamma_\Delta(A|\omega)$ for all $A\in \mathscr{F}$ and $\omega\in \Omega$. 

A \textit{Gibbsian specification} $\gamma^\Phi=(\gamma^\Phi_\Lambda)_{\Lambda \Subset V}$ for a given potential $\Phi$ is defined as 
\begin{equation}\label{eq: Gibbs specification}
    \gamma^\Phi_\Lambda(\omega_\Lambda|\omega_{\Lambda^c})=\frac{1}{Z_\Lambda^{\omega}}e^{-\beta H^\Phi_\Lambda(\omega_\Lambda \omega_{\Lambda^c})} 
\end{equation}
for all $\Lambda \Subset V$ and $\omega\in \Omega$. Here, $Z_\Lambda^\omega$ is the normalization constant and 
\begin{equation}\label{def: Hamiltonian}
    H^\Phi_\Lambda(\omega_\Lambda \omega_{\Lambda^c}):=\sum_{\substack{A\subset V\\ A\cap \Lambda\neq \emptyset}} \Phi_A(\omega_\Lambda \omega_{\Lambda^c})
\end{equation}
is the corresponding Hamiltonian for the boundary condition $\omega_{\Lambda^c}$ outside of $\Lambda$.

 Define for each edge $b=\{x,y\}\in E$ and a given configuration $\omega\in \Omega$ a transfer operator $Q_b(\omega):=e^{-\beta \Phi_b(\omega)}$. In this way, we can rewrite the Gibbsian specification in \eqref{eq: Gibbs specification} as
\begin{equation*}
    \gamma^\Phi_\Lambda(\omega_\Lambda|\omega_{\Lambda^c})=\frac{1}{Z^\omega_\Lambda}\prod_{\substack{\{x,y\}\cap \Lambda \neq \emptyset\\ x\sim y}}Q_{\{x,y\}}(\omega_x,\omega_y).
\end{equation*}
For the brevity of notation, we will write $\gamma$ instead of $\gamma^\Phi$ for the Gibbsian specifications which will be analyzed in this work.\newline

\noindent\textbf{Gibbs measures and extremality.} We are interested in infinite-volume measures being consistent with the given specification in the sense that these kernels provide a regular conditional distribution for this measure. In more detail, let $\gamma=(\gamma_\Lambda)_{\Lambda \Subset V}$ be a specification and $\mu\in \mathscr{M}_1(\Omega,\mathscr{F})$ be a probability measure. Then, we call $\mu$ a \textit{Gibbs measure} for $\gamma$ if it satisfies the \textit{DLR-equations}
\begin{equation}\label{eq: DLR-equation}
    \mu(A|\mathscr{F}_{\Lambda^c})=\gamma_\Lambda(A|\cdot)~~~~\mu\text{-almost surely}
\end{equation}
for all $\Lambda \Subset V$ and $A\in \mathscr{F}$. We denote with $\mathscr{G}(\gamma)$ the set of all Gibbs measures for $\gamma$. The extreme points of this set, i.e. the measures $\mu\in \mathscr{G}(\gamma)$ such that there is no $\alpha\in (0,1)$ and $\mu_1,\mu_2\in \mathscr{G}(\gamma)$ with $\mu=\alpha\mu_1+(1-\alpha)\mu_2$ where $\mu_1\neq \mu_2$, are called \textit{extremal Gibbs measures} and the set of these points is denoted with $\text{ex }\mathscr{G}(\gamma)$.

Given a Gibbs measure $\mu\in \mathscr{G}(\gamma)$, one can obtain an extremal Gibbs measure via the so-called \textit{$\pi$-kernel with boundary condition $\omega$ at infinity}. It is defined as 
\begin{equation}\label{eq: Limit pi-kernel}
    \pi(\cdot|\omega):=\lim_{\Lambda\uparrow V}\gamma_\Lambda(\cdot|\omega)
\end{equation}
if the limit exists for all cylinder events (generating $\mathscr{F}$) and $\pi(\cdot|\omega)=:\nu$ with $\nu\in \mathscr{M}_1(\Omega)$ is an arbitrary probability measure else. Indeed, one is able to show that the limit in \eqref{eq: Limit pi-kernel} exists and it is even an extremal Gibbs measure for $\mu$-a.e. $\omega\in \Omega$. Furthermore, these $\pi$-kernels enter the extremal decomposition of an arbitrary Gibbs measure $\mu\in \mathscr{G}(\gamma)$. For more details, see \cite{FV17} and \cite{Ge11} or \hyperref[Sec: Appendix A]{Appendix A}.\newline

\noindent\textbf{Tree-indexed Markov chains.} There is a natural and important generalization for the notion of a Markov chain on a tree, which is now indexed by its vertices $V$. In order to introduce this, we consider directed edges of the tree whose set is denoted by $\Vec{E}$ and define the subset of edges pointing away from a vertex $w\in V$ as
\begin{equation*}
    \Vec{E}_w:=\{\langle x,y\rangle\in \Vec{E} :~ d(w,y)=d(w,x)+1\}.
\end{equation*}
With the help of this notion, we are able to define the past of a directed edge $\langle x,y\rangle$ as
\begin{equation*}
    (-\infty, xy):=\{w\in V :~\langle x,y\rangle \in \Vec{E}_w\}.
\end{equation*}
Note that the \say{past} of a directed edge is just meaningful in the absence of loops. Finally, a probability measure $\mu\in \mathscr{M}_1(\Omega)$ is called a \textit{tree-indexed Markov chain} if 
\begin{equation*}
    \mu(\sigma_y=\omega_y|~\mathscr{F}_{(-\infty,xy)})=\mu(\sigma_y=\omega_y|~\mathscr{F}_{\{x\}})
\end{equation*}
$\mu$-a.s. for any $\langle x,y \rangle\in \Vec{E}$ and any $\omega_y\in \Z_q$. 

Moreover, if 
\begin{equation*}
    \mu(\sigma_y=j|~\mathscr{F}_{\{x\}})=P(\sigma_x,j)~~~\mu\text{-a.s.}
\end{equation*}
for all $j\in \Z_q$ and all $\langle x,y \rangle\in \Vec{E}$, the Markov chain $\mu$ is said to be \textit{spatially homogeneous} with transition matrix $P$.

\subsection{Gibbs states with finite localization sets}\label{Subsec: Gibbs states with finite loc set}

There is the following existence result from which we start our present analysis, which can be found as Theorem 3.1 in \cite{AbHeKuMa24} where it includes 
also a more general form for integer-valued gradient models. For more background 
on gradient models and gradient states, 
see also \cite{CaLuEyMaSlTo14}, \cite{CoDeuMu09}, \cite{CKu15}, \cite{DaHaPe21}, \cite{EnKu08}, \cite{FS97}, \cite{LaTo24}, \cite{Sh05}, \cite{Ve06}.

\begin{theorem}\label{thm: Existence A-localized states}
    Let $d\geq 2$ be an integer and $A\subset \Z_q$ with $n:=|A|\geq 1$. There is a positive quantity $\eta=\eta(d,n)$ such that the following statement holds. Assume that the transfer operator is normalized by $Q(0)=1$ and satisfies the condition 
    \begin{equation}\label{eq: Condition epsilon loc. Thm}
        \epsilon:=\| Q- \mathds{1}_{\{0\}}\|_{\frac{d+1}{2}}\leq \eta(d,n).
    \end{equation}
    Then, the Gibbsian specification which is induced by $Q$ on the $d$-regular tree with local state space $\Z_q$ admits a spatially homogeneous tree-indexed Markov chain Gibbs measure $\mu_A$. Denote by 
    \begin{align}\label{def: single-site marg, diagonal elements}
         \nonumber&\pi_A: \Z_q \rightarrow [0,1],~~~~~\pi_A(i):=\mu_A(\sigma_\rho=i),\\
         &\Delta: \Z_q \rightarrow [0,1],~~~~~~\Delta(i):=P_A(i,i),
    \end{align}
    the functions giving the single-site marginals of $\mu_A$ and the diagonal elements of the transition matrix $P_A$. Moreover, the single-site marginals $\pi_A$ are strictly positive and the following estimates hold:
    \begin{align}\label{eq: bounds A-localized states}
        \nonumber(i) ~\|\Delta|_{A^c}&\|_{\frac{d+1}{d-1}}\leq c_1 \epsilon^{d-1}, ~~~~(ii) ~\min_{i\in A} \Delta(i)>1-c_2\epsilon,\\
        (iii) ~\sum_{i\in A^c} \pi_A(i) \leq c_3& \epsilon^{d+1},~~~~(iv)~(1-c_5\epsilon)\frac{1}{|A|}\leq \pi_A|_A \leq (1-c_4\epsilon)^{-1} \frac{1}{|A|}
    \end{align}
    where $c_1,...,c_5$ are positive constants depending on $d$ and $n$.
    
\end{theorem}

For more details, see the \hyperref[Sec: Appendix B]{Appendix B} or if one is interested in the entire result, see \cite{AbHeKuMa24}. The properties $(i)-(iv)$ state that the corresponding process of $\mu_A$ would like to stay in a state in $A$ and if it is outside of $A$, it would like to change its state, see for an illustration Figure \ref{fig: coarse grained state space}. Note that this theorem implies, especially for our situation, the following result.

\begin{proposition}\label{prop: Bounds A-loc. states}
   Consider a nearest-neighbor $\Z_q$-valued clock model with Gibbsian specification defined via the transfer operator $Q(i)=e^{-\beta \overline{u}(i)}$ where the function $\overline{u}$ satisfies the $u,U,d$-condition \eqref{eq: u,U-bounds}. Then, we have 
   \begin{equation}\label{eq: Bound variable epsilon}
    (q-1)^{\frac{2}{d+1}}e^{-\beta U}\leq\epsilon \leq (q-1)^{\frac{2}{d+1}}e^{-\beta u}.
\end{equation}
   Furthermore, let $ A \subset \Z_q$ with $n=|A|\geq 1$ and $\beta$ be large enough, then there exists a spatially homogeneous tree-indexed Markov chain Gibbs measure $\mu_A$ which is localized on the subset $A$. The transition matrix $P_A$ is strictly positive and the following bounds hold:  
   \begin{align*}
       (i) ~\|&\Delta|_{A^c}\|_{\frac{d+1}{d-1}}\leq C_1 e^{-(d-1)\beta u},~~~~(ii)~\min_{i\in A} \Delta(i)>1-C_2e^{-\beta u},\\
        (iii) ~\sum_{i\in A^c} \pi_A(i) \leq C_1&  e^{-(d+1)\beta u},~~~~(iv)~ \Big(1-(C_1+C_2)  e^{-\beta u} \Big) \frac{1}{|A|}\leq \pi_A|_A  \leq \left(1-C_2e^{-\beta u}\right)^{-1} \frac{1}{|A|}
   \end{align*}
   where $C_1:=2^{d+1}d^3(n+1)^{d}(q-1)^{2}$ and $C_2:=3d(d^{\frac{1}{d-1}}-1)(n+1)(q-1)$.
   
\end{proposition}

The proof of this Proposition is given in \hyperref[Sec: Appendix B]{Appendix B}.

\subsection{Main results}\label{subsec: Main results}

In fact, it turns out that the $A$-localized states of Proposition \ref{prop: Bounds A-loc. states} in a range of $u$ and $U$ are no longer extreme in the set of all Gibbs measures being consistent with $\gamma$. Moreover, their extremal decomposition, see \hyperref[Sec: Appendix A]{Appendix A}, is supported on uncountably many inhomogeneous Gibbs states. These results are summarized in the following theorem.

\begin{theorem}\label{thm: Main results}
     Consider a $\Z_q$-valued nearest-neighbor clock model with Hamiltonian given in \eqref{def: Hamiltonian} and a potential defined in \eqref{eq: Potential ferr n.n. model} fulfilling the $u,U,d$-bounds in \eqref{eq: u,U-bounds}. Furthermore, let $A\subset \Z_q$ with $|A|\geq 2$ and $\beta$ be large enough. Then, the $A$-localized state $\mu_A$ for this model, see Proposition \ref{prop: Bounds A-loc. states}, fulfills the following properties:
     \begin{enumerate}[label=\roman*)]
         \item $\mu_A$ is not extremal in $\mathscr{G}(\gamma)$ and moreover,
         \item the corresponding extremal decomposition measure $\alpha_{\mu_A}$, see \eqref{eq: extremal decomposition measure}, appearing in the extremal decomposition 
                    \begin{equation}
        \mu_A=\int_{\text{ex }\mathscr{G}(\gamma)} \nu \cdot \alpha_{\mu_A}(d\nu)
    \end{equation}
    is supported on uncountably many inhomogeneous states, in particular $\alpha_{\mu_A}(\{\nu\})=0$ for all $\nu\in \text{ex }\mathscr{G}(\gamma)$.
     \end{enumerate}
\end{theorem}

The proofs of both parts of this theorem are based on the important results of Lemma \ref{lem: Key Lemma} which will be discussed in the next section. In Section \ref{sec: Non-extremality}, we deduce the two statements of Theorem \ref{thm: Main results} where the ideas of the proofs are adapted from \cite{CoKuLe24} but the properties of the $A$-localized states are treated carefully. Part $i)$ already follows then from a restricted single-site reconstruction bound for $\mu_A$ with explicit error bounds, see Theorem \ref{thm: mu_A reconstruction theorem}. Whereas part $ii)$ is a consequence of the almost-sure singularity of two $\pi$-kernels with different boundary conditions drawn from $\mu_A$, see Theorem \ref{thm: A.s. singularity extremals}. The proof of the almost-sure singularity relies on a multi-site reconstruction in the sense that we have reconstruction on thinned branch overlaps, compare Theorem \ref{thm: Branch overlap}.

\section{Control of $A$-irregularity: smallness of $B_A(\gamma)$ under $\mu_A$}\label{sec: Key lemma}

In the single- and multi-site reconstruction (Theorem \ref{thm: mu_A reconstruction theorem} and \ref{thm: Branch overlap}), the $\pi$-kernels $\pi(\cdot|\omega)$, see \eqref{eq: Limit pi-kernel}, are used as a reconstruction device for $\mu_A$-typical configurations $\omega$. Therefore, we want to show that a configuration $\sigma$ drawn from $\pi(\cdot|\omega)$ mostly resembles $\omega$ locally. This will be done by considering contours relative to $\omega$ which were  introduced in \cite{CoKuLe23}. They are characterized by connected components of mismatches of $\sigma$ with $\omega$.  

\begin{definition}[Contours with respect to a fixed configuration $\omega$]
    Let $\omega \in \Omega$ be a fixed reference configuration. A \uline{contour} for the spin configuration $\eta\in \Omega$ relative to $\omega$ is a pair
    \begin{equation*}
        \Bar{\gamma}=(\gamma,\eta_\gamma)
    \end{equation*}
    where the support $\gamma\subset \{x\in V:\eta_x\neq \omega_x\}$ is a finite connected component of the set of incorrect points for $\eta$ (with respect to $\omega$), and $\eta_\gamma=(\eta_x)_{x\in \gamma}$.
\end{definition}

A configuration $\omega$ drawn from an $A$-localized state $\mu_A$ typically contains a small density of \textit{$A$-irregular edges}, i.e. edges where the vertices change their state or one state is located outside of $A$. More precisely, they are defined as follows and an illustration of these $A$-irregular edges in a coarse-grained view, is given in Figure \ref{fig: broadcasting on the tree}.

\begin{definition}[Set of $A$-irregular edges]\label{def: bad events}
    Let $A\subset \Z_q$. We define the set of \uline{$A$-irregular edges} of the configuration $\omega \in \Omega$ by
    \begin{equation*}
        D_A(\omega):=\{\{x,y\}\in E:~\omega_x\neq \omega_y~\text{or}~\omega_x\in A^c~\text{or}~\omega_y\in A^c\}.
    \end{equation*}
    Furthermore, denote the set of edges attached to $\gamma \subset V$ by 
    \begin{equation*}
        E(\gamma):=\{\{x,y\}\in E,\{x,y\}\cap \gamma \neq \emptyset\}.
    \end{equation*}
\end{definition}

 However, the set of irregular edges is in general not uniformly sparse as needed in \cite{CoKuLe23} and \cite{GaRuSh12} to have excess energy estimates which give us reasonable Peierls bounds. Therefore, we need to treat these \say{bad events}, i.e. configurations containing regions with a large number of $A$-irregular edges.
 
 \begin{definition}[Bad events]
    Consider any model with Hamiltonian of the form \eqref{def: Hamiltonian}. Let
    \begin{equation}\label{eq: Definition delta_0}
        \delta_0:=\frac{1}{2}\cdot \frac{(d-1)u}{(u+U)}
    \end{equation}
    where $u,U\in \R_+$ fulfilling the bounds in \eqref{eq: u,U-bounds} and $d$ is the branching number of the tree.

    \noindent Denote by $B_{A,v}$ the \uline{bad event} that there exists a contour around $v$ with respect to $\omega$ which does not have good enough excess energy in the sense that 
    \begin{equation}\label{eq: Bad events}
        B_{A,v}:=\bigcup_{\Bar{\gamma}: v\in \gamma} B_A(\gamma)~~\text{where}~~B_A(\gamma):=\{\omega: |D_A(\omega)\cap E(\gamma)|\geq \delta_0|\gamma|\}.
    \end{equation}
\end{definition}
 
 We will see that these events are exponentially improbable under $\mu_A$ for small temperatures which will be subject of the following lemma. In order to control these bad events, we apply our new ideas of considering the model with a coarse-grained state space, compare Lemma \ref{lem: Upper bound coarse grain}, and propagate small quantities over problematic $A$-irregular edges of a given contour, see Lemma \ref{lem: Propagation lemma}.

\begin{lemma}[Control of $A$-irregularity]\label{lem: Key Lemma}
    Let $A\subset\Z_q$ with $|A|\geq 2$ and $A\neq \Z_q$. Furthermore, let $\beta$ be large enough and $\mu_A$ be the $A$-localized state for the ferromagnetic nearest-neighbor clock model with potential \eqref{eq: Potential ferr n.n. model} fulfilling the $u,U,d$-bounds in \eqref{eq: u,U-bounds}. Then, there exist functions $\epsilon_1(\beta)$, $\epsilon_2(\beta)$ and $\lambda(\beta)$ satisfying $\lim_{\beta\rightarrow \infty}\epsilon_1(\beta)=\lim_{\beta\rightarrow \infty}\epsilon_2(\beta)=0$ and $\lim_{\beta\rightarrow \infty}\lambda(\beta)=\infty$ such that the following statements hold.
    \begin{enumerate}[label=\alph*)]
        \item For $\mu_A$-almost every $\omega$ and every $v\in V$, we have
        \begin{equation}\label{eq: bound pi-Kernel key lemma}
        \pi(\sigma_v\neq \omega_v|\omega) \leq \mathds{1}_{B_{A,v}}(\omega)+\epsilon_1(\beta).
    \end{equation}
    Here, $B_{A,v}$ is the bad event defined in \eqref{eq: Bad events}.
    \item Moreover, the following inequality holds
    \begin{equation}\label{eq: ineq mu_A key lemma}
        \mu_A(B_A(\gamma)) \leq C_4e^{-\lambda(\beta) |\gamma|}
    \end{equation}
     for each finite subtree $\gamma \subset V$ where $B_A(\gamma)$ was given in \eqref{eq: Bad events}. Hence, we have $\mu_A(B_{A,v})\leq \epsilon_2(\beta)$ for any $v\in V$.
    \end{enumerate}
    The function $\lambda$ in $b)$ can be chosen as $\lambda(\beta)=\Tilde{c}+c\delta_0 \beta$ where $c:= \min\{\frac{d^2+1}{d}u-U,\frac{1}{d}\}$ and 
    \begin{equation*}
        \Tilde{c}:=d\log\left(\delta_0 \right)-d\log(d+1)-\delta_0\log(C_1+2C_3).
    \end{equation*}
    Here, $\delta_0$ is given in \eqref{eq: Definition delta_0} and $C_1$ is the positive constant defined in Proposition \ref{prop: Bounds A-loc. states}. Furthermore,
    \begin{equation*}
        \epsilon_1(\beta)=C_3e^{-\frac{(d-1)}{2}\beta u}~~\text{and}~~\epsilon_2(\beta)=C_5e^{-\lambda(\beta)}
    \end{equation*}
    and
    \begin{align*}
        C_3:=2(q-1),~~C_4:=2C_1(d+1)/\delta_0~~\text{and}~~C_5:=2(q-1)C_4.
    \end{align*}    
\end{lemma}

The essential part of the proof will be to show part $b)$ and this will be done in Subsection \ref{Subsec: Coarse graining method}-\ref{subsec: Computation of a minimizer} below. One can show then part $a)$ by steps similar to \cite{CoKuLe24} and using additionally the fact that $D_{\Z_q}(\omega)\subset D_A(\omega)$. The inequality $\mu_A(B_{A,v})\leq \epsilon_2(\beta)$ follows from \eqref{eq: ineq mu_A key lemma} by summation over all possible contours and using entropy bounds for the number of possible contours. 

\begin{remark}
    Note that the proof idea can be applied to the case $A=\Z_q$ and one would obtain a function $\epsilon_2(\beta)$ which satisfies $\epsilon_2(\beta)\downarrow 0$ as $\beta \uparrow \infty$ without demanding the relation $(d^2+1)u>d\cdot U$ in the $u,U,d$-condition \eqref{eq: u,U-bounds}. This is because we do not have to deal with the problematic situations where we have to ensure a one-step jump probability from $A^c$ to $A$ under the transition matrix $P_A$, see the proof idea in Lemma \ref{lem: Propagation lemma}.
\end{remark}

\textbf{Preparation.} It remains to prove the validity of inequality \eqref{eq: ineq mu_A key lemma}. First of all, we apply the exponential Markov inequality to bound $\mu_A(B_A(\gamma))$ as follows
\begin{align}\label{eq: application exp Markov}
    \nonumber\mu_A(B_A(\gamma))&=\mu_A(|D_A(\sigma)\cap E(\gamma)|\geq \delta_0 |\gamma|)\\
    &\leq \inf_{t\geq 0} e^{-t\delta_0 |\gamma|} \mu_A(e^{t|D_A(\sigma)\cap E(\gamma)|}).
\end{align}
In order to compute a minimizer, we need to analyze the expectation $\mu_A(e^{t|D_A(\sigma)\cap E(\gamma)|})$. For a given finite subtree $\gamma \subset V$ with a fixed vertex $x\in \gamma$, we can write 
\begin{equation}\label{eq: Markov chain prob}
     \mu_A\big(e^{t|D_A(\sigma) \cap E(\gamma)|}\big)=\sum_{\sigma_{\Bar{\gamma}}\in (\Z_q)^{\Bar{\gamma}}} \pi_A(\sigma_x) \prod_{(u,v)\in \Vec{E}_x(\gamma)}P_A(\sigma_u,\sigma_v) e^{t \mathds{1}_{\{u,v\}\in D_A(\sigma)}}
\end{equation}
where we used the Markov chain property of the measure $\mu_A(\cdot|\sigma_x=a)$ with its homogeneous transition matrix $P_A$. Moreover, we denote by
\begin{equation*}
  \Vec{E}_x(\gamma)=\{(u,v)\in \Vec{E}_x: \{u,v\}\cap \gamma \neq \emptyset\}
\end{equation*}
the set of directed edges which are attached to $\gamma \subset V$ and pointing away from the vertex $x\in V$.

\textbf{Outline of the proof.} The remaining steps for the proof of part $b)$ in Lemma \ref{lem: Key Lemma} are divided into three subsections.  In Subsection \ref{Subsec: Coarse graining method}, we switch to the model with a coarse-grained state space which just distinguishes if a spin is located in $A$ or in $A^c$, see Figure \ref{fig: coarse grained state space}. The expectation in \eqref{eq: Markov chain prob} can then be upper bounded by a similar quantity of the coarse-grained model, see Lemma \ref{lem: Upper bound coarse grain}. The coarse-grained model is described via a non-stochastic transition matrix $M$, defined in \eqref{eq: Matrix M}, assigning weights to each edge in the contour. Afterwards in Subsection \ref{Subsec: Propagation of smallness}, we show that this expression is small even for large $t$ via the idea of \textit{propagation of smallness}, presented in Lemma \ref{lem: Propagation lemma}. In this method, we propagate small quantities from the inside to the outside of the subtree over the problematic $A$-irregular edges where one jumps from $A^c$ to $A$ which might become very probable but carry the large $e^t$ factor. Both proofs of Lemma \ref{lem: Upper bound coarse grain} and Lemma \ref{lem: Propagation lemma} are based on an induction over the number of elements of the subtrees $\gamma \Subset V$. Finally in Subsection \ref{subsec: Computation of a minimizer}, we compute a unique minimizer for the upper bound of \eqref{eq: application exp Markov}, see Lemma \ref{lem: Minimizer of upper bounded expectation}, and present how this implies the statement of part $b)$ in Lemma \ref{lem: Key Lemma}.

\subsection{The coarse-graining method for $\Z_q$} \label{Subsec: Coarse graining method}

We will introduce for each vertex $x\in V$ the random variable $\tau_x:=\mathds{1}_{\sigma_x\in A}$ which we consider to be local coarse-graining. Moreover, let us define the matrix $M=\big(M(a,b)\big)_{a,b\in \{0,1\}}$ whose entries are given as
    \begin{align}\label{eq: Matrix M}
    \nonumber &M(0,0):= \sup_{a\in A^c} P_A(a,A^c)e^t,~~M(1,0):=\sup_{a\in A} P_A(a,A^c)e^t,\\
    M(0,1):=&\sup_{a\in A^c} P_A(a,A)e^t~~\text{and}~~M(1,1):=\sup_{a\in A}\big(P_A(a,A\setminus\{a\})e^t +P_A(a,a)\big).
\end{align}
Note that $M$ is in general not a stochastic matrix. We are able to give a bound of \eqref{eq: Markov chain prob} in terms of $M$ and $\pi_A$ as follows:
\begin{lemma}\label{lem: Upper bound coarse grain}
    For each finite subtree $\gamma \subset V$ with a given vertex $x\in \gamma$, the expression \eqref{eq: Markov chain prob} is upper bounded by 
    \begin{equation}\label{eq: coarse grained probability}
    \sum_{\tau_{\Bar{\gamma}}\in \{0,1\}^{\Bar{\gamma}}} \pi_A(\tau_x) \prod_{(u,v)\in \Vec{E}_x(\gamma)}M(\tau_u,\tau_v)
\end{equation}
where the single-site marginals $\pi_A$ of $\mu_A$ are given in \eqref{def: single-site marg, diagonal elements}.
\end{lemma}

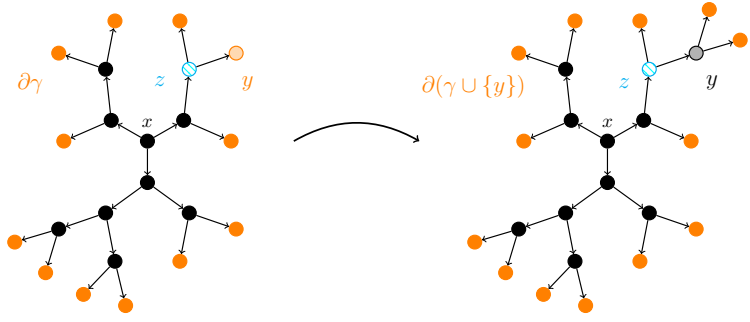
\begin{figure}[ht!]
    \centering
\scalebox{0.55}{
\begin{tikzpicture}[every label/.append style={scale=1.3}]
    
    \node[shape=circle,draw=black, fill=black, thick, minimum size=0.2cm, label=above:{$x$}] (A) at (0,0) {};
    
    \node[shape=circle,draw=black, fill=black,  thick, minimum size=0.2cm] (B) at (0,-1) {};
    \node[shape=circle,draw=black, fill=black,   thick, minimum size=0.2cm] (C) at (0.866,0.5) {};
    \node[shape=circle,draw=black, fill=black,  thick, minimum size=0.2cm] (D) at (-0.866,0.5) {};

    \node[shape=circle,draw=black, fill=black,thick, minimum size=0.2cm] (BA) at (1,-1.7321) {};
    \node[shape=circle,draw=black, fill=black,thick, minimum size=0.2cm] (BB) at (-1,-1.7321) {};

    \node[shape=circle,draw=orange, fill=orange,thick, minimum size=0.2cm] (BAA) at (2.1213,-2.1214) {};
    \node[shape=circle,draw=orange,  fill=orange, thick, minimum size=0.2cm] (BAB) at (0.7764,-2.8978) {};

    \node[shape=circle,draw=black, fill=black, thick, minimum size=0.2cm] (BBA) at (-0.7764,-2.8978) {};
    \node[shape=circle,draw=black, fill=black,  thick, minimum size=0.2cm] (BBB) at (-2.1213,-2.1214) {};

    \node[shape=circle,draw=orange, fill=orange,thick, minimum size=0.2cm] (BBAA) at (-0.522,-3.9658) {};
    \node[shape=circle,draw=orange, fill=orange,thick, minimum size=0.2cm] (BBAB) at (-1.5307,-3.6956) {};

    \node[shape=circle,draw=orange, fill=orange, thick, minimum size=0.2cm] (BBBA) at (-2.435,-3.1735) {};
    \node[shape=circle,draw=orange, fill=orange,thick, minimum size=0.2cm] (BBBB) at (-3.1734,-2.4352) {};

    \node[shape=circle,draw=orange, fill=orange, thick, minimum size=0.2cm] (CA) at (2,0) {};
    \node[shape=circle,draw=cyan, pattern=north west lines, pattern color=cyan, thick, minimum size=0.2cm] (CB) at (1,1.732) {};

    \node[minimum size=0.5cm, label=:{{\large\color{cyan}$z$}}] at (0.3,0.9212)  {};

    \node[shape=circle,draw=orange, fill=orange,  thick, minimum size=0.2cm] (CBA) at (0.7765,2.8977) {};
    \node[shape=circle,draw=orange,  fill=orange!30,  thick, minimum size=0.2cm] (CBB) at (2.1213,2.1212) {};

    \node[minimum size=0.5cm, label=:{{\large\color{orange}$y$}}] at (2.4,0.812)  {};

    \node[minimum size=0.5cm, label=:{\large{\color{orange}$\partial \gamma $}}] at (-2.8,0.7212)  {};

    \node[shape=circle,draw=orange, fill=orange, thick, minimum size=0.2cm] (DA) at (-2,0) {};
    \node[shape=circle,draw=black, fill=black, thick, minimum size=0.2cm] (DB) at (-1,1.732) {};

    \node[shape=circle,draw=orange, fill=orange, thick, minimum size=0.2cm] (DBA) at (-2.1213,2.1212) {};
    \node[shape=circle,draw=orange, fill=orange, thick, minimum size=0.2cm] (DBB) at (-0.7765,2.8977) {};

    \draw [->, thick] (A) -- (B);
    \draw [->, thick] (A) -- (C);
    \draw [->, thick, draw=black] (A) -- (D);

    \draw [->, thick] (B) -- (BA);
    \draw [->, thick] (B) -- (BB);

    \draw [->, thick, draw=black] (BA) -- (BAA);
    \draw [->, thick] (BA) -- (BAB);

    \draw [->, thick] (BB) -- (BBA);
    \draw [->, thick] (BB) -- (BBB);

    \draw [->, thick] (BBA) -- (BBAA);
    \draw [->, thick] (BBA) -- (BBAB);

    \draw [->, thick] (BBB) -- (BBBA);
    \draw [->, thick] (BBB) -- (BBBB);

    \draw [->, thick] (C) -- (CA);
    \draw [->, thick] (C) -- (CB);

    \draw [->, thick] (CB) -- (CBA);
    \draw [->, thick] (CB) -- (CBB);

    \draw [->, thick] (D) -- (DA);
    \draw [->, thick] (D) -- (DB);

    \draw [->, thick] (DB) -- (DBA);
    \draw [->, thick] (DB) -- (DBB);

    \draw[->, black, very thick] (3.5,0) to  [bend left] node [text width=2.5cm,midway,above=0.2cm, align=center] {} (6.5,0);

    \node[shape=circle,draw=black, fill=black, thick, minimum size=0.2cm, label=above:{$x$}] (W) at  (11,0) {};

    \node[shape=circle,draw=black, fill=black, thick, minimum size=0.2cm] (X) at (11,-1) {};
    \node[shape=circle,draw=black, fill=black, thick, minimum size=0.2cm] (Y) at (11.866,0.5) {};
    \node[shape=circle,draw=black, fill=black, thick, minimum size=0.2cm] (Z) at (10.134,0.5) {};

    \node[shape=circle,draw=black, fill=black, thick, minimum size=0.2cm] (XA) at (12,-1.7321) {};
    \node[shape=circle,draw=black, fill=black, thick, minimum size=0.2cm] (XB) at (10,-1.7321) {};

    \node[shape=circle,draw=orange, fill=orange, thick, minimum size=0.2cm] (ZA) at (9,0) {};
    \node[shape=circle,draw=black, fill=black, thick, minimum size=0.2cm] (ZB) at (10,1.732) {};

    \node[shape=circle,draw=orange, fill=orange, thick, minimum size=0.2cm] (ZBA) at (8.8787,2.1212) {};
    \node[shape=circle,draw=orange, fill=orange, thick, minimum size=0.2cm] (ZBB) at (10.2235,2.8977) {};

    \node[shape=circle,draw=orange, fill=orange,thick, minimum size=0.2cm] (XAA) at (13.1213,-2.1214) {};
    \node[shape=circle,draw=orange, fill=orange, thick, minimum size=0.2cm] (XAB) at (11.7764,-2.8978) {};

    \node[shape=circle,draw=black, fill=black, thick, minimum size=0.2cm] (XBA) at (10.2236,-2.8978) {};
    \node[shape=circle,draw=black, fill=black, thick, minimum size=0.2cm] (XBB) at (8.8787,-2.1214) {};

    \node[shape=circle,draw=orange, fill=orange, thick, minimum size=0.2cm] (XBAA) at (10.478,-3.9658) {};
    \node[shape=circle,draw=orange, fill=orange, thick, minimum size=0.2cm] (XBAB) at (9.4693,-3.6956) {};

    \node[shape=circle,draw=orange, fill=orange, thick, minimum size=0.2cm] (XBBA) at (8.565,-3.1735) {};
    \node[shape=circle,draw=orange, fill=orange, thick, minimum size=0.2cm] (XBBB) at (7.8266,-2.4352) {};

    \node[shape=circle,draw=orange, fill=orange, thick, minimum size=0.2cm] (YA) at (13,0) {};
    \node[shape=circle,draw=cyan, pattern=north west lines, pattern color=cyan, thick, minimum size=0.2cm] (YB) at (12,1.732) {};

     \node[minimum size=0.5cm, label=:{{\large\color{cyan}$z$}}] at (11.4,0.8712)  {};

    \node[shape=circle,draw=orange, fill=orange, thick, minimum size=0.2cm] (YBA) at (11.7765,2.8977) {};
    \node[shape=circle,draw=black, fill=black!30, thick, minimum size=0.2cm] (YBB) at (13.1213,2.1212) {};

     \node[minimum size=0.5cm, label=:{{\large\color{black}$y$}}] at (13.5,0.8212)  {};

    \node[minimum size=0.5cm, label=:{{\large\color{orange}$\partial (\gamma\cup \{y\})$}}] at (7.8,0.6212)  {};

    \node[shape=circle,draw=orange, fill=orange, thick, minimum size=0.2cm] (YBBA) at (13.435,3.1733) {};
    \node[shape=circle,draw=orange, fill=orange, thick, minimum size=0.2cm] (YBBB) at (14.1734,2.4349) {};

    \draw [->,thick] (W) -- (X);
    \draw [->,thick] (W) -- (Y);
    \draw [->,thick] (W) -- (Z);

    \draw [->,thick] (X) -- (XA);
    \draw [->,thick] (X) -- (XB);
    \draw [->,thick] (Z) -- (ZA);
    \draw [->,thick] (Z) -- (ZB);

    \draw [->,thick] (XA) -- (XAA);
    \draw [->,thick] (XA) -- (XAB);

    \draw [->,thick] (XB) -- (XBA);
    \draw [->,thick] (XB) -- (XBB);

    \draw [->,thick] (XBA) -- (XBAA);
    \draw [->,thick] (XBA) -- (XBAB);

    \draw [->,thick] (XBB) -- (XBBA);
    \draw [->,thick] (XBB) -- (XBBB);

    \draw [->,thick] (Y) -- (YA);
    \draw [->,thick] (Y) -- (YB);

    \draw [->,thick] (YB) -- (YBA);
    \draw [->,thick] (YB) -- (YBB);

    \draw [->,thick] (YBB) -- (YBBA);
    \draw [->,thick] (YBB) -- (YBBB);

     \draw [->,thick] (ZB) -- (ZBA);
    \draw [->,thick] (ZB) -- (ZBB);

    \end{tikzpicture}
}
\caption{Depicted are two directed finite subtrees of a binary tree with edges pointing away from a fixed vertex $x$. The left picture depicts a finite subtree $\gamma \subset V$ with vertices in black and one vertex $z\in \gamma$ is hatched and highlighted in blue. Moreover, the outer boundary $\partial \gamma$ is illustrated in orange (gray if you are looking at a
black and white version of this picture) where one of the vertices $y\in \partial \gamma$ which is a child of $z$ is highlighted with a light circle. The right subtree is an extended version of $\gamma$ where the outer vertex $y$ was added to $\gamma$. Hence it contains two new vertices in the boundary, the children of $y$}
\label{fig: Induction propagation lemma}

\end{figure}

\begin{proof}
    In order to prove this, we will show that
    \begin{align}\label{eq: Induction hypothesis coarse grain}
        \nonumber &\sum_{\sigma_{\gamma}\in (\Z_q)^\gamma}\pi_A(\sigma_x) \prod_{\substack{(u,v)\in \Vec{E}_x(\gamma)\\ v\in \gamma}}P_A(\sigma_u,\sigma_v)e^{t\mathds{1}_{\{u,v\}\in D_A(\sigma)}}\prod_{\substack{(u,v)\in \Vec{E}_x(\gamma)\\ v\notin \gamma}}\sum_{\sigma_v\in K_v}P(\sigma_u,\sigma_v)e^{t\mathds{1}_{\{u,v\}\in D_A(\sigma)}}\\
        &\leq \sum_{\tau_\gamma\in \{0,1\}^\gamma} \pi_A(\tau_x)\prod_{(u,v)\in \Vec{E}_x(\gamma)}M(\tau_u,\tau_v)
    \end{align}
    for all subtrees $\gamma\Subset V$ and all coarse-grained boundary conditions $K_{\partial\gamma}\in \{A,A^c\}^{\partial\gamma}$ with its corresponding family $\tau_{\partial\gamma}\in \{0,1\}^{\partial\gamma}$, i.e. 
    \begin{equation*}
        \tau_x=\begin{cases}
            1~~\text{if}~~K_x=A,\\
            0~~\text{if}~~K_x=A^c
        \end{cases}
    \end{equation*}
    for each $x\in \partial\gamma$. Note that \eqref{eq: coarse grained probability} follows directly from \eqref{eq: Induction hypothesis coarse grain} after the summation over all $K_{\partial\gamma}\in \{A,A^c\}^{\partial\gamma}$. 
    
    We will prove \eqref{eq: Induction hypothesis coarse grain} via an induction over the number of elements in the finite subtree $\gamma \subset V$. Note that if $|\gamma|=n\in \N$, there is a sequence of finite subtrees $(\gamma_i)_{i\in\{1,...,n\}}$ such that $\gamma_n=\gamma$, $\gamma_i \subset \gamma_{i+1}$ and $|\gamma_{i+1}\setminus \gamma_{i}|=1$ for all $i\in \{1,...,n-1\}$. One can obtain this sequence by starting with an arbitrary element $x\in\gamma$ and define $\gamma_1:=\{x\}$. If $\gamma_i$ is given for an $i\in \{1,...,n-1\}$, choose an $y\in \gamma\setminus \gamma_i$ such that there exists a $z\in \gamma_i$ with $z\sim y$ and define $\gamma_{i+1}:=\gamma_i \cup \{z\}$. An example of the step from a subtree $\gamma_{10}$ to a subtree $\gamma_{11}$ is illustrated in Figure \ref{fig: Induction propagation lemma}.

    \underline{Induction start:} ($|\gamma|=1)$ Assume that $\gamma=\{x\}$ and $K_{\partial \gamma}\in \{A,A^c\}^{\partial \gamma}$ with its corresponding family $\tau_{\partial \gamma}\in \{0,1\}^{\partial \gamma}$ is given. First of all, we can write the left hand side of \eqref{eq: Induction hypothesis coarse grain} as
    \begin{align}\label{eq: Induction start coarse grain 1}
        &\sum_{\sigma_x\in A}\pi_A(\sigma_x)\prod_{y\sim x} \sum_{\sigma_y\in K_y}P_A(\sigma_x,\sigma_y)e^{t\mathds{1}_{\{x,y\}\in D_A(\sigma)}}+\sum_{\sigma_x\in A^c}\pi_A(\sigma_x)\prod_{y\sim x} \sum_{\sigma_y\in K_y}P_A(\sigma_x,\sigma_y)e^{t}
    \end{align}
    where we split the sum over $\sigma_x$. Note that 
    \begin{equation}\label{eq: Induction start coarse grain 2}
        \sum_{\sigma_y\in K_y}P_A(\sigma_x,\sigma_y)e^{t\mathds{1}_{\{x,y\}\in D_A(\sigma)}}\leq M(1,0)\mathds{1}_{K_y=A^c}+M(1,1)\mathds{1}_{K_y=A}
    \end{equation}
    for all $\sigma_x\in A$ and $y\sim x$ as well as 
     \begin{equation}\label{eq: Induction start coarse grain 3}
        \sum_{\sigma_y\in K_y}P_A(\sigma_x,\sigma_y)e^{t}\leq M(0,0)\mathds{1}_{K_y=A^c}+M(0,1)\mathds{1}_{K_y=A}
    \end{equation}
    for all $\sigma_x\in A^c$ and $y\sim x$. Combining \eqref{eq: Induction start coarse grain 2} and \eqref{eq: Induction start coarse grain 3}  gives us 
    \begin{equation*}
        \sum_{\sigma_x\in A}\pi_A(\sigma_x)\prod_{y\sim x}M(1,\tau_y)+\sum_{\sigma_x\in A^c}\pi_A(\sigma_x)\prod_{y\sim x}M(0,\tau_y)
    \end{equation*}
    as an upper bound for \eqref{eq: Induction start coarse grain 1} and hence the statement of \eqref{eq: Induction hypothesis coarse grain} for $|\gamma|=1$.

    \underline{Induction step:} Let us assume that \eqref{eq: Induction hypothesis coarse grain} is true for a subtree $\gamma \Subset V$ with $|\gamma|\in \N$ and assume that $y\in \partial \gamma$. Define $\gamma_0:=\gamma\cup \{y\}$ and let $K_{\partial\gamma_0}\in \{A,A^c\}^{\partial\gamma_0}$ together with its corresponding coarse-grained family $\tau_{\partial\gamma_0}\in \{0,1\}^{\partial\gamma_0}$, compare Figure \ref{fig: Induction propagation lemma}. Then, we can rewrite the expression in \eqref{eq: Induction hypothesis coarse grain} as 
    \begin{align}\label{eq: Induction step coarse grain 1}
        \nonumber &\sum_{\sigma_{\gamma_0}\in (\Z_q)^{\gamma_0}}\pi_A(\sigma_x)\prod_{\substack{(u,v)\in \Vec{E}_x(\gamma_0)\\ v\in \gamma_0}}P_A(\sigma_u,\sigma_v)e^{t\mathds{1}_{\{u,v\}\in D_A(\sigma)}}\\
        &\times\prod_{\substack{(u,v)\in \Vec{E}_x(\gamma_0)\\ v\notin \gamma_0,~u\neq y}}\sum_{\sigma_v\in K_v}P_A(\sigma_u,\sigma_v)e^{t\mathds{1}_{\{u,v\}\in D_A(\sigma)}}\prod_{\substack{\Tilde{y}\sim y\\ \Tilde{y}\neq z}}\sum_{\sigma_{\Tilde{y}}\in K_{\Tilde{y}}}P_A(\sigma_y,\sigma_{\Tilde{y}})e^{t\mathds{1}_{\{y,\Tilde{y}\}\in D_A(\sigma)}}
    \end{align}
    by separating the terms corresponding to the directed edges starting in $y$. Here, $z$ is the unique nearest neighbor of $y$ laying inside of $\gamma_0$, see Figure \ref{fig: Induction propagation lemma}. We can upper bound each factor corresponding to a vertex $\Tilde{y}\in \partial\{y\}\setminus \{z\}$ as follows
    \begin{equation}\label{eq: Induction step coarse grain 2}
        \sum_{\sigma_{\Tilde{y}}\in K_{\Tilde{y}}}P_A(\sigma_y,\sigma_{\Tilde{y}})e^{t\mathds{1}_{\{y,\Tilde{y}\}\in D_A(\sigma)}}\leq M(1,\tau_{\Tilde{y}})\mathds{1}_{\sigma_y\in A}+M(0,\tau_{\Tilde{y}})\mathds{1}_{\sigma_y\in A^c}
    \end{equation}
    by considering the different coarse-grained possibilities of $\sigma_y$. Using the bound of \eqref{eq: Induction step coarse grain 2} in \eqref{eq: Induction step coarse grain 1} gives us 
    \begin{align}\label{eq: Induction step coarse grain 3}
        \nonumber &\sum_{\sigma_{\gamma}\in (\Z_q)^{\gamma}}\pi_A(\sigma_x)\prod_{\substack{(u,v)\in \Vec{E}_x(\gamma)\\ v\in \gamma}}P_A(\sigma_u,\sigma_v)e^{t\mathds{1}_{\{u,v\}\in D_A(\sigma)}} \prod_{\substack{(u,v)\in \Vec{E}_x(\gamma_0)\\ v\notin \gamma_0,~u\neq y}}\sum_{\sigma_v\in \eta_v}P_A(\sigma_u,\sigma_v)e^{t\mathds{1}_{\{u,v\}\in D_A(\sigma)}}\\
        &\times\left(\sum_{\sigma_y\in A}P_A(\sigma_z,\sigma_y)e^{t\mathds{1}_{\{z,y\}\in D_A(\sigma)}}\prod_{\substack{\Tilde{y}\sim y\\ \Tilde{y}\neq z}}M(1,\tau_{\Tilde{y}})+\sum_{\sigma_y\in A^c}P_A(\sigma_z,\sigma_y)e^{t}\prod_{\substack{\Tilde{y}\sim y\\ \Tilde{y}\neq z}}M(0,\tau_{\Tilde{y}})\right)
    \end{align}
    where we additionally separated the sum over $\sigma_y$. Afterwards, we apply the induction hypothesis \eqref{eq: Induction hypothesis coarse grain} for the subtree $\gamma$. For the first term, we use it in the case of $K_{\partial\gamma}\in \{A,A^c\}^{\partial\gamma}$ where $K_y=A$ and for the second term in the case where $K_y=A^c$. This results in the following upper bound for \eqref{eq: Induction step coarse grain 1} 
    \begin{equation*}
        \sum_{\tau_\gamma\in \{0,1\}^\gamma}\pi_A(\tau_x)\prod_{\substack{(u,v)\in \Vec{E}_x(\gamma)\\ v\neq y}}M(\tau_u,\tau_v)\left(M(\tau_z,1)\prod_{\substack{\Tilde{y}\sim y\\ \Tilde{y}\neq z}}M(1,\tau_{\Tilde{y}})+M(\tau_z,0)\prod_{\substack{\Tilde{y}\sim y\\ \Tilde{y}\neq z}}M(0,\tau_{\Tilde{y}})\right)
    \end{equation*}
    and consequently the statement of \eqref{eq: Induction hypothesis coarse grain} follows for $\gamma_0$. This ends the proof of Lemma \ref{lem: Upper bound coarse grain}.
\end{proof}

\subsection{Propagation of smallness}\label{Subsec: Propagation of smallness}

Let us proceed with the proof of \eqref{eq: ineq mu_A key lemma} in Lemma \ref{lem: Key Lemma}. We can achieve upper bounds for the entries of $M$ and $\pi_A$ related to the inverse temperature with the results of Proposition \ref{prop: Bounds A-loc. states}. These bounds are summarized in the following lemma and the proof can be found in \hyperref[Sec: Appendix C]{Appendix C}.

\begin{lemma}\label{lem: Bounds for M}
Let $\beta$ be large enough. Then, the single-site marginal of $\mu_A$, which was defined in \eqref{def: single-site marg, diagonal elements}, is bounded as follows
\begin{equation}\label{eq: Bounds for pi_A}
    \pi_A(A^c)\leq C_1 e^{-(d+1)\beta u}~~\text{and}~~\pi_A(A)\leq 1
\end{equation}
where $C_1$ is the positive constant given in Proposition \ref{prop: Bounds A-loc. states}. Moreover, the entries of the matrix $M$, introduced in \eqref{eq: Matrix M}, can be upper bounded in the following way
\begin{gather}
    \nonumber M(1,1)\leq C_3 e^{-\beta u+t}+1,~~M(1,0)\leq C_3 e^{-(d+1)\beta u+t}\\
    M(0,1)\leq  e^{t}~~\text{and}~~M(0,0)\leq C_3 e^{-\beta((d+1)u-U)+t}+C_1 e^{-(d-1)\beta u+t}.\label{eq: Bounds for M}
\end{gather}
The positive constant $C_3$ was introduced in Lemma \ref{lem: Key Lemma}.
\end{lemma}

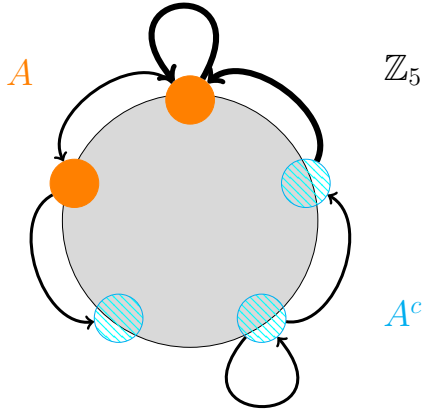
\begin{figure}[ht!]
\centering
\scalebox{0.8}{
\begin{tikzpicture}[every label/.append style={scale=1}]

 \node[shape=circle,draw=orange, fill=orange,   minimum size=0.8cm] (0) at (0,2) {};
 \node[shape=circle,draw=orange, fill=orange,  minimum size=0.8cm] (1) at (-1.9021130325903071442328786667587642868113972682515004448946112888,0.6180339887498948482045868343656381177203091798057628621354486226) {};
 
 \node[shape=circle,draw=cyan, pattern=north west lines, pattern color=cyan,  minimum size=0.8cm] (2) at (-1.1755705045849462583374119092781455371953048752862919821445449613,-1.6180339887498948482045868343656381177203091798057628621354486226) {};
 \node[shape=circle,draw=cyan,  pattern=north west lines, pattern color=cyan,minimum size=0.8cm] (3) at (1.1755705045849462583374119092781455371953048752862919821445449613,-1.6180339887498948482045868343656381177203091798057628621354486226) {};
 \node[shape=circle,draw=cyan, pattern=north west lines, pattern color=cyan,  minimum size=0.8cm] (4) at (1.9021130325903071442328786667587642868113972682515004448946112888,0.6180339887498948482045868343656381177203091798057628621354486226) {};

\draw[->,bend angle=80, bend right,line width=0.05cm] (0) to (1);
\draw[->,bend angle=80, bend right,line width=0.05cm] (1) to (2);
\draw[->,bend angle=80, bend right,line width=0.05cm] (3) to (4);
\draw[->,bend angle=80, line width=0.1cm, bend right] (4) to (0);

\draw[->,loop above, distance=2cm, in=130,out=60, line width=0.1cm] (0) to (0);
\draw[->,loop below, distance=2cm, in=-50,out=-130, line width=0.05cm] (3) to (3);

\begin{scope}[on background layer]
\node[shape=circle,draw=black, fill=black!15,  minimum size=4.2cm] (6) at (0,0) {};
\end{scope}


\node[label=:{\LARGE$\Z_5$}] () at (3.5,2) {};

\node[label=:{{\LARGE\color{orange}$A$}}] () at (-2.8,2) {};

\node[label=:{{\LARGE\color{cyan}$A^c$}}] () at (3.5,-2) {};

\end{tikzpicture}
}
\caption{An illustration of the coarse-grained state space $\Z_5$ where the localization set $A$ consists of two elements. The localization set $A$ is colored in full orange and its complement is hatched and colored in blue.  Moreover, the behavior of the transition matrix $P_A$ is visualized with arrows where transitions are more probable if the arrows are thicker}
\label{fig: coarse grained state space}
\end{figure}

Exploiting this lemma, we give an upper bound for the expression in \eqref{eq: coarse grained probability} of the coarse-grained process. 

\begin{lemma}[Propagation lemma]\label{lem: Propagation lemma}
   Let $\beta$ be large enough. Then, we have the following inequality
   \begin{equation*}
       \mu_A(e^{t|D_A(\sigma)\cap E(\gamma)|})\leq 2C_1\big(f(\beta)e^t+1\big)^{1+d|\gamma|}
   \end{equation*}
    for each finite subtree $\gamma\subset V$. Here $f(\beta):=(2C_3+C_1)e^{-c\beta}$ with $c=\min\{\frac{d^2+1}{d}u-U,\frac{1}{d}\}$, $C_1$ is the positive constant given in Proposition \ref{prop: Bounds A-loc. states} and $C_3$ is the positive constant given in Lemma \ref{lem: Key Lemma}.
\end{lemma}

 \begin{remark}
     Note that the form of the $d$-dependent exponents in \eqref{eq: Bounds for M} is important for the proof of the Propagation Lemma.  As in the proof of Lemma \ref{lem: Upper bound coarse grain}, this proof is based on an induction over the elements of a finite subtree $\gamma$. The idea is now to propagate small parts of the expression for a given finite subtree $\gamma \subset V$ over the bad $A$-irregular edges after adding a new vertex to $\gamma$. This is the reason for the terms $\frac{1}{d}$ in the constant $c$ because in each step, we split the term $e^{-\beta u}$ uniformly over the $d$ children of a given vertex, as we will see. 
 \end{remark}

  \begin{proof}[Proof of Lemma \ref{lem: Propagation lemma}]
  We will show that 
\begin{gather} \label{eq: Claim Induction}
    \sum_{\tau_{\gamma}\in \{0,1\}^{\gamma}}\pi_A(\tau_x)\prod_{(u,v)\in \Vec{E}_x(\gamma)}M(\tau_u,\tau_v) \leq 2C_1 \big(f(\beta)e^t+1\big)^{|\gamma|-1}\prod_{y\in\partial \gamma}g(\tau_y)
\end{gather}
holds for all finite subtrees $\gamma \subset V$ with $x\in \gamma$ and $\tau_{\partial\gamma}\in \{0,1\}^{\partial\gamma}$. Here, the function $g:\{0,1\}\rightarrow \R$ is defined as
\begin{equation*}
    g(0):=C_3e^{-\beta((d+1+\frac{1}{d})u-U)+t}+C_1e^{-(d-1+\frac{1}{d})\beta u+t}~~\text{and}~~g(1):=C_3e^{-\beta \frac{u}{d}+t}+1.
\end{equation*} 
Consequently, summing \eqref{eq: Claim Induction} up over $\tau_{\partial \gamma}\in \{0,1\}^{\partial\gamma}$ leads to
\begin{gather*}
   \sum_{\tau_{\Bar{\gamma}}\in \{0,1\}^{\Bar{\gamma}}}\pi_A(\tau_x)\prod_{(u,v)\in \Vec{E}_x(\gamma)}M(\tau_u,\tau_v) \leq 2C_1\big(f(\beta)e^t+1\big)^{|\gamma|-1}\bigg(g(0)+g(1)\bigg)^{|\partial \gamma|}.
\end{gather*}
Note that the relation $|\partial \gamma|=2+|\gamma|(d-1)$ is valid for all finite subtrees $\gamma \Subset V$ of a $d$-regular tree. Together with the fact that $U\geq u$, we obtain the statement of Lemma \ref{lem: Propagation lemma}. As mentioned before, we will prove \eqref{eq: Claim Induction} via an induction over the elements in a finite subtree $\gamma\Subset V$.
  
      \underline{Induction start:} ($|\gamma|=1$) Assume that $\gamma=\{x\}$ and $\tau_{\partial{\gamma}}\in \{0,1\}^{\partial{\gamma}}$ is given. On the one hand, we have that
    \begin{equation*}
        \pi_A(A^c) \prod_{ y\sim x}M(0,\tau_y)\leq C_1e^{-(d+1)\beta u} \prod_{y\sim x}\bigg((C_3e^{-\beta((d+1)u-U)+t}+C_1e^{-(d-1)\beta u+t})\mathds{1}_{\tau_y=0}+e^{t}\mathds{1}_{\tau_y=1}\bigg)
    \end{equation*}
    where we used the bounds in \eqref{eq: Bounds for pi_A} and \eqref{eq: Bounds for M}.
    Splitting the term $e^{-\beta(d+1)u}$ over the $(d+1)$ nearest neighbors of $x$ leads to
    \begin{equation}\label{eq: Bound induction 1}
    C_1 \prod_{y\sim x}\bigg((C_3e^{-\beta((d+2)u-U)+t}+C_1e^{-d\beta u+t})\mathds{1}_{\tau_y=0}+e^{-\beta u+t}\mathds{1}_{\tau_y=1}\bigg).
    \end{equation}
    On the other hand, 
    \begin{equation}\label{eq: Bound induction 2}
        \pi_A(A) \prod_{y\sim x}M(1,\tau_y)\leq \prod_{y\sim x}\bigg(C_3e^{-(d+1)\beta u+t}\mathds{1}_{\tau_y=0}+(C_3e^{-\beta u+t}+1)\mathds{1}_{\tau_y=1}\bigg)
    \end{equation}
    with the bounds in \eqref{eq: Bounds for pi_A} and \eqref{eq: Bounds for M}. Combining \eqref{eq: Bound induction 1} and \eqref{eq: Bound induction 2} results in
    \begin{align*}
        &\sum_{\tau_x\in \{0,1\}}\pi_A(\tau_x) \prod_{y\sim x}M(\tau_x,\tau_y)\\
        &\leq 2C_1\prod_{y\sim x}\bigg((C_3e^{-\beta((d+1+\frac{1}{d})u-U)+t}+C_1e^{-(d-1+\frac{1}{d})\beta u+t})\mathds{1}_{\tau_y=0}+(C_3e^{-\beta \frac{u}{d}+t}+1)\mathds{1}_{\tau_y=1}\bigg)
    \end{align*}
    where we used the relation $0<u\leq U$ and the fact that $C_1>1$ as well as $C_3>1$, see Proposition \ref{prop: Bounds A-loc. states} and Lemma \ref{lem: Key Lemma}. This is inequality \eqref{eq: Claim Induction} for $|\gamma|=1$.

    \underline{Induction step:} Let us assume that \eqref{eq: Claim Induction} holds for a finite subtree $\gamma\Subset V$ with $|\gamma|=n$ and for all $\tau_{\partial\gamma}\in \{0,1\}^{\partial\gamma}$. Now, let $y\in \partial \gamma$ and $\tau_{\partial(\gamma\cup \{y\})}\in \{0,1\}^{\partial(\gamma\cup \{y\})}$ be an arbitrary boundary condition, see Figure \ref{fig: Induction propagation lemma}. Note that
   \begin{align*}
       &\sum_{\tau_{\gamma\cup \{y\}}\in \{0,1\}^{\gamma\cup \{y\}}}\pi_A(\tau_x)\prod_{(u,v)\in \Vec{E}_x(\gamma\cup \{y\})}M(\tau_u,\tau_v)\\
       &=\sum_{\tau_y\in \{0,1\}} \sum_{\tau_\gamma \in \{0,1\}^\gamma} \pi_A(\tau_x)\prod_{(u,v)\in \Vec{E}_x(\gamma)}M(\tau_u,\tau_v) \prod_{\substack{z\in \partial \{y\}\\z\notin \gamma}}M(\tau_y,\tau_z).
   \end{align*}
   Using the induction hypothesis for both cases $\tau_y\in \{0,1\}$, gives us the following upper bound
   \begin{align}\label{eq: Bound induction 3}
       2C_1\big(f(\beta)e^t+1\big)^{|\gamma|-1}\prod_{\Tilde{y}\in\partial \gamma\setminus \{y\}}g(\tau_{\Tilde{y}}) \sum_{\tau_y\in \{0,1\}}g(\tau_y)\prod_{\substack{z\in \partial y\\z\notin \gamma}}M(\tau_y,\tau_z)
   \end{align}
   where we have separated the factor corresponding to $y$. Let us consider the terms depending on $y$, which are explicitly 
   \begin{equation}\label{eq: Bound induction 4}
      (C_3e^{-\beta((d+1+\frac{1}{d})u-U)+t}+C_1e^{-(d-1+\frac{1}{d})\beta u+t})\prod_{\substack{z\in \partial y\\z\notin \gamma}}M(0,\tau_z)+(C_3e^{-\beta \frac{u}{d}+t}+1)\prod_{\substack{z\in \partial y\\z\notin \gamma}}M(1,\tau_z).
   \end{equation}
The first term can be upper bounded by 
\begin{align}\label{eq: Bound induction 5}
    \nonumber&(C_3e^{-\beta((d+\frac{1}{d})u-U)+t}+C_1e^{-(d-2+\frac{1}{d})\beta u+t})\\
    &\times\prod_{\substack{z\in \partial y\\z\notin \gamma}}\bigg((C_3e^{-\beta((d+1+\frac{1}{d})u-U)+t}+C_1e^{-(d-1+\frac{1}{d})\beta u+t})\mathds{1}_{\tau_z=0}+e^{-\beta \frac{u}{d}+t}\mathds{1}_{\tau_z=1}\bigg)
\end{align}
where we used the bounds in \eqref{eq: Bounds for M} and split the term $e^{-\beta u}$ in the prefactor over the $d$ factors in the product. Using the bounds in \eqref{eq: Bounds for M} for the second term in \eqref{eq: Bound induction 4} gives the following upper bound 
\begin{equation}\label{eq: Bound induction 6}
    (C_3e^{-\beta \frac{u}{d}+t}+1)\prod_{\substack{z\in \partial y\\z\notin \gamma}}\bigg(C_3e^{-(d+1)\beta u+t}\mathds{1}_{\tau_z=0}+(C_3e^{-\beta u+t}+1)\mathds{1}_{\tau_z=1}\bigg).
\end{equation}
Combining \eqref{eq: Bound induction 5} and \eqref{eq: Bound induction 6} together with $0<u\leq U$ and $C_3>1$ leads to the following upper bound of \eqref{eq: Bound induction 4}
\begin{align*}
   &\big(C_3 e^{-\beta(\frac{d^2+1}{d}u-U)+t}+C_1e^{-(d-2+\frac{1}{d})\beta u+t}+C_3e^{-\beta \frac{u}{d}+t}+1\big)\\
   &\times\prod_{\substack{z\in \partial\gamma\\z\notin\gamma}}\bigg((C_3e^{-\beta((d+1+\frac{1}{d})u-U)+t}+C_1 e^{-(d-1+\frac{1}{d})\beta u+t})\mathds{1}_{\tau_z=0}+(C_3e^{-\beta \frac{u}{d}+t}+1)\mathds{1}_{\tau_z=1}\bigg).
\end{align*}
Exploiting this bound in \eqref{eq: Bound induction 3} and recall that $d\geq 2$ results in the desired upper bound of \eqref{eq: Claim Induction} for $|\gamma|=n+1$. 
\end{proof}

\subsection{Computation of a minimizer}\label{subsec: Computation of a minimizer}
Let us now conclude the proof of \eqref{eq: ineq mu_A key lemma} in Lemma \ref{lem: Key Lemma}. After the application of Lemma \ref{lem: Propagation lemma} in the upper bound of \eqref{eq: application exp Markov}, we obtain 
\begin{equation}\label{eq: bound expectation bad event 2}
   \mu_A(B_A(\gamma))\leq 2C_1\inf_{t\geq 0} e^{-t\delta_0|\gamma|}\big(f(\beta)e^t+1\big)^{1+d|\gamma|}
\end{equation}
and it remains to minimize this expression.

\begin{lemma}\label{lem: Minimizer of upper bounded expectation}
Let $h: \R \rightarrow \R$ be a function defined as
\begin{equation*}
    h(t):=e^{-t\delta_0|\gamma|}(f(\beta)e^{t}+1)^{1+d|\gamma|}
\end{equation*}
for all $t\in \R$. If $\beta$ is small enough, $h$ possesses a unique minimizer 
    \begin{equation}\label{eq: Minimizer of g}
        t^*=\log\left(\frac{\delta_0 |\gamma|}{(1+d|\gamma|)f(\beta)-\delta_0|\gamma|f(\beta)}\right)
    \end{equation}
    where the quantity $\delta_0$ is introduced in \eqref{eq: Definition delta_0} and the function $f$ in Lemma \ref{lem: Propagation lemma}.
\end{lemma}

\begin{proof}
    The first derivative of $h$ reads 
\begin{equation*}
    h'(t)=-\delta_0|\gamma|e^{-t\delta_0|\gamma|}(f(\beta)e^{t}+1)^{1+d|\gamma|}+(1+d|\gamma|)f(\beta)e^{t(1-\delta_0|\gamma|)}(f(\beta)e^{t}+1)^{d|\gamma|}
\end{equation*}
and hence a critical point $t^*$ is given as a solution of the equation
\begin{equation*}
    \delta_0|\gamma|(f(\beta)e^{t^*}+1)=(1+d|\gamma|)f(\beta)e^{t^*}
\end{equation*}
where we already cancelled the term $e^{-t^*\delta_0|\gamma|}(f(\beta)e^{t}+1)^{d|\gamma|}$. Solving for $f(\beta)e^{t^*}$ gives us 
\begin{equation*}
    f(\beta)e^{t^*}=\frac{\delta_0 |\gamma|}{(1+d|\gamma|)-\delta_0|\gamma|}
\end{equation*}
and thus the unique critical point of  $t^*$ has the form as stated in \eqref{eq: Minimizer of g}. Moreover, one can argue that $t^*$ is a minimizer because $\lim_{t\rightarrow \infty}h(t)=\lim_{t\rightarrow -\infty}h(t)=+\infty$. One can see $\lim_{t\rightarrow \infty}h(t)=\infty$ by rewriting $h(t)$ as
\begin{equation*}
    h(t)=e^{t(1+d|\gamma|-\delta_0|\gamma|)}(f(\beta)+e^{-t})^{1+d|\gamma|}
\end{equation*}
and note that $\delta_0<d$.
\end{proof}
Inserting the minimizer in $h$ leads to 
\begin{equation}\label{eq: Minimizer inserted}
    h(t^*)=\left(\frac{1+|\gamma|(d-\delta_0)}{\delta_0 |\gamma|}\right)^{\delta_0 |\gamma|}\left(\frac{1+d|\gamma|}{1+|\gamma|(d-\delta_0)}\right)^{1+d|\gamma|}f(\beta)^{\delta_0 |\gamma|}.
\end{equation}
Note that 
\begin{equation*}
    \frac{1+|\gamma|(d-\delta_0)}{\delta_0 |\gamma|}=\frac{1}{\delta_0 |\gamma|}+\frac{d}{\delta_0}-1\geq 1
\end{equation*}
because 
\begin{equation*}
    \frac{d}{\delta_0}=2\frac{u+U}{u}\frac{d}{d-1}\geq 2.
\end{equation*}
Consequently, we can bound \eqref{eq: Minimizer inserted} from above by
\begin{equation*}
   \left(\frac{1+d|\gamma|}{\delta_0 |\gamma|}\right)^{1+d|\gamma|}f(\beta)^{\delta_0 |\gamma|}
\end{equation*}
where we upper bounded the exponent $\delta_0|\gamma|$ of the first term by $1+d|\gamma|$. Finally, note that this expression is bounded from above by 
\begin{equation*}
   \left(\frac{1+d}{\delta_0}\right)^{1+d|\gamma|}f(\beta)^{\delta_0 |\gamma|}.
\end{equation*}
Combining everything we have done so far in \eqref{eq: bound expectation bad event 2} gives us
\begin{equation*}
    \mu_A(B_A(\gamma))\leq 2C_1\left(\frac{1+d}{\delta_0}\right)\left(\left(\frac{1+d}{\delta_0}\right)^{d}f(\beta)^{\delta_0}\right)^{|\gamma|}.
\end{equation*}
Inserting the definition of $\delta_0$, see \eqref{eq: Definition delta_0}, and the function $f$, introduced in Lemma \ref{lem: Propagation lemma}, one can write it in the form $C_4e^{-\lambda(\beta)|\gamma|}$ where $C_4$ and $\lambda(\beta)$ are given in Lemma \ref{lem: Key Lemma}. Finally, this ends the proof of part b) in Lemma \ref{lem: Key Lemma}.

\section{Consequences: Single-site and multi-site reconstruction}\label{sec: Non-extremality}

Using Lemma \ref{lem: Key Lemma}, we first prove in Subsection \ref{subsec: single-site reconstruction bound} a restricted single-site reconstruction bound (Theorem \ref{thm: mu_A reconstruction theorem}) and deduce $i)$ of Theorem \ref{thm: Main results}. Moreover, we show in Subsection \ref{subsec: Multi-site reconstruction} a multi-site reconstruction on thinned branches of the tree (Theorem \ref{thm: Branch overlap}) which leads to the almost-sure singularity of the $\pi$-kernels (Theorem \ref{thm: A.s. singularity extremals}) and hence to statement $ii)$ of Theorem \ref{thm: Main results}.

\subsection{Single-site reconstruction}\label{subsec: single-site reconstruction bound}

\begin{theorem}[$\mu_A$ single-site reconstruction theorem]\label{thm: mu_A reconstruction theorem}

Consider a $\Z_q$-valued nearest-neighbor clock model with Hamiltonian given in \eqref{def: Hamiltonian} and a potential defined in \eqref{eq: Potential ferr n.n. model} fulfilling the $u,U,d$-bounds in \eqref{eq: u,U-bounds}. Furthermore, let $A\subsetneq \Z_q$ with $|A|\geq 2$ and $\beta$ be large enough. Then, the $A$-localized state $\mu_A$ for this model, see Proposition \ref{prop: Bounds A-loc. states}, satisfies the following bound
\begin{equation}\label{eq: reconstruction bound inequality}
    \mu_A\bigg(\big\{\omega:~\pi(\sigma_\rho=a|\omega)< 1-\epsilon_1(\beta)\big\}\bigg| \sigma_\rho=a\bigg) \leq 2|A|\epsilon_2(\beta)
\end{equation}
for any $a\in A$. Here, $\epsilon_1(\beta)\downarrow 0$ and $\epsilon_2(\beta)\downarrow 0$ as $\beta \uparrow \infty$ and both quantities were introduced in Lemma \ref{lem: Key Lemma}.
\end{theorem}

\begin{remark}
    The statement of Theorem \ref{thm: mu_A reconstruction theorem} can be interpreted from a signal reconstruction point of view as follows: Assume that we send the information $\{\sigma_\rho=a\}$ with the conditional probability $\mu_A(\cdot| \sigma_\rho=a)$ to infinity. Let $\omega$ be drawn from this conditional probability of $\mu_A$. In a small temperature regime, the best tail-measurable predictor $\pi(\sigma_0=a|\omega)$ is now able to reconstruct this information if $a\in \Z_q$ was a state in the localization set $A$. Note that this works only for $a\in A$ because the process of $\mu_A$ prefers with a very high probability to stay in this state, see Figure \ref{fig: coarse grained state space} and Proposition \ref{prop: Bounds A-loc. states}. Therefore, the information gets lost only on a very few branches from the origin to infinity. Moreover, note that the probability to preserve the information increases by increasing the degree of the tree. Compare this discussion with picture i) in Figure \ref{fig: broadcasting on the tree}.

    On the other hand, if $a\in A^c$, the corresponding process of $\mu_A$ would like to change its state, see Figure \ref{fig: coarse grained state space} and Proposition \ref{prop: Bounds A-loc. states}. Once it jumps into a state of $A$ it prefers to stay there and hence the information $\{\sigma_\rho=a\}$ gets lost very fast on the way from the origin to the outside of the tree. For an illustration of this idea, see picture ii) in Figure \ref{fig: broadcasting on the tree}.  Therefore, the reconstruction property can not be expected to hold for all $a\in \Z_q$ and has to be restricted to $A$. 
\end{remark}

\begin{figure}[h]
  \centering
  \scalebox{0.55}{
  \begin{minipage}[b]{0.8\textwidth}
    \begin{tikzpicture}[every label/.append style={scale=1.3}, rotate=30]

\node[label=:{\Large i)}] () at (7.8,0.4) {};

\node[label=:{{\Large\color{orange}$A$}}] () at (-0.5,2.5) {};

\node[label=:{{\Large\color{cyan}$A^c$}}] () at (-2.7,-1.3) {};

\node[label=:{{\large\color{black}$\rho$}}] () at (-1.5,0.5) {};

    \node [shape=circle,draw=orange, fill=orange, thick, minimum size=0.4cm] (B) at (0,-1) {};
    \node[shape=circle,draw=orange, fill=orange, thick, minimum size=0.4cm] (C) at (0.866,0.5) {};
    \node[shape=circle,draw=orange, fill=orange, thick, minimum size=0.4cm] (A) at (-0.866,0.5) {};

    \node[shape=circle,draw=orange, fill=orange,thick, minimum size=0.4cm] (BA) at (1,-1.7321) {};
    \node[shape=circle,draw=orange, fill=orange, thick, minimum size=0.4cm] (BB) at (-1,-1.7321) {};

    \node[shape=circle,draw=orange, fill=orange, thick, minimum size=0.4cm] (BAA) at (2.1213,-2.1214) {};
    \node[shape=circle,draw=orange, fill=orange,  thick, minimum size=0.4cm] (BAB) at (0.7764,-2.8978) {};

    \node[shape=circle,draw=orange, fill=orange, thick, minimum size=0.4cm] (BBA) at (-0.7764,-2.8978) {};
    \node[shape=circle,draw=orange, fill=orange,thick, minimum size=0.4cm] (BBB) at (-2.1213,-2.1214) {};

    \node[shape=circle,draw=orange, fill=orange, thick, minimum size=0.4cm] (BAAA) at (3.1734,-2.4352) {};
    \node[shape=circle,draw=orange, fill=orange,thick, minimum size=0.4cm] (BAAB) at (2.4350,-3.1735) {};

    \node[shape=circle,draw=orange, fill=orange,thick, minimum size=0.4cm] (BABA) at (1.5307,-3.6956) {};
    \node[shape=circle,draw=orange, fill=orange,thick, minimum size=0.4cm] (BABB) at (0.5220,-3.9658) {};

    \node[shape=circle,draw=orange,fill=orange,thick, minimum size=0.4cm] (BBAA) at (-0.522,-3.9658) {};
    \node[shape=circle,draw=orange, fill=orange,thick, minimum size=0.4cm] (BBAB) at (-1.5307,-3.6956) {};

    \node[shape=circle,draw=orange,fill=orange,thick, minimum size=0.4cm] (BBBA) at (-2.435,-3.1735) {};
    \node[shape=circle,draw=orange,fill=orange,thick, minimum size=0.4cm] (BBBB) at (-3.1734,-2.4352) {};

    \node[shape=circle,draw=orange,fill=orange,thick, minimum size=0.4cm] (BAAAA) at (4.1573,-2.778) {};
    \node[shape=circle,draw=orange,fill=orange,thick, minimum size=0.4cm] (BAAAB) at (3.7735,-3.2805) {};

    \node[shape=circle,draw=orange,fill=orange,thick, minimum size=0.4cm] (BAABA) at (3.2967,-3.7593) {};
    \node[shape=circle,draw=orange,fill=orange,thick, minimum size=0.4cm](BAABB) at (2.7959,-4.1453) {};

    \node[shape=circle,draw=orange,fill=orange,thick, minimum size=0.4cm] (BABAA) at (2.2114,-4.4845) {};
    \node[shape=circle,draw=orange,fill=orange,thick, minimum size=0.4cm] (BABAB) at (1.6278,-4.7277) {};

    \node[shape=circle,draw=orange,fill=orange,thick, minimum size=0.4cm] (BABBA) at (0.9753,-4.9040) {};
    \node[shape=circle,draw=orange,fill=orange,thick, minimum size=0.4cm] (BABBB) at (0.3487,-4.9878) {};

    \node[shape=circle,draw=cyan,pattern=north west lines, pattern color=cyan,thick, minimum size=0.4cm] (BBAAA) at (-0.3269,-4.9893) {};
    \node[shape=circle,draw=orange,fill=orange,thick, minimum size=0.4cm] (BBAAB) at (-0.9539,-4.9082) {};

    \node[shape=circle,draw=orange,fill=orange,thick, minimum size=0.4cm] (BBABA) at (-1.6071,-4.7348) {};
    \node[shape=circle,draw=orange,fill=orange,thick, minimum size=0.4cm] (BBABB) at (-2.1918,-4.4941) {};

    \node[shape=circle,draw=orange,fill=orange,thick, minimum size=0.4cm] (BBBAA) at (-2.7778,-4.1575) {};
    \node[shape=circle,draw=orange,fill=orange,thick, minimum size=0.4cm] (BBBAB) at (-3.2802,-3.7737) {};

    \node[shape=circle,draw=orange,fill=orange,thick, minimum size=0.4cm] (BBBBA) at (-3.7592,-3.2969) {};
    \node[shape=circle,draw=orange,fill=orange,thick, minimum size=0.4cm] (BBBBB) at (-4.1452,-2.7962) {};

    \node[shape=circle,draw=orange,fill=orange,thick, minimum size=0.4cm] (BAAAAA) at (5.0952,-3.1686) {};
    \node[shape=circle,draw=orange,fill=orange,thick, minimum size=0.4cm] (BAAAAB) at (4.877,-3.495) {};

    \node[shape=circle,draw=orange,fill=orange,thick, minimum size=0.4cm] (BAAABA) at (4.6546,-3.7863) {};
    \node[shape=circle,draw=orange,fill=orange,thick, minimum size=0.4cm](BAAABB) at (4.3970,-4.0827) {};

    \node[shape=circle,draw=orange,fill=orange,thick, minimum size=0.4cm](BAABAA) at (4.1015,-4.3793) {};
    \node[shape=circle,draw=orange,fill=orange,thick, minimum size=0.4cm] (BAABAB) at (3.8063,-4.6382) {};

    \node[shape=circle,draw=orange, fill=orange,thick, minimum size=0.4cm] (BAABBA) at (3.516,-4.8619) {};
    \node[shape=circle,draw=orange,fill=orange,thick, minimum size=0.4cm] (BAABBB) at (3.1905,-5.0815) {};

    \node[shape=circle,draw=orange,fill=orange,thick, minimum size=0.4cm]  (BABAAA) at (2.8283,-5.2917) {};
    \node[shape=circle,draw=orange,fill=orange,thick, minimum size=0.4cm]  (BABAAB) at (2.4762,-5.4653) {};

    \node[shape=circle,draw=orange,fill=orange,thick, minimum size=0.4cm]  (BABABA) at (2.1379,-5.6063) {};
    \node[shape=circle,draw=orange,fill=orange,thick, minimum size=0.4cm]  (BABABB) at (1.7667,-5.7341) {};

    \node[shape=circle,draw=orange,fill=orange,thick, minimum size=0.4cm]  (BABBAA) at (1.3623,-5.8434) {};
    \node[shape=circle,draw=orange,fill=orange,thick, minimum size=0.4cm]  (BABBAB) at (0.9772,-5.9199) {};

    \node[shape=circle,draw=orange, fill=orange,thick, minimum size=0.4cm] (BABBBA) at (0.6141,-5.9685) {};
    \node[shape=circle,draw=orange,fill=orange,thick, minimum size=0.4cm] (BABBBB) at (0.2224,-5.9958) {};

    \node[shape=circle,draw=orange,fill=orange,thick, minimum size=0.4cm] (BBAAAA) at (-0.1962,-5.9968) {};
    \node[shape=circle,draw=orange,fill=orange,thick, minimum size=0.4cm] (BBAAAB) at (-0.588,-5.9711) {};

    \node[shape=circle,draw=orange,fill=orange,thick, minimum size=0.4cm] (BBAABA) at (-0.9514,-5.9241) {};
    \node[shape=circle,draw=orange,fill=orange,thick, minimum size=0.4cm] (BBAABB) at (-1.3368,-5.8492) {};

    \node[shape=circle,draw=orange,fill=orange,thick, minimum size=0.4cm] (BBABAA) at (-1.7416,-5.7418) {};
    \node[shape=circle,draw=orange,fill=orange,thick, minimum size=0.4cm] (BBABAB) at (-2.1134,-5.6156) {};

    \node[shape=circle,draw=orange, fill=orange,thick, minimum size=0.4cm] (BBABBA) at (-2.4523,-5.4761) {};
    \node[shape=circle,draw=orange,fill=orange,thick, minimum size=0.4cm] (BBABBB) at (-2.8052,-5.304) {};

    \node[shape=circle,draw=orange,fill=orange,thick, minimum size=0.4cm] (BBBAAA) at (-3.1683,-5.0954) {};
    \node[shape=circle,draw=orange,fill=orange,thick, minimum size=0.4cm] (BBBAAB) at (-3.4948,-4.8773) {};

    \node[shape=circle,draw=orange,fill=orange,thick, minimum size=0.4cm] (BBBABA) at (-3.786,-4.6548) {};
    \node[shape=circle,draw=orange,fill=orange,thick, minimum size=0.4cm] (BBBABB) at (-4.0823,-4.3972) {};

    \node[shape=circle,draw=orange,fill=orange,thick, minimum size=0.4cm](BBBBAA) at (-4.3792,-4.1018) {};
    \node[shape=circle,draw=orange,fill=orange,thick, minimum size=0.4cm] (BBBBAB) at (-4.6381,-3.8066) {};

    \node[shape=circle,draw=orange,fill=orange,thick, minimum size=0.4cm] (BBBBBA) at (-4.8618,-3.5164) {};
    \node[shape=circle,draw=orange,fill=orange,thick, minimum size=0.4cm] (BBBBBB) at (-5.0814,-3.1909) {};

    \node[shape=circle,draw=orange,fill=orange,thick, minimum size=0.4cm] (CA) at (2,0) {};
    \node[shape=circle,draw=orange,fill=orange,thick, minimum size=0.4cm] (CB) at (1,1.732) {};

    \node[shape=circle,draw=orange,fill=orange,thick, minimum size=0.4cm] (CAA) at (2.8978,0.7765) {};
    \node[shape=circle,draw=orange,fill=orange,thick, minimum size=0.4cm] (CAB) at (2.8978,-0.7765) {};

    \node[shape=circle,draw=orange,fill=orange,thick, minimum size=0.4cm] (CBA) at (0.7765,2.8977) {};
    \node[shape=circle,draw=orange,fill=orange,thick, minimum size=0.4cm] (CBB) at (2.1213,2.1212) {};

    \node[shape=circle,draw=cyan, pattern=north west lines, pattern color=cyan,thick, minimum size=0.4cm] (CAAA) at (3.6955,1.5308) {};
    \node[shape=circle,draw=orange,fill=orange,thick, minimum size=0.4cm](CAAB) at (3.9658,0.5222) {};

    \node[shape=circle,draw=orange,fill=orange,thick, minimum size=0.4cm] (CABA) at (3.9658,-0.5222) {};
    \node[shape=circle,draw=orange,fill=orange,thick, minimum size=0.4cm] (CABB) at (3.6955,-1.5308) {};

    \node[shape=circle,draw=orange,fill=orange,thick, minimum size=0.4cm] (CBAA) at (0.5222,3.9657) {};
    \node[shape=circle,draw=orange,fill=orange,thick, minimum size=0.4cm] (CBAB) at (1.5308,3.6954) {};

    \node[shape=circle,draw=orange,fill=orange,thick, minimum size=0.4cm] (CBBA) at (2.435,3.1733) {};
    \node[shape=circle,draw=orange,fill=orange,thick, minimum size=0.4cm] (CBBB) at (3.1734,2.4349) {};

    \node[shape=circle,draw=orange,fill=orange,thick, minimum size=0.4cm] (CAAAA) at (4.4843,2.2115) {};
    \node[shape=circle,draw=orange,fill=orange,thick, minimum size=0.4cm](CAAAB) at (4.7276,1.6279) {};

    \node[shape=circle,draw=orange,fill=orange,thick, minimum size=0.4cm] (CAABA) at (4.9039,0.9756) {};
    \node[shape=circle,draw=orange,fill=orange,thick, minimum size=0.4cm] (CAABB) at (4.9879,0.3489) {};

    \node[shape=circle,draw=orange,fill=orange,thick, minimum size=0.4cm](CABAA) at (4.9893,-0.3271) {};
    \node[shape=circle,draw=orange, fill=orange,thick, minimum size=0.4cm](CABAB) at (4.9082,-0.9542) {};

    \node[shape=circle,draw=orange,fill=orange,thick, minimum size=0.4cm] (CABBA) at (4.7346,-1.6073) {};
    \node[shape=circle,draw=orange,fill=orange,thick, minimum size=0.4cm] (CABBB) at (4.4939,-2.1919) {};

    \node[shape=circle,draw=orange,fill=orange,thick, minimum size=0.4cm] (CBAAA) at (0.3271,4.9892) {};
    \node[shape=circle,draw=orange, fill=orange,thick, minimum size=0.4cm] (CBAAB) at (0.9542,4.908) {};

    \node[shape=circle,draw=orange,fill=orange,thick, minimum size=0.4cm] (CBABA) at (1.6073,4.7345) {};
    \node[shape=circle,draw=orange,fill=orange,thick, minimum size=0.4cm] (CBABB) at (2.1919,4.4938) {};

    \node[shape=circle,draw=orange,fill=orange,thick, minimum size=0.4cm] (CBBAA) at (2.7778,4.1572) {};
    \node[shape=circle,draw=orange,fill=orange,thick, minimum size=0.4cm] (CBBAB) at (3.2802,3.7734) {};

    \node[shape=circle,draw=orange, fill=orange,thick, minimum size=0.4cm] (CBBBA) at (3.7592,3.2965) {};
    \node[shape=circle,draw=orange,fill=orange,thick, minimum size=0.4cm] (CBBBB) at (4.1452,2.7958) {};

    \node[shape=circle,draw=orange,fill=orange,thick, minimum size=0.4cm] (CAAAAA) at (5.2914,2.8284) {};
    \node[shape=circle,draw=orange,fill=orange,thick, minimum size=0.4cm] (CAAAAB) at (5.4651,2.4763) {};

    \node[shape=circle,draw=orange,fill=orange,thick, minimum size=0.4cm](CAAABA) at (5.6062,2.1381) {};
    \node[shape=circle,draw=orange,fill=orange,thick, minimum size=0.4cm] (CAAABB) at (5.734,1.7668) {};

    \node[shape=circle,draw=orange, fill=orange,thick, minimum size=0.4cm] (CAABAA) at (5.8432,1.3626) {};
    \node[shape=circle,draw=orange,fill=orange,thick, minimum size=0.4cm] (CAABAB) at (5.9198,0.9776) {};

    \node[shape=circle,draw=orange,fill=orange,thick, minimum size=0.4cm] (CAABBA) at (5.9686,0.6143) {};
    \node[shape=circle,draw=orange,fill=orange,thick, minimum size=0.4cm] (CAABBB) at (5.996,0.2226) {};

    \node[shape=circle,draw=orange,fill=orange,thick, minimum size=0.4cm](CABAAA) at (5.9968,-0.1964) {};
    \node[shape=circle,draw=orange,fill=orange,thick, minimum size=0.4cm] (CABAAB) at (5.9711,-0.5882) {};

    \node[shape=circle,draw=orange,fill=orange,thick, minimum size=0.4cm] (CABABA) at (5.9242,-0.9517) {};
    \node[shape=circle,draw=orange,fill=orange,thick, minimum size=0.4cm](CABABB) at (5.8492,-1.3371) {};

    \node[shape=circle,draw=orange,fill=orange,thick, minimum size=0.4cm] (CABBAA) at (5.7416,-1.7418) {};
    \node[shape=circle,draw=cyan,pattern=north west lines, pattern color=cyan,thick, minimum size=0.4cm] (CABBAB) at (5.6154,-2.1136) {};

    \node[shape=circle,draw=orange,fill=orange,thick, minimum size=0.4cm] (CABBBA) at (5.4759,-2.4524) {};
    \node[shape=circle,draw=orange,fill=orange,thick, minimum size=0.4cm] (CABBBB) at (5.3037,-2.8053) {};

    \node[shape=circle,draw=orange,fill=orange,thick, minimum size=0.4cm] (CBAAAA) at (0.1964,5.9967) {};
    \node[shape=circle,draw=orange,fill=orange,thick, minimum size=0.4cm] (CBAAAB) at (0.5882,5.9710) {};

    \node[shape=circle,draw=orange,fill=orange,thick, minimum size=0.4cm] (CBAABA) at (0.9517,5.9239) {};
    \node[shape=circle,draw=orange,fill=orange,thick, minimum size=0.4cm] (CBAABB) at (1.3371,5.849) {};

    \node[shape=circle,draw=orange,fill=orange,thick, minimum size=0.4cm] (CBABAA) at (1.7418,5.7415) {};
    \node[shape=circle,draw=orange,fill=orange,thick, minimum size=0.4cm](CBABAB) at (2.1136,5.6153) {};

    \node[shape=circle,draw=orange, fill=orange,thick, minimum size=0.4cm] (CBABBA) at (2.4524,5.4757) {};
    \node[shape=circle,draw=orange,fill=orange,thick, minimum size=0.4cm] (CBABBB) at (2.8053,5.3036) {};

    \node[shape=circle,draw=orange,fill=orange,thick, minimum size=0.4cm] (CBBAAA) at (3.1684,5.0950) {};
    \node[shape=circle,draw=orange,fill=orange,thick, minimum size=0.4cm] (CBBAAB) at (3.4948,4.8769) {};

    \node[shape=circle,draw=cyan,pattern=north west lines, pattern color=cyan,thick, minimum size=0.4cm](CBBABA) at (3.786,4.6544) {};
    \node[shape=circle,draw=orange,fill=orange,thick, minimum size=0.4cm] (CBBABB) at (4.0823,4.3969) {};

    \node[shape=circle,draw=orange,fill=orange,thick, minimum size=0.4cm](CBBBAA) at (4.3792,4.1013) {};
    \node[shape=circle,draw=orange,fill=orange,thick, minimum size=0.4cm] (CBBBAB) at (4.6381,3.8061) {};

    \node[shape=circle,draw=orange,fill=orange,thick, minimum size=0.4cm] (CBBBBA) at (4.8618,3.5159) {};
    \node[shape=circle,draw=orange,fill=orange,thick, minimum size=0.4cm] (CBBBBB) at (5.0813,3.1904) {};

    \draw [->, very thick] (A) -- (B);
    \draw [->, very thick] (A) -- (C);

    \draw [->, very thick] (B) -- (BA);
    \draw [->, very thick] (B) -- (BB);

    \draw [->, very thick, draw=black] (BA) -- (BAA);
    \draw [->, very thick] (BA) -- (BAB);

    \draw [->, very thick] (BB) -- (BBA);
    \draw [->, very thick] (BB) -- (BBB);

    \draw [->, very thick] (BAA) -- (BAAA);
    \draw [->, very thick] (BAA) -- (BAAB);

    \draw [->, very thick] (BAB) -- (BABA);
    \draw [->, very thick] (BAB) -- (BABB);

    \draw [->, very thick] (BBA) -- (BBAA);
    \draw [->, very thick] (BBA) -- (BBAB);

    \draw [->, very thick, dashed] (BBB) -- (BBBA);
    \draw [->, very thick] (BBB) -- (BBBB);

    \draw [->, very thick] (BAAA) -- (BAAAA);
    \draw [->, very thick] (BAAA) -- (BAAAB);

    \draw [->, very thick, dashed] (BAAB) -- (BAABA);
    \draw [->, very thick] (BAAB) -- (BAABB);

    \draw [->, very thick] (BABA) -- (BABAA);
    \draw [->, very thick, draw=black] (BABA) -- (BABAB);

    \draw [->, very thick, draw=black] (BABB) -- (BABBA);
    \draw [->, very thick, draw=black] (BABB) -- (BABBB);

    \draw [->, very thick, draw=black, dashed] (BBAA) -- (BBAAA);
    \draw [->, very thick, draw=black] (BBAA) -- (BBAAB);

    \draw [->, very thick, draw=black] (BBAB) -- (BBABA);
    \draw [->, very thick, draw=black] (BBAB) -- (BBABB);

    \draw [->, very thick, draw=black] (BBBA) -- (BBBAA);
    \draw [->, very thick, draw=black] (BBBA) -- (BBBAB);

    \draw [->, very thick, draw=black] (BBBB) -- (BBBBA);
    \draw [->, very thick, draw=black] (BBBB) -- (BBBBB);

    \draw [->, very thick, dashed] (BAAAA) -- (BAAAAA);
    \draw [->, very thick] (BAAAA) -- (BAAAAB);

    \draw [->, very thick] (BAAAB) -- (BAAABA);
    \draw [->, very thick] (BAAAB) -- (BAAABB);

    \draw [->, very thick] (BAABA) -- (BAABAA);
    \draw [->, very thick] (BAABA) -- (BAABAB);

    \draw [->, very thick] (BAABB) -- (BAABBA);
    \draw [->, very thick] (BAABB) -- (BAABBB);

    \draw [->, very thick] (BABAA) -- (BABAAA);
    \draw [->, very thick, dashed] (BABAA) -- (BABAAB);

    \draw [->, very thick] (BABAB) -- (BABABA);
    \draw [->, very thick] (BABAB) -- (BABABB);

    \draw [->, very thick] (BABBA) -- (BABBAA);
    \draw [->, very thick] (BABBA) -- (BABBAB);

    \draw [->, very thick] (BABBB) -- (BABBBA);
    \draw [->, very thick] (BABBB) -- (BABBBB);

    \draw [->, very thick, dashed] (BBAAA) -- (BBAAAA);
    \draw [->, very thick, dashed] (BBAAA) -- (BBAAAB);

    \draw [->, very thick] (BBAAB) -- (BBAABA);
    \draw [->, very thick] (BBAAB) -- (BBAABB);

    \draw [->, very thick] (BBABA) -- (BBABAA);
    \draw [->, very thick] (BBABA) -- (BBABAB);

    \draw [->, very thick] (BBABB) -- (BBABBA);
    \draw [->, very thick] (BBABB) -- (BBABBB);

    \draw [->, very thick] (BBBAA) -- (BBBAAA);
    \draw [->, very thick] (BBBAA) -- (BBBAAB);

    \draw [->, very thick] (BBBAB) -- (BBBABA);
    \draw [->, very thick] (BBBAB) -- (BBBABB);

    \draw[->, very thick, dashed]  (BBBBA) -- (BBBBAA);
    \draw[->, very thick]  (BBBBA) -- (BBBBAB);

    \draw[->, very thick]  (BBBBB) -- (BBBBBA);
    \draw[->, very thick]  (BBBBB) -- (BBBBBB);

    \draw [->, very thick] (C) -- (CA);
    \draw [->, very thick] (C) -- (CB);

    \draw [->, very thick, draw=black] (CA) -- (CAA);
    \draw [->, very thick, draw=black] (CA) -- (CAB);

    \draw [->, very thick, dashed] (CB) -- (CBA);
    \draw [->, very thick] (CB) -- (CBB);

    \draw [->, very thick, dashed] (CAA) -- (CAAA);
    \draw [->, very thick,] (CAA) -- (CAAB);

    \draw [->, very thick] (CAB) -- (CABA);
    \draw [->, very thick] (CAB) -- (CABB);

    \draw [->, very thick] (CBA) -- (CBAA);
    \draw [->, very thick] (CBA) -- (CBAB);

    \draw [->, very thick] (CBB) -- (CBBA);
    \draw [->, very thick] (CBB) -- (CBBB);

    \draw [->, very thick, dashed] (CAAA) -- (CAAAA);
    \draw [->, very thick, dashed] (CAAA) -- (CAAAB);

    \draw [->, very thick] (CAAB) -- (CAABA);
    \draw [->, very thick] (CAAB) -- (CAABB);

    \draw [->, very thick] (CABA) -- (CABAA);
    \draw [->, very thick] (CABA) -- (CABAB);

    \draw [->, very thick] (CABB) -- (CABBA);
    \draw [->, very thick] (CABB) -- (CABBB);

    \draw [->, very thick, draw=black] (CBAA) -- (CBAAA);
    \draw [->, very thick, draw=black] (CBAA) -- (CBAAB);

    \draw [->, very thick, draw=black] (CBAB) -- (CBABA);
    \draw [->, very thick, draw=black] (CBAB) -- (CBABB);

    \draw [->, very thick, draw=black] (CBBA) -- (CBBAA);
    \draw [->, very thick, draw=black] (CBBA) -- (CBBAB);
    
    \draw [->, very thick] (CBBB) -- (CBBBA);
    \draw [->, very thick, draw=black] (CBBB) -- (CBBBB);

    \draw [->, very thick] (CAAAA) -- (CAAAAA);
    \draw [->, very thick] (CAAAA) -- (CAAAAB);

    \draw [->, very thick] (CAAAB) -- (CAAABA);
    \draw [->, very thick] (CAAAB) -- (CAAABB);

    \draw [->, very thick] (CAABA) -- (CAABAA);
    \draw [->, very thick] (CAABA) -- (CAABAB);

    \draw [->, very thick] (CAABB) -- (CAABBA);
    \draw [->, very thick, dashed] (CAABB) -- (CAABBB);

    \draw [->, very thick] (CABAA) -- (CABAAA);
    \draw [->, very thick] (CABAA) -- (CABAAB);
    
    \draw [->, very thick] (CABAB) -- (CABABA);
    \draw [->, very thick] (CABAB) -- (CABABB);

    \draw [->, very thick] (CABBA) -- (CABBAA);
    \draw [->, very thick, dashed] (CABBA) -- (CABBAB);

    \draw [->, very thick] (CABBB) -- (CABBBA);
    \draw [->, very thick] (CABBB) -- (CABBBB);

    \draw [->, very thick] (CBAAA) -- (CBAAAA);
    \draw [->, very thick] (CBAAA) -- (CBAAAB);

    \draw [->, very thick] (CBAAB) -- (CBAABA);
    \draw [->, very thick] (CBAAB) -- (CBAABB);

    \draw [->, very thick] (CBABA) -- (CBABAA);
    \draw [->, very thick] (CBABA) -- (CBABAB);

    \draw [->, very thick] (CBABB) -- (CBABBA);
    \draw [->, very thick] (CBABB) -- (CBABBB);

    \draw [->, very thick] (CBBAA) -- (CBBAAA);
    \draw [->, very thick] (CBBAA) -- (CBBAAB);

    \draw [->, very thick, dashed] (CBBAB) -- (CBBABA);
    \draw [->, very thick] (CBBAB) -- (CBBABB);
    
    \draw [->, very thick] (CBBBA) -- (CBBBAA);
    \draw [->, very thick] (CBBBA) -- (CBBBAB);

    \draw [->, very thick] (CBBBB) -- (CBBBBA);
    \draw [->, very thick] (CBBBB) -- (CBBBBB);

    \end{tikzpicture}
  \end{minipage}
  } 
  \hfill
  \scalebox{0.55}{
  \begin{minipage}[b]{0.8\textwidth}
    \begin{tikzpicture}[every label/.append style={scale=1.3},rotate=30]

\node[label=:{\Large ii)}] () at (7.8,0.4) {};

\node[label=:{{\Large\color{orange}$A$}}] () at (-0.5,2.5) {};

\node[label=:{{\Large\color{cyan}$A^c$}}] () at (-2.7,-1.3) {};

\node[label=:{{\large\color{black}$\rho$}}] () at (-1.5,0.5) {};

    \node [shape=circle,draw=orange, fill=orange, thick, minimum size=0.4cm] (B) at (0,-1) {};
    \node[shape=circle,draw=cyan, pattern=north west lines, pattern color=cyan,  thick, minimum size=0.4cm] (C) at (0.866,0.5) {};
    \node[shape=circle,draw=cyan, pattern=north west lines, pattern color=cyan, thick, minimum size=0.4cm] (A) at (-0.866,0.5) {};

    \node[shape=circle,draw=orange, fill=orange,thick, minimum size=0.4cm] (BA) at (1,-1.7321) {};
    \node[shape=circle,draw=orange, fill=orange, thick, minimum size=0.4cm] (BB) at (-1,-1.7321) {};

    \node[shape=circle,draw=orange, fill=orange, thick, minimum size=0.4cm] (BAA) at (2.1213,-2.1214) {};
    \node[shape=circle,draw=orange, fill=orange,  thick, minimum size=0.4cm] (BAB) at (0.7764,-2.8978) {};

    \node[shape=circle,draw=orange, fill=orange, thick, minimum size=0.4cm] (BBA) at (-0.7764,-2.8978) {};
    \node[shape=circle,draw=cyan, pattern=north west lines, pattern color=cyan,  thick, minimum size=0.4cm] (BBB) at (-2.1213,-2.1214) {};

    \node[shape=circle,draw=orange, fill=orange, thick, minimum size=0.4cm] (BAAA) at (3.1734,-2.4352) {};
    \node[shape=circle,draw=orange, fill=orange,thick, minimum size=0.4cm] (BAAB) at (2.4350,-3.1735) {};

    \node[shape=circle,draw=orange, fill=orange,thick, minimum size=0.4cm] (BABA) at (1.5307,-3.6956) {};
    \node[shape=circle,draw=orange, fill=orange,thick, minimum size=0.4cm] (BABB) at (0.5220,-3.9658) {};

    \node[shape=circle,draw=orange,fill=orange,thick, minimum size=0.4cm] (BBAA) at (-0.522,-3.9658) {};
    \node[shape=circle,draw=orange, fill=orange,thick, minimum size=0.4cm] (BBAB) at (-1.5307,-3.6956) {};

    \node[shape=circle,draw=orange,fill=orange,thick, minimum size=0.4cm] (BBBA) at (-2.435,-3.1735) {};
    \node[shape=circle,draw=orange,fill=orange,thick, minimum size=0.4cm] (BBBB) at (-3.1734,-2.4352) {};

    \node[shape=circle,draw=orange,fill=orange,thick, minimum size=0.4cm] (BAAAA) at (4.1573,-2.778) {};
    \node[shape=circle,draw=cyan,pattern=north west lines, pattern color=cyan,thick, minimum size=0.4cm] (BAAAB) at (3.7735,-3.2805) {};

    \node[shape=circle,draw=orange,fill=orange,thick, minimum size=0.4cm] (BAABA) at (3.2967,-3.7593) {};
    \node[shape=circle,draw=orange,fill=orange,thick, minimum size=0.4cm](BAABB) at (2.7959,-4.1453) {};

    \node[shape=circle,draw=orange,fill=orange,thick, minimum size=0.4cm] (BABAA) at (2.2114,-4.4845) {};
    \node[shape=circle,draw=orange,fill=orange,thick, minimum size=0.4cm] (BABAB) at (1.6278,-4.7277) {};

    \node[shape=circle,draw=orange,fill=orange,thick, minimum size=0.4cm] (BABBA) at (0.9753,-4.9040) {};
    \node[shape=circle,draw=orange,fill=orange,thick, minimum size=0.4cm] (BABBB) at (0.3487,-4.9878) {};

    \node[shape=circle,draw=orange,fill=orange,thick, minimum size=0.4cm] (BBAAA) at (-0.3269,-4.9893) {};
    \node[shape=circle,draw=orange,fill=orange,thick, minimum size=0.4cm] (BBAAB) at (-0.9539,-4.9082) {};

    \node[shape=circle,draw=orange,fill=orange,thick, minimum size=0.4cm] (BBABA) at (-1.6071,-4.7348) {};
    \node[shape=circle,draw=orange,fill=orange,thick, minimum size=0.4cm] (BBABB) at (-2.1918,-4.4941) {};

    \node[shape=circle,draw=orange,fill=orange,thick, minimum size=0.4cm] (BBBAA) at (-2.7778,-4.1575) {};
    \node[shape=circle,draw=orange,fill=orange,thick, minimum size=0.4cm] (BBBAB) at (-3.2802,-3.7737) {};

    \node[shape=circle,draw=orange,fill=orange,thick, minimum size=0.4cm] (BBBBA) at (-3.7592,-3.2969) {};
    \node[shape=circle,draw=orange,fill=orange,thick, minimum size=0.4cm] (BBBBB) at (-4.1452,-2.7962) {};

    \node[shape=circle,draw=orange,fill=orange,thick, minimum size=0.4cm] (BAAAAA) at (5.0952,-3.1686) {};
    \node[shape=circle,draw=orange,fill=orange,thick, minimum size=0.4cm] (BAAAAB) at (4.877,-3.495) {};

    \node[shape=circle,draw=orange,fill=orange,thick, minimum size=0.4cm] (BAAABA) at (4.6546,-3.7863) {};
    \node[shape=circle,draw=orange,fill=orange,thick, minimum size=0.4cm](BAAABB) at (4.3970,-4.0827) {};

    \node[shape=circle,draw=orange,fill=orange,thick, minimum size=0.4cm](BAABAA) at (4.1015,-4.3793) {};
    \node[shape=circle,draw=orange,fill=orange,thick, minimum size=0.4cm] (BAABAB) at (3.8063,-4.6382) {};

    \node[shape=circle,draw=cyan,pattern=north west lines, pattern color=cyan,thick, minimum size=0.4cm] (BAABBA) at (3.516,-4.8619) {};
    \node[shape=circle,draw=orange,fill=orange,thick, minimum size=0.4cm] (BAABBB) at (3.1905,-5.0815) {};

    \node[shape=circle,draw=orange,fill=orange,thick, minimum size=0.4cm]  (BABAAA) at (2.8283,-5.2917) {};
    \node[shape=circle,draw=orange,fill=orange,thick, minimum size=0.4cm]  (BABAAB) at (2.4762,-5.4653) {};

    \node[shape=circle,draw=orange,fill=orange,thick, minimum size=0.4cm]  (BABABA) at (2.1379,-5.6063) {};
    \node[shape=circle,draw=orange,fill=orange,thick, minimum size=0.4cm]  (BABABB) at (1.7667,-5.7341) {};

    \node[shape=circle,draw=orange,fill=orange,thick, minimum size=0.4cm]  (BABBAA) at (1.3623,-5.8434) {};
    \node[shape=circle,draw=orange,fill=orange,thick, minimum size=0.4cm]  (BABBAB) at (0.9772,-5.9199) {};

    \node[shape=circle,draw=orange,fill=orange,thick, minimum size=0.4cm] (BABBBA) at (0.6141,-5.9685) {};
    \node[shape=circle,draw=orange,fill=orange,thick, minimum size=0.4cm] (BABBBB) at (0.2224,-5.9958) {};

    \node[shape=circle,draw=orange,fill=orange,thick, minimum size=0.4cm] (BBAAAA) at (-0.1962,-5.9968) {};
    \node[shape=circle,draw=orange,fill=orange,thick, minimum size=0.4cm] (BBAAAB) at (-0.588,-5.9711) {};

    \node[shape=circle,draw=orange,fill=orange,thick, minimum size=0.4cm] (BBAABA) at (-0.9514,-5.9241) {};
    \node[shape=circle,draw=orange,fill=orange,thick, minimum size=0.4cm] (BBAABB) at (-1.3368,-5.8492) {};

    \node[shape=circle,draw=orange,fill=orange,thick, minimum size=0.4cm] (BBABAA) at (-1.7416,-5.7418) {};
    \node[shape=circle,draw=orange,fill=orange,thick, minimum size=0.4cm] (BBABAB) at (-2.1134,-5.6156) {};

    \node[shape=circle,draw=cyan,pattern=north west lines, pattern color=cyan,thick, minimum size=0.4cm] (BBABBA) at (-2.4523,-5.4761) {};
    \node[shape=circle,draw=orange,fill=orange,thick, minimum size=0.4cm] (BBABBB) at (-2.8052,-5.304) {};

    \node[shape=circle,draw=orange,fill=orange,thick, minimum size=0.4cm] (BBBAAA) at (-3.1683,-5.0954) {};
    \node[shape=circle,draw=orange,fill=orange,thick, minimum size=0.4cm] (BBBAAB) at (-3.4948,-4.8773) {};

    \node[shape=circle,draw=orange,fill=orange,thick, minimum size=0.4cm] (BBBABA) at (-3.786,-4.6548) {};
    \node[shape=circle,draw=orange,fill=orange,thick, minimum size=0.4cm] (BBBABB) at (-4.0823,-4.3972) {};

    \node[shape=circle,draw=orange,fill=orange,thick, minimum size=0.4cm](BBBBAA) at (-4.3792,-4.1018) {};
    \node[shape=circle,draw=orange,fill=orange,thick, minimum size=0.4cm] (BBBBAB) at (-4.6381,-3.8066) {};

    \node[shape=circle,draw=orange,fill=orange,thick, minimum size=0.4cm] (BBBBBA) at (-4.8618,-3.5164) {};
    \node[shape=circle,draw=orange,fill=orange,thick, minimum size=0.4cm] (BBBBBB) at (-5.0814,-3.1909) {};

    \node[shape=circle,draw=orange,fill=orange,thick, minimum size=0.4cm] (CA) at (2,0) {};
    \node[shape=circle,draw=orange,fill=orange,thick, minimum size=0.4cm] (CB) at (1,1.732) {};

    \node[shape=circle,draw=orange,fill=orange,thick, minimum size=0.4cm] (CAA) at (2.8978,0.7765) {};
    \node[shape=circle,draw=orange,fill=orange,thick, minimum size=0.4cm] (CAB) at (2.8978,-0.7765) {};

    \node[shape=circle,draw=orange,fill=orange,thick, minimum size=0.4cm] (CBA) at (0.7765,2.8977) {};
    \node[shape=circle,draw=orange,fill=orange,thick, minimum size=0.4cm] (CBB) at (2.1213,2.1212) {};

    \node[shape=circle,draw=orange,fill=orange,thick, minimum size=0.4cm] (CAAA) at (3.6955,1.5308) {};
    \node[shape=circle,draw=orange,fill=orange,thick, minimum size=0.4cm](CAAB) at (3.9658,0.5222) {};

    \node[shape=circle,draw=orange,fill=orange,thick, minimum size=0.4cm] (CABA) at (3.9658,-0.5222) {};
    \node[shape=circle,draw=orange,fill=orange,thick, minimum size=0.4cm] (CABB) at (3.6955,-1.5308) {};

    \node[shape=circle,draw=orange,fill=orange,thick, minimum size=0.4cm] (CBAA) at (0.5222,3.9657) {};
    \node[shape=circle,draw=orange,fill=orange,thick, minimum size=0.4cm] (CBAB) at (1.5308,3.6954) {};

    \node[shape=circle,draw=orange,fill=orange,thick, minimum size=0.4cm] (CBBA) at (2.435,3.1733) {};
    \node[shape=circle,draw=orange,fill=orange,thick, minimum size=0.4cm] (CBBB) at (3.1734,2.4349) {};

    \node[shape=circle,draw=orange,fill=orange,thick, minimum size=0.4cm] (CAAAA) at (4.4843,2.2115) {};
    \node[shape=circle,draw=orange,fill=orange,thick, minimum size=0.4cm](CAAAB) at (4.7276,1.6279) {};

    \node[shape=circle,draw=orange,fill=orange,thick, minimum size=0.4cm] (CAABA) at (4.9039,0.9756) {};
    \node[shape=circle,draw=orange,fill=orange,thick, minimum size=0.4cm] (CAABB) at (4.9879,0.3489) {};

    \node[shape=circle,draw=orange,fill=orange,thick, minimum size=0.4cm](CABAA) at (4.9893,-0.3271) {};
    \node[shape=circle,draw=orange,fill=orange,thick, minimum size=0.4cm](CABAB) at (4.9082,-0.9542) {};

    \node[shape=circle,draw=orange,fill=orange,thick, minimum size=0.4cm] (CABBA) at (4.7346,-1.6073) {};
    \node[shape=circle,draw=orange,fill=orange,thick, minimum size=0.4cm] (CABBB) at (4.4939,-2.1919) {};

    \node[shape=circle,draw=orange,fill=orange,thick, minimum size=0.4cm] (CBAAA) at (0.3271,4.9892) {};
    \node[shape=circle,draw=orange,fill=orange,thick, minimum size=0.4cm] (CBAAB) at (0.9542,4.908) {};

    \node[shape=circle,draw=orange,fill=orange,thick, minimum size=0.4cm] (CBABA) at (1.6073,4.7345) {};
    \node[shape=circle,draw=orange,fill=orange,thick, minimum size=0.4cm] (CBABB) at (2.1919,4.4938) {};

    \node[shape=circle,draw=orange,fill=orange,thick, minimum size=0.4cm] (CBBAA) at (2.7778,4.1572) {};
    \node[shape=circle,draw=orange,fill=orange,thick, minimum size=0.4cm] (CBBAB) at (3.2802,3.7734) {};

    \node[shape=circle,draw=cyan,pattern=north west lines, pattern color=cyan,thick, minimum size=0.4cm] (CBBBA) at (3.7592,3.2965) {};
    \node[shape=circle,draw=orange,fill=orange,thick, minimum size=0.4cm] (CBBBB) at (4.1452,2.7958) {};

    \node[shape=circle,draw=orange,fill=orange,thick, minimum size=0.4cm] (CAAAAA) at (5.2914,2.8284) {};
    \node[shape=circle,draw=orange,fill=orange,thick, minimum size=0.4cm] (CAAAAB) at (5.4651,2.4763) {};

    \node[shape=circle,draw=orange,fill=orange,thick, minimum size=0.4cm](CAAABA) at (5.6062,2.1381) {};
    \node[shape=circle,draw=orange,fill=orange,thick, minimum size=0.4cm] (CAAABB) at (5.734,1.7668) {};

    \node[shape=circle,draw=cyan,pattern=north west lines, pattern color=cyan,thick, minimum size=0.4cm] (CAABAA) at (5.8432,1.3626) {};
    \node[shape=circle,draw=orange,fill=orange,thick, minimum size=0.4cm] (CAABAB) at (5.9198,0.9776) {};

    \node[shape=circle,draw=orange,fill=orange,thick, minimum size=0.4cm] (CAABBA) at (5.9686,0.6143) {};
    \node[shape=circle,draw=orange,fill=orange,thick, minimum size=0.4cm] (CAABBB) at (5.996,0.2226) {};

    \node[shape=circle,draw=orange,fill=orange,thick, minimum size=0.4cm](CABAAA) at (5.9968,-0.1964) {};
    \node[shape=circle,draw=orange,fill=orange,thick, minimum size=0.4cm] (CABAAB) at (5.9711,-0.5882) {};

    \node[shape=circle,draw=orange,fill=orange,thick, minimum size=0.4cm] (CABABA) at (5.9242,-0.9517) {};
    \node[shape=circle,draw=orange,fill=orange,thick, minimum size=0.4cm](CABABB) at (5.8492,-1.3371) {};

    \node[shape=circle,draw=orange,fill=orange,thick, minimum size=0.4cm] (CABBAA) at (5.7416,-1.7418) {};
    \node[shape=circle,draw=orange,fill=orange,thick, minimum size=0.4cm] (CABBAB) at (5.6154,-2.1136) {};

    \node[shape=circle,draw=orange,fill=orange,thick, minimum size=0.4cm] (CABBBA) at (5.4759,-2.4524) {};
    \node[shape=circle,draw=orange,fill=orange,thick, minimum size=0.4cm] (CABBBB) at (5.3037,-2.8053) {};

    \node[shape=circle,draw=orange,fill=orange,thick, minimum size=0.4cm] (CBAAAA) at (0.1964,5.9967) {};
    \node[shape=circle,draw=orange,fill=orange,thick, minimum size=0.4cm] (CBAAAB) at (0.5882,5.9710) {};

    \node[shape=circle,draw=orange,fill=orange,thick, minimum size=0.4cm] (CBAABA) at (0.9517,5.9239) {};
    \node[shape=circle,draw=orange,fill=orange,thick, minimum size=0.4cm] (CBAABB) at (1.3371,5.849) {};

    \node[shape=circle,draw=orange,fill=orange,thick, minimum size=0.4cm] (CBABAA) at (1.7418,5.7415) {};
    \node[shape=circle,draw=orange,fill=orange,thick, minimum size=0.4cm](CBABAB) at (2.1136,5.6153) {};

    \node[shape=circle,draw=cyan,pattern=north west lines, pattern color=cyan,thick, minimum size=0.4cm] (CBABBA) at (2.4524,5.4757) {};
    \node[shape=circle,draw=orange,fill=orange,thick, minimum size=0.4cm] (CBABBB) at (2.8053,5.3036) {};

    \node[shape=circle,draw=orange,fill=orange,thick, minimum size=0.4cm] (CBBAAA) at (3.1684,5.0950) {};
    \node[shape=circle,draw=orange,fill=orange,thick, minimum size=0.4cm] (CBBAAB) at (3.4948,4.8769) {};

    \node[shape=circle,draw=orange,fill=orange,thick, minimum size=0.4cm] (CBBABA) at (3.786,4.6544) {};
    \node[shape=circle,draw=orange,fill=orange,thick, minimum size=0.4cm] (CBBABB) at (4.0823,4.3969) {};

    \node[shape=circle,draw=orange,fill=orange,thick, minimum size=0.4cm](CBBBAA) at (4.3792,4.1013) {};
    \node[shape=circle,draw=orange,fill=orange,thick, minimum size=0.4cm] (CBBBAB) at (4.6381,3.8061) {};

    \node[shape=circle,draw=orange,fill=orange,thick, minimum size=0.4cm] (CBBBBA) at (4.8618,3.5159) {};
    \node[shape=circle,draw=orange,fill=orange,thick, minimum size=0.4cm] (CBBBBB) at (5.0813,3.1904) {};

    \draw [->, very thick, dashed] (A) -- (B);
    \draw [->, very thick, dashed] (A) -- (C);

    \draw [->, very thick] (B) -- (BA);
    \draw [->, very thick] (B) -- (BB);

    \draw [->, very thick, draw=black] (BA) -- (BAA);
    \draw [->, very thick] (BA) -- (BAB);

    \draw [->, very thick] (BB) -- (BBA);
    \draw [->, very thick, dashed] (BB) -- (BBB);

    \draw [->, very thick] (BAA) -- (BAAA);
    \draw [->, very thick] (BAA) -- (BAAB);

    \draw [->, very thick] (BAB) -- (BABA);
    \draw [->, very thick] (BAB) -- (BABB);

    \draw [->, very thick] (BBA) -- (BBAA);
    \draw [->, very thick] (BBA) -- (BBAB);

    \draw [->, very thick, dashed] (BBB) -- (BBBA);
    \draw [->, very thick, dashed] (BBB) -- (BBBB);

    \draw [->, very thick] (BAAA) -- (BAAAA);
    \draw [->, very thick, dashed] (BAAA) -- (BAAAB);

    \draw [->, very thick] (BAAB) -- (BAABA);
    \draw [->, very thick] (BAAB) -- (BAABB);

    \draw [->, very thick, dashed] (BABA) -- (BABAA);
    \draw [->, very thick, draw=black] (BABA) -- (BABAB);

    \draw [->, very thick, draw=black] (BABB) -- (BABBA);
    \draw [->, very thick, draw=black] (BABB) -- (BABBB);

    \draw [->, very thick, draw=black] (BBAA) -- (BBAAA);
    \draw [->, very thick, draw=black] (BBAA) -- (BBAAB);

    \draw [->, very thick, draw=black] (BBAB) -- (BBABA);
    \draw [->, very thick, draw=black] (BBAB) -- (BBABB);

    \draw [->, very thick, draw=black] (BBBA) -- (BBBAA);
    \draw [->, very thick, draw=black] (BBBA) -- (BBBAB);

    \draw [->, very thick, draw=black] (BBBB) -- (BBBBA);
    \draw [->, very thick, draw=black] (BBBB) -- (BBBBB);

    \draw [->, very thick] (BAAAA) -- (BAAAAA);
    \draw [->, very thick] (BAAAA) -- (BAAAAB);

    \draw [->, very thick, dashed] (BAAAB) -- (BAAABA);
    \draw [->, very thick, dashed] (BAAAB) -- (BAAABB);

    \draw [->, very thick] (BAABA) -- (BAABAA);
    \draw [->, very thick] (BAABA) -- (BAABAB);

    \draw [->, very thick, dashed] (BAABB) -- (BAABBA);
    \draw [->, very thick] (BAABB) -- (BAABBB);

    \draw [->, very thick] (BABAA) -- (BABAAA);
    \draw [->, very thick] (BABAA) -- (BABAAB);

    \draw [->, very thick] (BABAB) -- (BABABA);
    \draw [->, very thick] (BABAB) -- (BABABB);

    \draw [->, very thick] (BABBA) -- (BABBAA);
    \draw [->, very thick] (BABBA) -- (BABBAB);

    \draw [->, very thick, dashed] (BABBB) -- (BABBBA);
    \draw [->, very thick] (BABBB) -- (BABBBB);

    \draw [->, very thick] (BBAAA) -- (BBAAAA);
    \draw [->, very thick] (BBAAA) -- (BBAAAB);

    \draw [->, very thick] (BBAAB) -- (BBAABA);
    \draw [->, very thick] (BBAAB) -- (BBAABB);

    \draw [->, very thick] (BBABA) -- (BBABAA);
    \draw [->, very thick] (BBABA) -- (BBABAB);

    \draw [->, very thick,dashed] (BBABB) -- (BBABBA);
    \draw [->, very thick] (BBABB) -- (BBABBB);

    \draw [->, very thick] (BBBAA) -- (BBBAAA);
    \draw [->, very thick] (BBBAA) -- (BBBAAB);

    \draw [->, very thick] (BBBAB) -- (BBBABA);
    \draw [->, very thick] (BBBAB) -- (BBBABB);

    \draw[->, very thick]  (BBBBA) -- (BBBBAA);
    \draw[->, very thick]  (BBBBA) -- (BBBBAB);

    \draw[->, very thick]  (BBBBB) -- (BBBBBA);
    \draw[->, very thick]  (BBBBB) -- (BBBBBB);

    \draw [->, very thick,dashed] (C) -- (CA);
    \draw [->, very thick, dashed] (C) -- (CB);

    \draw [->, very thick, draw=black] (CA) -- (CAA);
    \draw [->, very thick, draw=black] (CA) -- (CAB);

    \draw [->, very thick] (CB) -- (CBA);
    \draw [->, very thick] (CB) -- (CBB);

    \draw [->, very thick] (CAA) -- (CAAA);
    \draw [->, very thick,] (CAA) -- (CAAB);

    \draw [->, very thick] (CAB) -- (CABA);
    \draw [->, very thick] (CAB) -- (CABB);

    \draw [->, very thick] (CBA) -- (CBAA);
    \draw [->, very thick] (CBA) -- (CBAB);

    \draw [->, very thick] (CBB) -- (CBBA);
    \draw [->, very thick] (CBB) -- (CBBB);

    \draw [->, very thick] (CAAA) -- (CAAAA);
    \draw [->, very thick] (CAAA) -- (CAAAB);

    \draw [->, very thick] (CAAB) -- (CAABA);
    \draw [->, very thick] (CAAB) -- (CAABB);

    \draw [->, very thick] (CABA) -- (CABAA);
    \draw [->, very thick] (CABA) -- (CABAB);

    \draw [->, very thick] (CABB) -- (CABBA);
    \draw [->, very thick] (CABB) -- (CABBB);

    \draw [->, very thick, draw=black] (CBAA) -- (CBAAA);
    \draw [->, very thick, draw=black, dashed] (CBAA) -- (CBAAB);

    \draw [->, very thick, draw=black] (CBAB) -- (CBABA);
    \draw [->, very thick, draw=black] (CBAB) -- (CBABB);

    \draw [->, very thick, draw=black] (CBBA) -- (CBBAA);
    \draw [->, very thick, draw=black] (CBBA) -- (CBBAB);
    
    \draw [->, very thick, dashed] (CBBB) -- (CBBBA);
    \draw [->, very thick, draw=black] (CBBB) -- (CBBBB);

    \draw [->, very thick] (CAAAA) -- (CAAAAA);
    \draw [->, very thick] (CAAAA) -- (CAAAAB);

    \draw [->, very thick] (CAAAB) -- (CAAABA);
    \draw [->, very thick] (CAAAB) -- (CAAABB);

    \draw [->, very thick, dashed] (CAABA) -- (CAABAA);
    \draw [->, very thick] (CAABA) -- (CAABAB);

    \draw [->, very thick] (CAABB) -- (CAABBA);
    \draw [->, very thick] (CAABB) -- (CAABBB);

    \draw [->, very thick] (CABAA) -- (CABAAA);
    \draw [->, very thick] (CABAA) -- (CABAAB);
    
    \draw [->, very thick] (CABAB) -- (CABABA);
    \draw [->, very thick, dashed] (CABAB) -- (CABABB);

    \draw [->, very thick] (CABBA) -- (CABBAA);
    \draw [->, very thick] (CABBA) -- (CABBAB);

    \draw [->, very thick] (CABBB) -- (CABBBA);
    \draw [->, very thick] (CABBB) -- (CABBBB);

    \draw [->, very thick] (CBAAA) -- (CBAAAA);
    \draw [->, very thick] (CBAAA) -- (CBAAAB);

    \draw [->, very thick] (CBAAB) -- (CBAABA);
    \draw [->, very thick] (CBAAB) -- (CBAABB);

    \draw [->, very thick] (CBABA) -- (CBABAA);
    \draw [->, very thick] (CBABA) -- (CBABAB);

    \draw [->, very thick,dashed] (CBABB) -- (CBABBA);
    \draw [->, very thick] (CBABB) -- (CBABBB);

    \draw [->, very thick] (CBBAA) -- (CBBAAA);
    \draw [->, very thick] (CBBAA) -- (CBBAAB);

    \draw [->, very thick] (CBBAB) -- (CBBABA);
    \draw [->, very thick] (CBBAB) -- (CBBABB);
    
    \draw [->, very thick,dashed] (CBBBA) -- (CBBBAA);
    \draw [->, very thick,dashed] (CBBBA) -- (CBBBAB);

    \draw [->, very thick] (CBBBB) -- (CBBBBA);
    \draw [->, very thick] (CBBBB) -- (CBBBBB);

    \end{tikzpicture}
    
  \end{minipage}
  }
  \caption{Broadcasting to infinity from the root $\rho$, in a coarse-grained description.  
    In picture i), 
    we are sending a ($\mu_A$-typical) signal $\{\sigma_\rho=a\}$ where $a\in A$ under the measure $\mu_A(\cdot|\sigma_\rho=a)$ to infinity. The dashed lines correspond to the $A$-irregular edges and the bold lines correspond to $A$-regular edges. Furthermore, the spins taking values in the localization set $A$ are completely colored in orange and those who take values in the complement $A^c$ are hatched and colored in blue. $A$-regular edges have high density under $\mu_A$ in the sense of Lemma \ref{lem: Key Lemma}.\\
    In picture ii), we are sending a ($\mu_A$-untypical) signal $\{\sigma_\rho=a\}$ where $a\in A^c$ under the measure $\mu_A(\cdot|\sigma_\rho=a)$ to infinity. The initial configuration will most likely fall into the localization set $A$, namely with probability $P_A(a,A)$, however with possibly different target states, from which it will continue to the outside as in picture i)}
    \label{fig: broadcasting on the tree}
\end{figure}
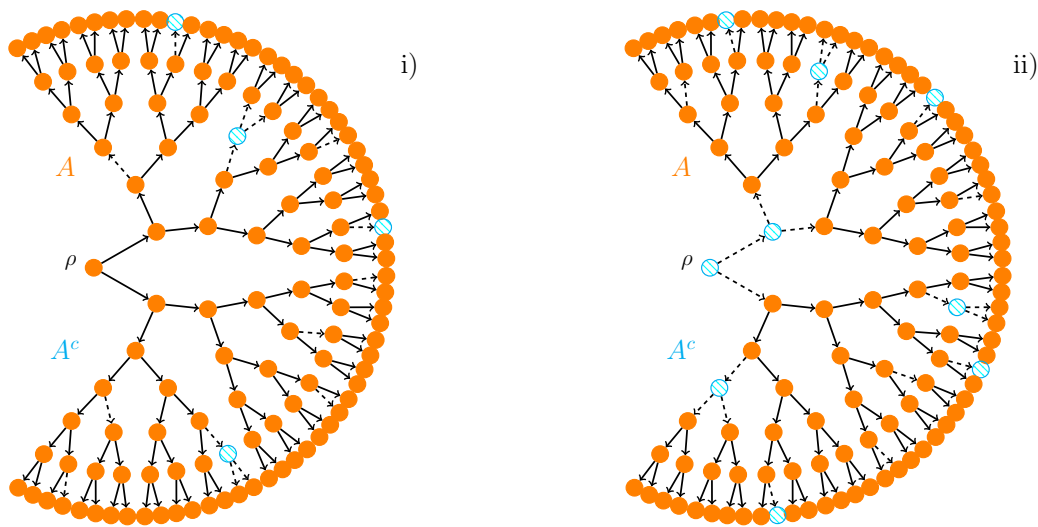

The proof of this theorem is a direct consequence of Lemma \ref{lem: Key Lemma}.

\begin{proof}
Let us define
\begin{equation*}
    D:=\{\omega\in \Omega:~\pi(\sigma_\rho=\omega_\rho|\omega)< 1-\epsilon_1(\beta)\}.
\end{equation*}
    If $\omega\in (B_{\rho,A})^c$, we have
    \begin{equation*}
        \pi(\sigma_\rho= \omega_\rho|\omega)\geq 1-\epsilon_1(\beta)
    \end{equation*}
    for $\mu_A$-almost every $\omega$ due to Lemma \ref{lem: Key Lemma} and hence $(B_{\rho,A})^c\subset D^c$. Consequently,
    \begin{equation*}
        \mu_A(D|\sigma_\rho=a)\leq \mu_A(B_{\rho,A}|\sigma_\rho=a)\leq \frac{\mu_A(B_{\rho,A})}{\mu_A(\sigma_\rho=a)}
    \end{equation*}
    and hence \eqref{eq: reconstruction bound inequality} after using part $b)$ of Lemma \ref{lem: Key Lemma} and Inequality (iv) of Proposition \ref{prop: Bounds A-loc. states} together with $(1-(C_1+C_2)e^{-\beta u})>\frac{1}{2}$.
\end{proof}

In order to show that $\mu_A$ is not extremal, we will consider the Edwards-Anderson parameter which is defined as
\begin{equation}\label{def: Edwards-Anderson parameter}
    q^{\mu_A}_{\text{EA}}:=\frac{1}{q}\sum_{a\in \Z_q}\text{Var}_{\mu_A}\left(\pi(\sigma_0=a|\cdot)\right).
\end{equation}
This quantity measures the degree of randomness of the spin magnetisation at the origin under the influence of the boundary condition at infinity and is thus a good quantity to measure the non-extremality of $\mu_A$.

\begin{theorem}[Non-extremality of $\mu_A$]\label{thm: Non-extremality of mu_A}

Consider a $\Z_q$-valued nearest-neighbor clock model with Hamiltonian given in \eqref{def: Hamiltonian} and a potential defined in \eqref{eq: Potential ferr n.n. model} fulfilling the $u,U,d$-bounds in \eqref{eq: u,U-bounds}. Furthermore, let $A\subsetneq \Z_q$ with $|A|\geq 2$ and $\beta$ be large enough. Then, the Edwards-Anderson parameter, defined in \eqref{def: Edwards-Anderson parameter}, corresponding to the $A$-localized state $\mu_A$ for this model, see Proposition \ref{prop: Bounds A-loc. states}, is bounded as follows 
    \begin{equation}\label{eq: Bound Edward Anderson}
        q^{\mu_A}_{\text{EA}}\geq \frac{1}{2q}\bigg((1-\epsilon_1(\beta))^2(1-2|A|\epsilon_2(\beta))-\frac{|A|-1}{|A|}\bigg)
    \end{equation}
    where $\epsilon_1(\beta)\downarrow 0$ and $\epsilon_2(\beta)\downarrow 0$ as $\beta\uparrow \infty$ and the quantities are given in Lemma \ref{lem: Key Lemma}. In particular, we obtain 
    \begin{equation*}
        q^{\mu_A}_{\text{EA}}>0
    \end{equation*}
    for large enough $\beta$ and thus the non-extremality of $\mu_A$.
\end{theorem}

\begin{proof}
    First of all, note that by dropping spin-values outside of $A$, we obtain
    \begin{equation*}
        q^{\mu_A}_{\text{EA}} \geq \frac{1}{q}\sum_{a\in A}\Big(\mu_A\left(\pi(\sigma_0=a|\cdot)^2\right)-\mu_A(\sigma_0=a)^2\Big).
    \end{equation*}
    Rewriting this expression leads to 
    \begin{align}\label{eq: Bound Proof EA Parameter}
        \nonumber&\frac{1}{q}\sum_{a\in A}\left(\sum_{b\in \Z_q} \mu_A(\sigma_0=b)\mu_A\left(\pi(\sigma_0=a|\cdot)^2\big| \sigma_0=b\right)-\mu_A(\sigma_0=a)^2\right)\\
        &\geq \frac{1}{q}\sum_{a\in A} \left(\mu_A(\sigma_0=a)\mu_A\left(\pi(\sigma_0=a|\cdot)^2\big|\sigma_0=a\right)-\mu_A(\sigma_0=a)^2\right)
    \end{align}
    where we kept the term for $b=a$ in the lower bound. In order to use the reconstruction bound \eqref{eq: reconstruction bound inequality}, we restrict our consideration on the case where $\pi(\sigma_0=a|\cdot)\geq 1-\epsilon_1(\beta)$. In this way, we obtain the following lower bound for \eqref{eq: Bound Proof EA Parameter}
    \begin{equation*}
        \frac{1}{q}\sum_{a\in A}\bigg(\mu_A(\sigma_0=a)\mu_A\big(\pi(\sigma_0=a|\cdot)^2 \mathds{1}_{\pi(\sigma_0=a|\cdot)\geq 1-\epsilon_1(\beta)}|\sigma_0=a\big)-\mu_A(\sigma_0=a)^2\bigg).
    \end{equation*}
    By applying the reconstruction bound \eqref{eq: reconstruction bound inequality} and comparing one factor of the negative terms in the sum with its maximum over $A$, this term is lower bounded by 
    \begin{equation*}
        \bigg((1-\epsilon_1(\beta))^2(1-2|A|\epsilon_2(\beta))-\max_{a\in A} \mu_A(\sigma_0=a)\bigg)\frac{1}{q}\sum_{a\in A}\mu_A(\sigma_0=a).
    \end{equation*}
    Choosing $\beta$ large enough such that $(1-C_2e^{-\beta u})^{-1}<(|A|-1)$ and $(1-(C_1+C_2)e^{-\beta u})> \frac{1}{2}$ in combination with the bounds for the single-site marginal in Proposition \ref{prop: Bounds A-loc. states}, results in the desired bound given in \eqref{eq: Bound Edward Anderson}.

    Recall that this implies the non-extremality of $\mu_A$: From the fact that $q_{EA}^{\mu_A}>0$, it follows that there exists at least one $a\in \Z_q$ such that $\text{Var}_{\mu_A}\left(\pi(\sigma_\rho=a|\cdot)\right)>0$ and hence $\pi(\sigma_\rho=a|\cdot)$ is not constant $\mu_A$-almost surely. However, $\pi(\sigma_\rho=a|\cdot)$ is $\mathscr{F}_\infty$-measurable and thus $\mu_A$ can not be an extremal Gibbs measure. 
\end{proof}

\subsection{Multi-site reconstruction}\label{subsec: Multi-site reconstruction}

Our goal is to prove Theorem \ref{thm: A.s. singularity extremals}, which states that the two extremals $\pi(\cdot|\omega)$ and $\pi(\cdot|\omega')$ are singular for independently chosen $\mu_A$-typical boundary conditions $\omega$ and $\omega'$ at infinity. The idea in \cite{CoKuLe24} to show a respective statement in their context was to consider a random variable which is able to distinguish typical extremals. In more detail, they considered the variable $\underline{\phi}^\omega$ which counts the matches of a realization $\sigma$ with $\omega$ along thinned enough branches of the tree, see \eqref{eq: thinned branch overlap}. The reason why this variable works also for our purposes is that Lemma \ref{lem: Key Lemma} does not only imply reconstruction on single vertices but also on these thinned branches, see Theorem \ref{thm: Branch overlap}. More explicitly, the expectation $\pi(\underline{\phi}^\omega|\omega)$ will be large for $\mu_A$-typical boundary conditions $\omega$ at infinity. This will be discussed in the first part of this subsection and the second part deals then with its application to prove $\pi(\underline{\phi}^\omega|\omega)>\pi(\underline{\phi}^{\omega}|\omega')$ in Corollary \ref{cor: distinct extremals}. This implies directly the almost-sure singularity of the $\pi$-kernels entering the extremal decomposition of the $A$-localized states (Theorem \ref{thm: A.s. singularity extremals}).  

\subsubsection{Reconstruction on thinned branches}

Let us start by giving the definition of a thinned branch and its related random variable $\underline{\phi}^\omega$.

\begin{definition}
Let $r=(r_i)_{i\in \N}$ be an increasing sequence with $r_i\in \N$. Then, define a \uline{thinned branch} as a sequence $(\Lambda^r_n)_{n\in \N}$ of subsets $\Lambda^r_n:=\{v_1,v_2,...,v_{n^2}\} \subset V$ where the vertices $v_i$ are chosen along a branch of the tree such that $d(v_{i},v_{i+1})=:r_i$ for all $i\in \{1,...,n^2\}$ and for all $n\in \N$. For a given thinned branch $(\Lambda^r_n)_{n\in \N}$ and a fixed configuration $\omega\in \Omega$, define the random variable
\begin{equation}\label{eq: thinned branch overlap}
    \underline{\phi}^\omega:=\liminf_{n\uparrow \infty} \frac{1}{|\Lambda_n^r|}\sum_{v\in \Lambda^r_n}\mathds{1}_{\sigma_v=\omega_v}
\end{equation}
which is called \uline{thinned branch-overlap}.
\end{definition}

$\underline{\phi}^\omega$ is a tail-measurable observable with values in the interval $[0,1]$ which quantifies how much the configuration $\sigma$ matches with $\omega$ on an increasing sparse volume $\Lambda_n^r$. The following theorem describes that the $\pi$-kernel $\pi(\cdot|\omega)$ is able to reconstruct the signal $\omega$ on a thinned branch $\Lambda_n^r$ which was send under the measure $\mu_A$ to the outside. 

\begin{theorem}[Multi-site reconstruction] \label{thm: Branch overlap}
    Let $A\subsetneq \Z_q$ with $|A|\geq 2$. Consider the $A$-localized state $\mu_A$, see Proposition \ref{prop: Bounds A-loc. states}, for the class of models with Hamiltonian of the form \eqref{def: Hamiltonian} and a potential defined in \eqref{eq: Potential ferr n.n. model} fulfilling the $u,U,d$-bounds in \eqref{eq: u,U-bounds}. Then, the thinned branch-overlap $\underline{\phi}^\omega$ is $\pi(\cdot|\omega)$-a.s. lower bounded by 
    \begin{equation*}
        \underline{\phi}^\omega=\liminf_{n\uparrow \infty}\frac{1}{|\Lambda_n^r|}\sum_{v\in \Lambda_n^r}\mathds{1}_{\sigma_v=\omega_v}\geq 1- \epsilon_1(\beta)-\epsilon_2(\beta)
    \end{equation*}
    for sparse enough sets $\Lambda_n^r$ and $\mu_A$-a.e. $\omega\in \Omega$. The quantities $\epsilon_1$ and $\epsilon_2$ are given in Lemma \ref{lem: Key Lemma} and they satisfy $\epsilon_1(\beta)\downarrow 0$ as well as $\epsilon_2(\beta)\downarrow 0$ for $\beta\uparrow \infty$.
\end{theorem}

We will give a shortened proof by explaining the important steps that have to be done and where one has to pay attention to the properties of the $A$-localized states. For more details about the proof, see \cite{CoKuLe24}. First of all, one is able to relate the variable $\underline{\phi}^\omega$ to the $\pi$-kernels in the following sense. 

\begin{lemma}\label{lem: Concentration of spin overlaps under tp extr.}
     There exists a sequence of spacings $r=(r_i)_{i\in \N}$, in general depending on the model parameter, such that for any thinned branch $(\Lambda_n^r)_{n\in \N}$ with $\sum^\infty_{n=1}\frac{1}{|\Lambda_n^r|}<\infty$, the following holds. For $\mu_A$-almost every $\omega\in \Omega$, for $\pi(\cdot|\omega)$-almost every realization of $\sigma$, we have 
     \begin{equation}\label{eq: Concentration spin-overlaps under typ extr.}
         \limsup_{n\uparrow \infty}\frac{1}{|\Lambda_n^r|}\bigg|\sum_{v\in \Lambda^r_n}\bigg(\mathds{1}_{\sigma_v\neq \omega_v}-\pi(\sigma_v \neq \omega_v|\omega)\bigg)\bigg|=0.
     \end{equation}
\end{lemma}

Its proof uses the Borel-Cantelli lemma and Chebychev's inequality. In this way, it remains to show the summability of covariances of the form $\pi(\sigma_v\neq \omega_v;\sigma_u\neq \omega_u|\omega)$ where $u,v\in \Lambda_n^r$. Due to the fact that $\omega\in \Omega$ is $\mu_A$-typical, $\pi(\cdot|\omega)$ is an extremal Gibbs measure, see \hyperref[Sec: Appendix A]{Appendix A}. Consequently, one can choose a thinned enough branch and uses the decorrelation of far enough separated events under extremal Gibbs measures, see for example \cite{Ge11}, to prove the desired summability.

With the results of Lemma \ref{lem: Concentration of spin overlaps under tp extr.} and Lemma \ref{lem: Key Lemma} a), we have so far
\begin{equation}\label{eq: Bounds for thinned branch overlap 1}
    \underline{\phi}^\omega=1-\limsup_{n\uparrow \infty} \frac{1}{|\Lambda^r_n|}\sum_{v\in \Lambda^r_n}\pi(\sigma_v\neq \omega_v|\omega)\geq 1-\liminf_{n\uparrow \infty} \frac{1}{|\Lambda^r_n|}\sum_{v\in \Lambda^r_n}\mathds{1}_{B_{A,v}}(\omega) -\epsilon_1(\beta)\\
\end{equation}
$\pi(\cdot|\omega)$-almost surely for $\mu_A$-almost every $\omega$. Furthermore, one can say more about the above sum in the lower bound.

\begin{lemma}\label{lem: A.s. convergence emp. mean of bad vertices}
    For any thinned branch $(\Lambda^r_n)_{n\in \N}$ such that $\sum_{n=1}^\infty\frac{1}{|\Lambda^r_n|}<\infty$ and for $\mu_A$-almost every $\omega$, the following limit holds
    \begin{equation*}
        \lim_{n\uparrow \infty} \frac{1}{|\Lambda^r_n|} \sum_{v\in \Lambda_n^r}\mathds{1}_{B_{A,v}}(\omega)=\mu_A(B_{A,\rho})
    \end{equation*}
    where $B_{A,v}$ was defined in \eqref{eq: Bad events} for $v\in V$.
\end{lemma}

Again, the proof idea is to use the Borel-Cantelli lemma in combination with Chebychev's inequality such that it remains to show the summability of certain covariances. In this case, one has to consider covariances of bad events, see \eqref{eq: Bad events}, around vertices on the thinned branch. They are defined as follows
 \begin{equation}\label{def: Site covariances}
        \text{Cov}(u,v):=\mu_A\Big(\big(\mathds{1}_{B_{A,u}}(\omega)-\mu_A(B_{A,\rho})\big)\cdot\big(\mathds{1}_{B_{A,v}}(\omega)-\mu_A(B_{A,\rho})\big)\Big)
    \end{equation}
for vertices $u,v\in V$. One is now able to show that these covariances decay exponentially fast in the minimal distance between the two vertices.

\begin{lemma}\label{lem: Decorrelation of bad events}
  For small enough $\beta$, there exists $C>0$ such that for any $c_1\in (0,\frac{1}{3})$ and $c_2\in (0,1)$, for any $u,v\in V$,
    \begin{equation*}
        |\text{Cov}(u,v)|\leq Ce^{-c d(u,v)}
    \end{equation*}
    where we recall that $d(u,v)$ is the graph distance on the tree and 
    \begin{equation*}
        e^{-c}:=\max\{e^{-c_1\lambda(\beta)},|\lambda_2(P_A)|^{\frac{c_2}{6}}\}
    \end{equation*}
    where $\lambda(\beta)$ is defined in Lemma \ref{lem: Key Lemma} and $\lambda_2(P_A)$ is the second largest eigenvalue (in modulus) of the transition matrix $P_A$ of $\mu_A$.
\end{lemma}

The proof of this lemma follows an analogous idea given in \cite{CoKuLe24}, but we recall the main ideas for the sake of completeness. Consider two distinct subtrees $H_u(r)$ and $H_v(r)$ containing the vertices $u$ and $v$ where $r$ determines the distance between these subtrees. These subtrees can be chosen in such a way that the minimal path between these subtrees but also each minimal path from $u$ (respectively $v$) to the outside of $H_u(r)$ (respectively $H_v(r)$) contain a large enough amount of vertices. Then, one can investigate correlations of two bad events produced by contours where both laying completely inside these subtrees or contours where at least one has non-empty intersection with vertices outside and possibly connect these two subtrees. In the first case, one can use a consequence of the Perron-Frobenius theorem, see Example 4.3.9 in \cite{Bre20}, to argue that the transition matrix $P_A$ converges exponentially fast in the minimal distance between $H_u(r)$ and $H_v(r)$ to the stationary distribution. Note that we require the irreducibility and aperiodicty of $P_A$ which is however given because the state space is finite and $P_A$ is strictly positive, see Proposition \ref{prop: Bounds A-loc. states}. In the second case exists at least one contour with a large amount of vertices due to the fact that one of the contours connects to the outside of the subtree. Here, one can use the exponential smallness of this bad event under $\mu_A$ shown in Lemma \ref{lem: Key Lemma}. 

Choosing the thinning of the branch properly and using the result of Lemma \ref{lem: Decorrelation of bad events} leads then to the statement of Lemma \ref{lem: A.s. convergence emp. mean of bad vertices}. Finally, applying Lemma \ref{lem: A.s. convergence emp. mean of bad vertices} and part b) of Lemma \ref{lem: Key Lemma} in \eqref{eq: Bounds for thinned branch overlap 1} and we obtain 
\begin{equation*}
    \underline{\phi}^\omega \geq 1-\epsilon_1(\beta)-\mu_A(B_{A,\rho})\geq 1- \epsilon_1(\beta)-\epsilon_2(\beta)
\end{equation*}
which is the desired statement of the multi-site reconstruction in Theorem \ref{thm: Branch overlap}.

\subsubsection{Almost-sure singularity of $\mu_A$-typical extremals}

Let us discuss how Theorem \ref{thm: Branch overlap} can be used to show the following statements about the extremal decomposition of $A$-localized states $\mu_A$.

\begin{theorem}[Almost-sure singularity of extremals]\label{thm: A.s. singularity extremals}
Consider a $\Z_q$-valued nearest-neighbor clock model with Hamiltonian given in \eqref{def: Hamiltonian} and a potential defined in \eqref{eq: Potential ferr n.n. model} fulfilling the $u,U,d$-bounds in \eqref{eq: u,U-bounds}. Furthermore, let $\beta$ be large enough, $A\subsetneq \Z_q$ such that $|A|\geq 2$ and denote with $\mu_A$ its corresponding $A$-localized state, see Proposition \ref{prop: Bounds A-loc. states}. Then, for $\mu_A \otimes \mu_A$-a.e. pair $(\omega,\omega')$ the extremal Gibbs measures $\pi(\cdot|\omega)$ and $\pi(\cdot|\omega')$ are singular with respect to each other, i.e.
    \begin{equation*}
        \mu_A \otimes \mu_A\bigg(\big\{(\omega,\omega')\in \Omega \times \Omega : \pi(\cdot|\omega) \perp \pi(\cdot|\omega')\big\}\bigg)=1.
    \end{equation*}
\end{theorem}

As a direct consequence, we can deduce the following statement and hence $ii)$ of Theorem \ref{thm: Main results}.

\begin{corollary}
    The decomposition measure $\alpha_{\mu_A}$, introduced in Theorem \ref{thm: decomposition of Gibbs measures}, of the $A$-localized state $\mu_A$, given in Proposition \ref{prop: Bounds A-loc. states}, has no atoms, i.e. $\alpha_{\mu_A}(\{\nu\})=0$ for all $\nu\in \text{ex } \mathscr{G}(\gamma)$. In particular, there are uncountably many extremal states which enter in the extremal decomposition of $\mu_A$.
\end{corollary}

The proof of this corollary follows as in \cite{CoKuLe24} and the main idea is the fact that distinct extremals are singular, see Theorem 7.7 in \cite{Ge11}. In order to prove Theorem \ref{thm: A.s. singularity extremals}, one can use the multi-site reconstruction (Theorem \ref{thm: Branch overlap}) and deduce the following statement.

\begin{corollary}\label{cor: distinct extremals}
    Let $\beta$ be large enough such that 
    \begin{equation*}
        \epsilon_1(\beta)+\epsilon_2(\beta)<\frac{1}{2}\left(1-\sum_{a\in \Z_q}\mu_A(\sigma_\rho=a)^2\right).
    \end{equation*}
    Then, there is a thinned branch $(\Lambda_n^r)_{n\in \N}$ with increasing sequence $r=(r_i)_{i\in \N}$ such that for the correponding tail-measurable observable $\underline{\phi}^\omega$ we have the strict inequality
    \begin{equation*}
        \pi(\underline{\phi}^\omega|\omega)>\pi(\underline{\phi}^\omega|\omega')
    \end{equation*}
    for $\mu_A \otimes \mu_A$-a.e. $(\omega,\omega')$.
\end{corollary}

Note that this implies $\pi(\cdot|\omega)\neq \pi(\cdot|\omega')$ for $\mu_A \otimes \mu_A$-a.e. $(\omega,\omega')$. However, the measures satisfy $\pi(\cdot|\omega),\pi(\cdot|\omega')\in \text{ex}\mathscr{G}(\gamma)$ and thus they are singular. This implies the statement of Theorem \ref{thm: A.s. singularity extremals}. 

Let us give a sketch of the proof for Corollary \ref{cor: distinct extremals}. Note that 
\begin{equation}\label{eq: lower bound mismatches pi kernel}
        \pi(\underline{\phi}^\omega|\omega)\geq 1-\epsilon_1(\beta)-\epsilon_2(\beta)
    \end{equation}
follows directly from Theorem \ref{thm: Branch overlap}. On the other hand, one can show easily that $\mathds{1}_{\sigma_v=\omega_v}\leq \mathds{1}_{\omega_v=\omega_v'}+\mathds{1}_{\sigma_v\neq \omega_v'}$ and hence 
\begin{equation}\label{eq: upper bound pi branch overlaps}
        \pi(\underline{\phi}^\omega|\omega')\leq \limsup_{n\uparrow \infty}\frac{1}{|\Lambda^r_n|}\sum_{v\in \Lambda^r_n}\mathds{1}_{\omega_v=\omega'_{v}}+\pi\left(\limsup_{n\uparrow \infty}\frac{1}{|\Lambda^r_n|}\sum_{v\in \Lambda^r_n}\mathds{1}_{\sigma_v\neq \omega'_v}\Bigg|~ \omega'\right).
        \end{equation}
        The right term is again bounded by Theorem \ref{thm: Branch overlap} as follows
        \begin{equation*}
            \pi(1-\underline{\phi}^{\omega'}|\omega')\leq \epsilon_1(\beta)+\epsilon_2(\beta).
        \end{equation*}
        For the left term in \eqref{eq: upper bound pi branch overlaps}, we are able to prove that 
        \begin{equation}\label{eq: Convergence matches}
    \limsup_{n\uparrow \infty}\frac{1}{|\Lambda^r_n|}\sum_{v\in \Lambda^r_n}\mathds{1}_{\omega_v=\omega'_{v}}=\sum_{a\in \Z_q} \mu_A(\sigma_0=a)^2
\end{equation}
for $\mu_A \otimes \mu_A$-almost every pair $(\omega,\omega')$. The proof of this equality is again based on the application of the Borel-Cantelli lemma in combination with Chebyshev's inequality. It remains to prove the summability of covariances between the random variables $\mathds{1}_{\omega_u=\omega'_u}$ and $\mathds{1}_{\omega_v=\omega'_v}$ under the measure $\mu_A \otimes \mu_A$ for vertices $u,v\in \Lambda_n^r$. However, for different vertices and sparse enough branches, the covariance is exponentially small due to the convergence of $P_A$ to the stationary distribution. Consequently, we obtain the desired results of Corollary \ref{cor: distinct extremals}.

\section*{Open problems}\addcontentsline{toc}{section}{Open problems}

One remaining challenging problem is the investigation of the extremal decomposition of $A$-localized states in models with countably infinite state spaces, e.g. the $p$-SOS model. For these models, it is not possible to directly apply our coarse-graining method. The problem is that one 
can not be sure that $A$ is reached from any $a\in A^c$ in only one jump, with sufficiently high probability \text{uniformly} in $a$. The relevant governing probabilty for this is
the Markovian transition matrix $P_A$ of $\mu_A$. For all models except Potts it is only implicitly given, 
and the bounds might not be good enough, but moreover the desired statements might simply not hold. 
It could be possible, that one can modify the method such that it is sufficient to have a high probability of jumping back into $A$ only 
in a few jumps and thereby handle this problem.  But this will be a work for the future.

Another interesting and challenging question is to analyze the intermediate temperature regimes. The investigation of the Potts model on the Cayley tree in \cite{KuRo17} already showed that there are states with finite localization set possessing two different phase transitions. Similarly to the case of the free state in the Ising model on the regular tree, these states are provably not extreme in a low temperature regime (however with no information about the extremal decomposition measure given in \cite{KuRo17}), and provably extreme in an intermediate regime which 
is below the high temperature region where the model possesses a unique Gibbs state. 
If this intermediate regime always exists for general clock models is, to the best of our knowledge, still unknown.

\section*{Appendix A: The structure of $\mathscr{G}(\gamma)$ and background on the extremal decomposition}\label{Sec: Appendix A}
\addcontentsline{toc}{section}{Appendix A}

In statistical mechanics, we are interested in the structure of the set $\mathscr{G}(\gamma)$ for a given specification $\gamma$. First of all, note that $\Omega=(\Z_q)^{\otimes V}$ is compact with respect to the product topology and hence $\mathscr{M}_1(\Omega)$ is weakly compact. This implies that $\mathscr{G}(\gamma^\Phi)\neq \emptyset$ for the finite-range interaction potentials $\Phi$ which we consider in this work. \newline

\noindent\textbf{Extremal Gibbs measures.} Moreover, one can conclude from \eqref{eq: DLR-equation} that the DLR-equations are linear and the set $\mathscr{G}(\gamma)$ is convex. Hence, we are interested in the extreme points of $\mathscr{G}(\gamma)$. These are points of particular interest because they are considered to be the physical states of the system, see \cite{FV17}, \cite{Ge11} and \cite{Le08}. Let us give an intuition behind this statement by considering the tail $\sigma$-algebra of asymptotic events
$ \mathscr{F}_\infty :=\bigcap_{\Lambda \Subset V} \mathscr{F}_{\Lambda^c}$.

In a real physical system in equilibrium, one can observe that quantities on a microscopic scale will fluctuate rapidly but quantities on a macroscopic scale are not affected by these fluctuations and remain constant. Here, the macroscopic quantities can be seen as tail measurable observables because they are not altered by finitely many spin flips. Why does an extremal Gibbs measure $\mu \in \text{ex } \mathscr{G}(\gamma)$ describes such a system in equilibrium? One can show that there is an equivalent characterization for the extremal Gibbs measures in the sense that they are trivial on $\mathscr{F}_\infty$, see Theorem 7.7 in \cite{Ge11}. Hence, the extreme elements of $\mathscr{G}(\gamma)$ describe a possible deterministic outcome for the macroscopic quantities of the system and hence a possible equilibrium state of the system. 

Let $\mu \in \mathscr{G}(\gamma)$, then we are able to obtain an extremal Gibbs measure by considering the $\pi$-kernel, see \eqref{eq: Limit pi-kernel}. In more detail, we have 
\begin{equation*}
    \mu\Big(\omega \in \Omega:~\pi(\cdot|\omega)\in \text{ex } \mathscr{G}(\gamma)\Big)=1.
\end{equation*}
For a detailed proof of this statement, we would recommend to consult the books \cite{FV17} and \cite{Ge11}.\newline

\noindent \textbf{Extremal decomposition of Gibbs measures.} One can say even more about the geometry of $\mathscr{G}(\gamma)$, namely that this set is a simplex and thus every $\mu\in \mathscr{G}(\gamma)$ can be uniquely decomposed into a convex combination of elements in $\text{ex}~\mathscr{G}(\gamma)$. A measure $\mu \in \mathscr{G}(\gamma)$ decomposes into the extremal Gibbs measures $\pi(\cdot |\omega)$ with $\mu$-typical boundary conditions $\omega\in \Omega$ at infinity. We want to discuss this extremal decomposition in the following in more detail.

Our goal is to consider probability measures on subsets $\mathscr{M} \subset\mathscr{M}_1(\Omega)$ and thus we need a measurable structure on these sets. This can be done by considering the \textit{evaluation $\sigma$-algebra} $e(\mathscr{M})$ which is defined as follows: Define \textit{evaluation maps} $e_A:\mathscr{M}\rightarrow [0,1]$ for each $A\in \mathscr{F}$ as $e_A(\mu):=\mu(A)$. Then, define 
\begin{equation*}
    e(\mathscr{M}):=\sigma\Big(\{e_A\leq c\},~A\in \mathscr{F},~0\leq c\leq 1\Big)
\end{equation*}
which is the smallest $\sigma$-algebra on $\mathscr{M}$ such that $e_A$ is measurable for each $A\in \mathscr{F}$. Finally, there is the following theorem which can be found as Theorem 7.26 in \cite{Ge11} or as Theorem 6.72 in \cite{FV17}.

\begin{theorem}\label{thm: decomposition of Gibbs measures}
    Assume that $\mathscr{G}(\gamma)\neq \emptyset$. Then $\mathscr{G}(\gamma)$ is a convex subset of $\mathscr{M}_1(\Omega)$ and it satisfies the following properties
    \begin{equation}\label{eq: Extremal decomposition GM}
        \mu=\int_{\text{ex }\mathscr{G}(\gamma)} \nu \cdot \alpha_\mu(d\nu)
    \end{equation}
    for all $\mu \in \mathscr{G}(\gamma)$. Here, $\alpha_\mu \in \mathscr{M}_1\Big(\text{ex }\mathscr{G}(\gamma),e(\text{ex }\mathscr{G}(\gamma))\Big)$ is defined for all $M\in e(\text{ex }\mathscr{G}(\gamma))$ by 
    \begin{equation}\label{eq: extremal decomposition measure}
        \alpha_\mu(M)=\mu\Big(\{\omega\in \Omega:~\pi(\cdot|\omega)\in M\}\Big).
    \end{equation}
\end{theorem}
It is interesting to note, that this result does not assume any regularity (like Markovianness or 
quasilocality) of the specification at hand. 

One could physically interpret the extremal decomposition in \eqref{eq: Extremal decomposition GM} as follows: A state $\mu \in \mathscr{G}(\gamma)$ represents a measurement of a system in equilibrium. This state decomposes into a convex combination of the real physical states $\text{ex }\mathscr{G}(\gamma)$ where the weights, given by $\alpha_\mu$, represent the uncertainty of this measurement.

\section*{Appendix B: Concentration and transition properties of $A$-localized clock states and proof of Proposition \ref{prop: Bounds A-loc. states}}\label{Sec: Appendix B} 
\addcontentsline{toc}{section}{Appendix B}

In Appendix B, we aim at a more detailed description of the quantities stated in Theorem \ref{thm: Existence A-localized states} which is needed to prove the bounds given in Proposition \ref{prop: Bounds A-loc. states}. The details of the next part can also be found in \cite{AbHeKuMa24}.

\subsection*{Details for $A$-localized states}

Define the number $\rho=\rho(d,n)$ as the unique positive solution of the equation
    \begin{equation*}
    (d-1)\rho^{d+1}+dn\rho^{d-1}-n=0.
\end{equation*}
The quantity $\eta=\eta(d,n)$ in Theorem \ref{thm: Existence A-localized states} depends on $\rho$ and is defined as
\begin{equation}\label{eq: def eta}
    \eta(d,n):=\frac{\rho-\rho^d}{(\rho^{d+1}+n)^{\frac{d}{d+1}}}.
\end{equation}
Furthermore, one is able to show that $\eta\in (\eta_0,1)$ and $\rho\in (0,d^{-\frac{1}{d-1}})$ where 
    \begin{equation}\label{eq: def eta_0}
        \eta_0:=\eta_0(d,n):=d^{-\frac{1}{d-1}}(1-d^{-1})(n+1)^{-\frac{d}{d+1}},
    \end{equation}
    see \cite{AbHeKuMa24} for a proof. The constants in Theorem \ref{thm: Existence A-localized states} depend on $n$ and $d$ and are more precisely given as
\begin{align*}
        c_1:=\frac{\rho^{d-1}}{\eta^{d-1}},~~~~c_2:=d(d^{\frac{1}{d-1}}&-1)(n+1)^{\frac{d}{d+1}},~~~~c_3:=n^{-1}d^{\frac{d+1}{d-1}}c_1^{\frac{d+1}{d-1}},\\
        c_4:=\frac{d+1}{d-1}&c_2~~~\text{and}~~~c_5:=n^{-1}c_1^{\frac{d+1}{d-1}}+c_4.
    \end{align*}
    The proof of Theorem \ref{thm: Existence A-localized states} relies on solving a recursive system of equations on the tree for the given transfer operator $Q$. In more detail, the goal is to find strictly positive solutions $u\in \ell^{\frac{d+1}{d}}(\Z_q)$ of the equation $u=c(Q * u)^d$ for some $c>0$ which are called \textit{boundary laws}. Note that $Q*u$ is the convolution of $Q$ and $u$ and it is defined as $(Q*u)(i)=\sum_{j\in \Z_q}Q(i-j)u(j)$ for all $i\in \Z_q$. Each of these solutions is then related to a Gibbs state for the corresponding model, see \cite{Ge11} or \cite{Za83} for more information about the boundary law formalism. The single-site marginal $\pi_A$ of $\mu_A$ can be expressed in terms of the corresponding strictly positive boundary law $u_A$ as follows $\pi_A(i)=u_A(i)^{\frac{d+1}{d}}(\|u_A^{\frac{d+1}{d}}\|_1)^{-1}$ for all $i\in \Z_q$ and hence $\pi_A$ is strictly positive, see Proposition 5.4 in \cite{AbHeKuMa24}. Moreover, the transition matrix $P_A$ can be written as
\begin{equation}\label{eq: Transition matrix}
    P_A(i,j)=\frac{\pi_A(j)^{\frac{d}{d+1}}Q(i-j)}{\big(Q * \pi_A^{\frac{d}{d+1}}\big)(i)}
\end{equation}
for all $i,j\in \Z_q$, whose proof is given in Remark 5.5 of \cite{AbHeKuMa24}.

\subsection*{Proof of Proposition \ref{prop: Bounds A-loc. states}}

Together with the details for the constants $c_1,...,c_5$, the quantities $\eta$ and $\rho$, and the relation \eqref{eq: Transition matrix}, given in the last section, we are able to prove Proposition \ref{prop: Bounds A-loc. states} which gives more explicit bounds in our situation of $u,U,d$-potentials.

\begin{proof}
    Note that \eqref{eq: Bound variable epsilon} follows directly after using the bound $e^{-\beta U}\leq Q(i)\leq e^{-\beta u}$ for each $i\in \Z_q\setminus \{0\}$, see \eqref{eq: u,U-bounds}. Let us continue by verifying the conditions for Theorem \ref{thm: Existence A-localized states}. First of all, we have $Q(0)=1$ since $\Bar{u}(0)=0$. Secondly, the quantity $\epsilon$ is upper bounded by $\eta$ because choosing $\beta$ large enough such that $(q-1)^{\frac{2}{d+1}}e^{-\beta  u}\leq \eta_0$ leads to the desired statement. Note that the left hand side is an upper bound for $\epsilon$, see \eqref{eq: Bound variable epsilon}, and $\eta_0$ is a lower bound for $\eta$, given in \eqref{eq: def eta_0}. Therefore, all conditions for Theorem \ref{thm: Existence A-localized states} are valid and we can use the results of this theorem to prove the remaining statements of Proposition \ref{prop: Bounds A-loc. states}. 
    
    First of all, note that the strict positivity of $P_A$ follows from the relation in \eqref{eq: Transition matrix}. The transfer operator $Q$ is strictly positive due to the definition of the potential \eqref{eq: Potential ferr n.n. model} together with the $u,U,d$-bounds in \eqref{eq: u,U-bounds}. Moreover, the single-site marginal $\pi_A$ is strictly positive due to Theorem \ref{thm: Existence A-localized states} and thus is $P_A$ strictly positive.
    
    Secondly, let us use $(i)-(iv)$ of Theorem \ref{thm: Existence A-localized states} to prove the bounds $(i)-(iv)$ stated in Proposition \ref{prop: Bounds A-loc. states}. Using the lower bound $\eta \geq \eta_0$ and the upper bound $\rho<d^{-\frac{1}{d-1}}$ in the definition of $c_1$ leads to
    \begin{equation}\label{eq: upper bound c1}
        c_1\leq\left(\frac{d}{d-1}\right)^{d-1}(n+1)^{\frac{d(d-1)}{(d+1)}}.
    \end{equation}
    Combining the upper bound for $c_1$, the upper bound in \eqref{eq: Bound variable epsilon} and inequality $(i)$ of Theorem \ref{thm: Existence A-localized states} and we obtain
    \begin{equation*}
        \|\Delta|_{A^c}\|_{\frac{d+1}{d-1}}\leq \left(\frac{d}{d-1}\right)^{d-1}(n+1)^{\frac{d(d-1)}{(d+1)}} (q-1)^{\frac{2(d-1)}{d+1}}e^{-(d-1)\beta u}.
    \end{equation*}
    Note that this gives us the desired result $(i)$ of Proposition \ref{prop: Bounds A-loc. states} after the application of $\frac{d-1}{d+1}\leq 1$ and $\frac{d}{d-1}\leq 2$.

Note that $\frac{d}{d+1}\leq 1$ directly implies $c_2 \leq d(d^{\frac{1}{d-1}}-1)(n+1)$. Consequently, we obtain 
    \begin{equation*}
        \min_{i\in A}\Delta(i)>1-d(d^{\frac{1}{d-1}}-1)(n+1)(q-1)^{\frac{2}{d+1}}e^{-\beta u},
    \end{equation*}
    with inequality $(ii)$ of Theorem \ref{thm: Existence A-localized states} where we additionally used \eqref{eq: Bound variable epsilon}. Combining this lower bound with the fact that $\frac{2}{d+1}\leq 1$ for all $d\geq 2$ and these considerations result in statement $(ii)$.

Exploiting the bound \eqref{eq: upper bound c1} in the definition of $c_3$ gives us
    \begin{equation*}
        c_3\leq n^{-1}d^{\frac{d+1}{d-1}}\left(\frac{d}{d-1}\right)^{d+1}(n+1)^d.
    \end{equation*}
    Using $\frac{d+1}{d-1}\leq 3$, $\frac{d}{d-1}\leq 2$ for $d\geq 2$ and we obtain further $c_3\leq 2^{d+1}d^{3}(n+1)^d$. Consequently, $(iii)$ of Theorem \ref{thm: Existence A-localized states} together with the upper bound in \eqref{eq: Bound variable epsilon} leads to 
    \begin{equation*}
        \sum_{i\in A^c} \pi_A(i) \leq 2^{d+1}d^{3}(n+1)^d(q-1)^{2}  e^{-(d+1)\beta u}
    \end{equation*}
    and hence it implies $(iii)$ of Proposition \ref{prop: Bounds A-loc. states}.

    Let us discuss the remaining bounds in $(iv)$ of Proposition \ref{prop: Bounds A-loc. states}. One can upper bound $c_4$ as $c_4\leq 3d(d^{\frac{1}{d-1}}-1)(n+1)$ where we used again the inequality $\frac{d+1}{d-1}\leq 3$ for $d\geq 2$ and the upper bound of $c_2$. This implies
    \begin{equation*}
        \pi_A|_A\leq \left(1-3d(d^{\frac{1}{d-1}}-1)(n+1)(q-1)^{\frac{2}{d+1}}e^{-\beta u}\right)^{-1}\frac{1}{|A|}
    \end{equation*}
    with $(iv)$ of Theorem \ref{thm: Existence A-localized states} and the upper bound of $\epsilon$ in \eqref{eq: Bound variable epsilon}. Hence, we obtain the second inequality of $(iv)$ in Proposition \ref{prop: Bounds A-loc. states}. For the lower bound of $(iv)$ in Propositon \ref{prop: Bounds A-loc. states}, we need to upper bound $c_5$ as follows
    \begin{equation*}
        c_5\leq \left(\frac{d}{d-1}\right)^{d+1}(n+1)^{d}+3d(d^{\frac{1}{d-1}}-1)(n+1).
    \end{equation*}
 Here, we applied the upper bound of $c_1$ in \eqref{eq: upper bound c1} and the upper bound of $c_4$. Combining these considerations in $(iv)$ of Theorem \ref{thm: Existence A-localized states} together with \eqref{eq: Bound variable epsilon} and it results in
 \begin{equation*}
        \pi_A|_A\geq \left(1-\Bigg(\left(\frac{d}{d-1}\right)^{d+1}(n+1)^{d}+3d(d^{\frac{1}{d-1}}-1)(n+1)\Bigg)(q-1)^{\frac{2}{d+1}}e^{-\beta u}\right)\frac{1}{|A|}.
    \end{equation*}
    Comparing this inequality with statement $(iii)$ of Proposition \ref{prop: Bounds A-loc. states} leads to the desired result.
\end{proof}

\section*{Appendix C: Proof of Lemma \ref{lem: Bounds for M}}\label{Sec: Appendix C}
\addcontentsline{toc}{section}{Appendix C}

We will prove this lemma from Proposition \ref{prop: Bounds A-loc. states}.

\begin{proof}
 First of all, the bounds in \eqref{eq: Bounds for pi_A} follow directly from $(iii)$ of Proposition \ref{prop: Bounds A-loc. states} and the fact that $\pi_A$ is close to the equidistribution on $A$.

Let us continue to prove \eqref{eq: Bounds for M}. In order to do this, recall from \eqref{eq: Transition matrix} that the transition matrix $P_A$ of $\mu_A$ can be written in terms of $\pi_A$ and the transfer operator $Q$. Let us start to give lower bounds for the denominator in different situations for the value $i\in \Z_q$. If $i\in A$, the denominator is bounded from below by
\begin{align}\label{eq: bound denominator A}
    \big(Q * \pi_A^{\frac{d}{d+1}}\big)(i)\geq \pi_A^{\frac{d}{d+1}}(i)\geq (1-(C_1+C_2)e^{-\beta u})^{\frac{d}{d+1}}|A|^{-\frac{d}{d+1}}
\end{align}
and if $i\in A^c$, we have 
\begin{equation}\label{eq: bound denominator A^c}
    \big(Q * \pi_A^{\frac{d}{d+1}}\big)(i)\geq Q(i-\Tilde{i}) \pi_A^{\frac{d}{d+1}}(\Tilde{i})\geq (1-(C_1+C_2)e^{-\beta u})^{\frac{d}{d+1}} |A|^{-\frac{d}{d+1}}e^{-\beta U}
\end{equation}
where $\Tilde{i}\in A$ arbitrary. Here, we used $(iv)$ of Proposition \ref{prop: Bounds A-loc. states} together with the $u,U,d$-bounds, given in \eqref{eq: u,U-bounds}.

We begin with the upper bound of $M(1,1)$. First of all, note that $P_A(a,a)$ is very probable, see $(ii)$ of Proposition \ref{prop: Bounds A-loc. states},  and hence it is reasonable to bound it by one. Secondly, we can upper bound the numerator of $P_A(a,A\setminus \{a\})$ with Hölder's inequality as follows 
\begin{equation}\label{eq: 1. bound M(1,1)}
    \sum_{j\in A\setminus \{a\}} \pi_A(j)^{\frac{d}{d+1}}Q(a-j)\leq \|\pi_A|_{A\setminus \{a\}}\|^{\frac{d}{d+1}}_1\|Q-\mathds{1}_{\{0\}}\|_{d+1},
\end{equation}
where we used $\frac{d+1}{d}$ and $d+1$ as the exponents of Hölder's inequality. We can treat the first factor with $(iv)$ of Proposition \ref{prop: Bounds A-loc. states} as
\begin{equation}\label{eq: 2. bound M(1,1)}
    \left(\sum_{j\in A\setminus \{a\}}\pi_A(j)\right)^{\frac{d}{d+1}}\leq (1-C_2e^{-\beta u})^{-\frac{d}{d+1}} (1-|A|^{-1})^{\frac{d}{d+1}}
\end{equation}
and the second one with the definition of $Q$ and \eqref{eq: u,U-bounds} as 
\begin{equation}\label{eq: 3. bound M(1,1)}
    \left(\sum_{j=1}^{q-1}Q(j)^{d+1}\right)^{\frac{1}{d+1}}\leq (q-1)^{\frac{1}{d+1}}e^{-\beta u}.
\end{equation}
Combining \eqref{eq: bound denominator A}, \eqref{eq: 1. bound M(1,1)}-\eqref{eq: 3. bound M(1,1)} and the fact that $C_1+C_2\geq C_2$ leads to 
\begin{align}
     M(1,1)\leq (1-(C_1+C_2)e^{-\beta u})^{-\frac{2d}{d+1}}(|A|-1)^{\frac{d}{d+1}} (q-1)^{\frac{1}{d+1}}e^{-\beta u+t}+1.
\end{align} 
Choosing $(C_1+C_2)e^{-\beta u}<\frac{1}{2}$ and recall that $C_1>1$, see definition in Proposition \ref{prop: Bounds A-loc. states}, and thus we obtain the desired bound for $M(1,1)$ in \eqref{eq: Bounds for M}.

Similarly, let us take a look at the entry $M(1,0)$. The denominator of $P_A(a,A^c)$ is bounded by \eqref{eq: bound denominator A} and the numerator can be bounded again with Hölder's inequality 
\begin{equation}\label{eq: 1. bound M(1,0)}
    \sum_{j\in A^c} \pi_A(j)^{\frac{d}{d+1}}Q(a-j)\leq \|\pi_A|_{A^c}\|^{\frac{d}{d+1}}_1\|Q-\mathds{1}_{\{0\}}\|_{d+1}.
\end{equation}
For the first term, we can use $(iii)$ of Proposition \ref{prop: Bounds A-loc. states} and obtain the bound 
\begin{equation}\label{eq: 2. bound M(1,0)}
    \left(\sum_{j\in A^c} \pi_A(j)\right)^{\frac{d}{d+1}}\leq C_1^{\frac{d}{d+1}} e^{-d\beta u}.
\end{equation}
Using \eqref{eq: bound denominator A}, \eqref{eq: 3. bound M(1,1)}, \eqref{eq: 1. bound M(1,0)} and \eqref{eq: 2. bound M(1,0)} results in
\begin{align*}
     M(1,0)\leq C_1^{\frac{d}{d+1}}(1-(C_1+C_2)e^{-\beta u})^{-\frac{d}{d+1}}|A|^{\frac{d}{d+1}}(q-1)^{\frac{1}{d+1}}e^{-(d+1)\beta u+t}.
\end{align*}
Using again the fact that $C_1>1$ and $(C_1+C_2)e^{-\beta u}<\frac{1}{2}$ and we are able to achieve the prefactor $C_3$ as stated in \eqref{eq: Bounds for M}.

Let us continue by considering the case $M(0,1)$. In order to avoid the factor $e^{\beta U}$, which is bad for our consideration, we will bound $\sup_{a\in A^c}P(a,A)$ with one and hence $M(0,1)$ with $e^t$.

It remains to bound $M(0,0)$. We split $P_A(a,A^c)=P_A(a,A^c\setminus \{a\})+P_A(a,a)$ for all $a\in A^c$ and consider each term independently. Let us start with $P_A(a,A^c\setminus \{a\})$ whose denominator is bounded by \eqref{eq: bound denominator A^c}. The numerator can be bounded with Hölder's inequality as follows 
\begin{equation}\label{eq: 1. bound M(0,0)}
    \sum_{j\in A^c\setminus \{a\}}\pi_A(j)^{\frac{d}{d+1}}Q(a-j)\leq \|\pi_A|_{A^c} \|_1^{\frac{d}{d+1}} \|Q-\mathds{1}_{\{0\}}\|_{d+1}.
\end{equation}
Note that 
\begin{equation}\label{eq: 2. bound M(0,0)}
     \|\pi_A|_{A^c} \|_1^{\frac{d}{d+1}} \|Q-\mathds{1}_{\{0\}}\|_{d+1}\leq C_1^{\frac{d}{d+1}} (q-1)^{\frac{1}{d+1}}e^{-(d+1)\beta u},
\end{equation}
compare \eqref{eq: 3. bound M(1,1)} and \eqref{eq: 2. bound M(1,0)}. Consequently, we have 
\begin{equation}\label{eq: 3. bound M(0,0)}
    P_A(a,A^c\setminus \{a\})\leq C_1^{\frac{d}{d+1}} (1-(C_1+C_2)e^{-\beta u})^{-\frac{d}{d+1}}|A|^{\frac{d}{d+1}}(q-1)^{\frac{1}{d+1}}e^{-\beta((d+1)u-U)}.
\end{equation}
The other term can be upper bounded with $(i)$ of Proposition \ref{prop: Bounds A-loc. states} as 
\begin{equation}\label{eq: 4. bound M(0,0)}
    P_A(a,a)\leq \left(\sum_{i\in A^c}P_A(i,i)^{\frac{d+1}{d-1}}\right)^{\frac{d-1}{d+1}}\leq C_1e^{-(d-1)\beta u}.
\end{equation}
Combining \eqref{eq: 3. bound M(0,0)} and \eqref{eq: 4. bound M(0,0)} leads to 
\begin{equation*}
    M(0,0)\leq C_1^{\frac{d}{d+1}} (1-(C_1+C_2)e^{-\beta u})^{-\frac{d}{d+1}}|A|^{\frac{d}{d+1}}(q-1)^{\frac{1}{d+1}}e^{-\beta((d+1)u-U)+t}+C_1e^{-(d-1)\beta u+t}.
\end{equation*}
With similar considerations as before, this expression is bounded as in \eqref{eq: Bounds for M} and this ends the proof of Lemma \ref{lem: Bounds for M}.
\end{proof}

\hypersetup{urlcolor=blue}
\begin{footnotesize}

\end{footnotesize}

\end{document}